\documentclass{amsart}

\usepackage{amssymb}
\usepackage{mathtools}
\usepackage{hyperref}
\usepackage{tikz}
\usepackage{tikz-cd}
\usepackage{caption}
\usepackage{subcaption}

\DeclareMathOperator{\N}{\mathbb{N}}
\DeclareMathOperator{\Z}{\mathbb{Z}}
\DeclareMathOperator{\Q}{\mathbb{Q}}
\DeclareMathOperator{\R}{\mathbb{R}}
\DeclareMathOperator{\C}{\mathbb{C}}
\DeclareMathOperator{\kk}{k}

\DeclareMathOperator{\A}{\mathbb{A}}
\DeclareMathOperator{\Aut}{Aut}
\DeclareMathOperator{\Burn}{B}
\DeclareMathOperator{\Bir}{Bir}
\DeclareMathOperator{\Bs}{\mathcal{B}}
\DeclareMathOperator{\Cb}{\mathcal{C}_b}
\DeclareMathOperator{\Cbk}{\mathcal{C}_{b,k}}
\DeclareMathOperator{\CC}{\mathcal{C}}

\DeclareMathOperator{\conv}{Conv}
\DeclareMathOperator{\Div}{Div}
\DeclareMathOperator{\dist}{d}
\DeclareMathOperator{\Exc}{Exc}
\DeclareMathOperator{\F}{\mathbb{F}}
\DeclareMathOperator{\GL}{GL}
\DeclareMathOperator{\Gal}{Gal}
\DeclareMathOperator{\Hh}{\mathbb{H}}
\DeclareMathOperator{\id}{id}
\DeclareMathOperator{\Ind}{Ind}
\DeclareMathOperator{\Jonq}{Jonq}
\DeclareMathOperator{\Nb}{N}
\DeclareMathOperator{\NS}{NS}
\DeclareMathOperator{\PGL}{PGL}
\DeclareMathOperator{\Psaut}{Psaut}
\newcommand{\PP}{\mathbb{P}}
\DeclareMathOperator{\SL}{SL}
\DeclareMathOperator{\Lk}{Lk}
\DeclareMathOperator{\Min}{Min}
\DeclareMathOperator{\Fix}{Fix}
\DeclareMathOperator{\Spec}{Spec}
\DeclareMathOperator{\Hyp}{Hyp}
\DeclareMathOperator{\Hypt}{\tilde{Hyp}}

\newcommand{\dashmapsto}{\mapstochar\dashrightarrow}

\newtheorem{theorem}{Theorem}[section]
\newtheorem{lemma}[theorem]{Lemma}
\newtheorem{proposition}[theorem]{Proposition}
\newtheorem{corollary}[theorem]{Corollary}

\theoremstyle{definition}
\newtheorem{definition}[theorem]{Definition}
\newtheorem{example}[theorem]{Example}

\newtheorem{question}[theorem]{Question}
\newtheorem{fact}[theorem]{Fact}

\theoremstyle{remark}
\newtheorem{remark}[theorem]{Remark}

\numberwithin{equation}{section}

\begin{document}

\title{Survey on Cremona groups from a median geometric point of view}

\author{Anne Lonjou}
\address{Laboratoire de mathématiques d’Orsay, Université Paris-Saclay, 91405, Orsay, France}
\email{anne.lonjou@universite-paris-saclay.fr}
\thanks{Partially supported by MCIN /AEI /10.13039/501100011033 / FEDER through the spanish grant Proyecto PID2022-138719NA-I00, and by the french National Agency of Research (ANR) through the project GOFR ANR-22-CE40-0004.}


\date{\today}

\begin{abstract}
	This expository article presents recent constructions of actions of Cremona groups on median graphs. It is aimed at both geometric group theorists and algebraic geometers.
\end{abstract}

\maketitle

\tableofcontents

\section{Introduction}

Groups of birational transformations of algebraic varieties, i.e. transformations given locally by quotients of polynomials, have a special place in 
algebraic geometry as the classification of algebraic varieties is up to birational transformations instead of isomorphisms (given locally by polynomials). One of the richest and interesting ones is the Cremona group, i.e., the group of birational transformations of the projective space.

Even if these groups are coming from algebraic geometry, geometric group theoretic methods have been particularly powerful in order to study Cremona groups of rank $2$ (see Subsection \ref{Subsection_State_art}). Nevertheless, until recently, these methods could not be used to study Cremona groups of higher ranks as no (non-trivial) action on geometric spaces where known.

The aim of this survey is to introduce gently recent constructions (\cite{Lonjou_Urech_cube_complexe}, \cite{GLU_Neretin}, \cite{Lonjou_Przytycki_Urech}, \cite{Cantat_Cornulier_Commensurating}) of actions of Cremona groups on median graphs (or CAT(0) cube complexes) to both algebraic geometers and geometric group theorists, as well as explaining which kind of results can be obtained from these actions.

On the other hand, having examples of groups that are huge (in the sense for instance that they are not finitely generated) and acting on geometrical objects not locally compact or of infinite dimension gives a motivation to try to extend the scope of classical results of geometric group theory.

Let us illustrate this briefly on a specific question that will be of interest along this survey.
Some birational transformations are conjugate (by a birational map) to an automorphism of a variety; they are called \emph{regularizable}.
One of the question that has lead to these constructions of median graphs is the following. 

\begin{question}\label{question_loc_elliptic}
	Let $G$ be a finitely generated subgroup of $\Bir(X)$ such that for any $g\in G$, $g$ is regularizable. Is $G$ regularizable, i.e., do there exist a variety $Y$ and a birational map $\varphi : Y \dashrightarrow X$ such that $G$ is conjugated to a subgroup of automorphisms of $Y$:  $\varphi^{-1}G\varphi\subset\Aut(Y)$?
\end{question}

As we will see in Section \ref{Section_Median_graphs_surfaces}, the plane Cremona group acts on a median graph, called the blow-up graph. Moreover, the Cremona transformations inducing a elliptic isometries on this graph, namely the ones fixing a vertex, are exactly the Cremona transformations that are regularizable.  
Hence, from a geometric group theoretic point of view, the question can be rephrased as follows: if a finitely generated group acts on a median graph purely elliptically, does the entire group fix a same vertex. This is a natural question from the theory of median graphs and several works have been done in this direction (see Subsection \ref{Subsection_purely_elliptic}). The answer to the above question is no in general, but it has been answered positively by \cite{GLU_Neretin} with the extra condition that the median graph is without infinite cubes. This has been motivated by the study of the action of the Neretin group and on the Cremona group over finite fields acting on median graphs without infinite cubes.

This expository article has been built on a mini-course given for early career mathematicians working in geometric group theory at ``Riverside workshop on geometric group theory'' in March 2023 and organized by Matthew Durham and Thomas Koberda. It also includes proofs or idea of proofs when considered better for the understanding. Everything in this survey can be found either in the classical literature or in the references given, except the construction of the rational blow-up graph constructed in \ref{Subsection_rational_blow_up graph}.

\subsection*{Acknowledgements}
The author would like to thank warmly Matthew Durham and Thomas Koberda to have organized this nice workshop, as well as the referees and Serge Cantat for their careful reading and comments that helped to strongly improve this survey. Finally, the author would like to thank strongly her co-authors Anthony Genevois and Christian Urech whose many joint works are presented here and are the heart of this survey. She thanks them also for the several discussions and references given.

\subsection*{Organization of the survey}
In order to make it the simplest possible for geometric group theorists, the scope of the constructions has been narrowed compared to the original articles. For instance, whenever possible, the constructions will be done in the context of quasi-projective varieties over algebraically closed fields. Nevertheless, one should keep in mind that these restrictions are only done to remove a bit of technicality but that one of the strengths of these constructions is that they allow to prove results for groups of birational transformations over arbitrary fields.

 Section \ref{Section_Median} is an introduction to median geometry. A large part of this section is inspired by the book \cite{Genevois_book_median}. Subsection \ref{Subsection_purely_elliptic} is dedicated to purely elliptic actions on median graphs. 
 Section \ref{Section_Cremona_groups}, largely inspired by the book \cite{Lamy_book}, introduces Cremona groups and is principally addressed to non algebraic geometers. A small state of the art on geometric group theoretic aspects of Cremona groups can be found in Subsection~\ref{Subsection_State_art}. 
In Section \ref{Section_Median_graphs_surfaces}, we present several recent constructions of actions of groups of birational transformations of surfaces on median graphs (\cite{Lonjou_Urech_cube_complexe}, \cite{GLU_Neretin}, \cite{Lonjou_Przytycki_Urech}) with some results that can be obtained from those actions. In Subsections \ref{Subsection_rational_blow_up graph} and \ref{Subsection_Jonquieres} we answer positively Question \ref{question_loc_elliptic} respectively for Cremona groups of rank $2$ over finite fields and for the Jonquières group.
Finally, in Section \ref{Section_median_graphs_higher_ranks} we present the construction of actions of groups of birational transformations of varieties of any dimension on median graphs, and some results that can be obtained, as well as a comparison with the related construction of \cite{Cantat_Cornulier_Commensurating} and \cite{Cornulier_FW} in Subsubsection \ref{subsubsection_Cornulier_Cantat}.

\section{Median geometry}\label{Section_Median}
In this Section we introduce basic material of median geometry with an accent in Subsection \ref{Subsection_purely_elliptic} on purely elliptic actions.

\subsection{CAT(0) cube complexes vs median graphs}\label{Subsection_CC_vs_median} In this subsection, we introduce the notions of CAT(0) cube complexes and median graphs and see how they are related. We then justify our choice of using the median vocabulary, which emphasis the combinatorial aspect rather than the CAT(0) one. There exists a large literature on this topic. You can find this material for instance in \cite{Bridson_Haefliger}, \cite{Sageev_lecturenotes}, \cite{Hagen_lecture_notes_CAT0_Median}, \cite{Genevois_book_median}.

\medskip
\subsubsection{Cube complexes}\label{subsec:cc}
A \emph{face} of an Euclidean cube $C$ is the convex hull of a non-empty subset of the set of vertices of the cube $C$.
 A \emph{cube complex} $\mathcal{C}$ is a union of finite dimensional Euclidean unit cubes glued together by isometries between some faces. We denote by $\mathcal{C}^0$ its vertex set. The cube complex is said \emph{of finite dimension} if it admits a global bound on the dimension of its cubes (in this case, its dimension is the maximal dimension of its cubes), otherwise it is \emph{of infinite dimension}. It is \emph{locally compact} if any vertex belongs to only finitely many edges. If no vertex of $\mathcal{C}$ belongs to an increasing sequence of cubes $(C_i)_{1\leq i\leq n}$, meaning that the dimension of $C_i$ is $i$ and $C_i$ is a face of $C_{i+1}$, then it is said \emph{locally of finite dimension}. Otherwise, the cube complex contains an \emph{infinite cube} meaning a copy of the graph whose vertices are the finitely supported sequences in $\{0,1\}^I$, for some infinite set $I$, and whose edges connect two sequences whenever they differ on a single coordinate.

 From now on, we assume cube complexes to be connected. A \emph{path}, in the 1-skeleton of a cube complex $\mathcal{C}$, between two vertices $x,y\in\mathcal{C}^0$ is a sequence of vertices of $\mathcal{C}$, $p=\left(x_0=x, x_1,\dots, x_n=y\right)$, such that two consecutive vertices are adjacent, i.e., connected by an edge. The integer $n$ is called the \emph{length} of this path, and it is denoted by $L(p)$. A \emph{string} between two points $x,y\in \mathcal{C}$ is a sequence of points of $\mathcal{C}$, $s=\left(x_0=x, x_1,\dots,x_n=y\right)$, such that two consecutive points $x_i,x_{i+1}$ belong to a common Euclidean cube. The integer $n$ is called the \emph{size} of the string, and it is denoted by $S(s)$. Note that the Euclidean distance $\lvert x_{i+1}-x_i\rvert$ between two successive points is well defined. The \emph{length} $L$ of this string will be the sum of the Euclidean distances between all pairs of consecutive points of this string: 
 \[L(s):=\sum_{i=0}^{n-1}\lvert x_{i+1}-x_i\rvert.\]
(Connected) cube complexes can be endowed with several natural metrics. Let us introduce the most classical ones. Let $\mathcal{C}$ be a (connected) cube complex.\begin{itemize}
\item The one usually used is the \emph{graph-metric} with the convention that any edge of the cube complex is isometric to the Euclidean interval $[0,1]$. It is denoted by $\dist$. The graph-metric between two vertices $x,y \in\mathcal{C}^0$ is defined as the infinimum of the lengths of all paths between $x$ and $y$:
\[\dist(x,y):= \inf\{L(p)\mid p \text{ is a path between } x \text{ and }y \}.\]

\item The \emph{Euclidean metric} is the one that will justify the name of CAT(0) cube complex even if in practice it is not so easy to deal with this metric. 
The \emph{Euclidean distance} between two points $x$ and $y$ of $\mathcal{C}$  is defined as the infinimum of the lengths of all strings between $x$ and $y$:
\[\dist_{Euc}(x,y):= \inf\{L(s)\mid s \text{ is a string between } x \text{ and }y \}.\]

\item The $\dist_\infty$ metric is the minimal size of all the strings between $x$ and $y$:
\[\dist_{\infty}(x,y):= \min\{S(s) \mid s \text{ is a string between } x \text{ and }y \}.\]
 \end{itemize}
 For instance, in Figure \ref{fig:cc}, $\dist(u,x_4)=4$, $\dist_{Euc}(u,x_4)= \sqrt{3} +1$ and $\dist_{\infty}(u,x_4)=2$.
  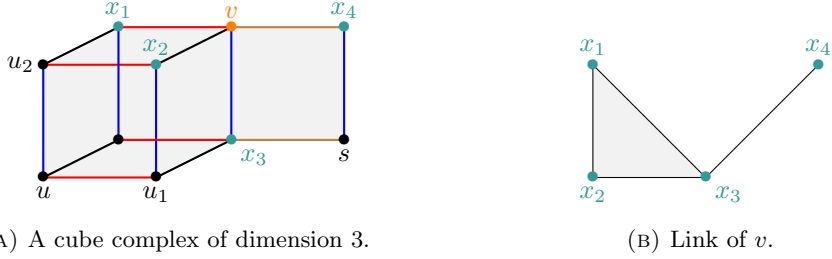
\begin{figure}
 	\begin{subfigure}{.45\textwidth}
 		\begin{center}
 			\begin{tikzpicture}
 				\fill[gray!10] (0,0) -- (1.5,0)-- (1.5,1.5) -- (0,1.5)-- cycle;
 				\fill[gray!10](1,0.5) -- (1,2)-- (2.5,2) -- (2.5,0.5)-- cycle;
 				\fill[gray!10] (0,0)  -- (0,1.5)-- (1,2) -- (1,0.5) -- cycle;
 				\fill[gray!10](1.5,0) -- (1.5,1.5)-- (2.5,2) -- (2.5,0.5)-- cycle;
 				\fill[gray!10] (2.5,2) -- (4,2) -- (4,0.5)-- (2.5,0.5)--cycle;
 				\draw[thick,red] (0,0) -- (1.5,0);
 				\draw[thick,red] (1.5,1.5) -- (0,1.5);
 				\draw[thick,blue] (1.5,0)-- (1.5,1.5);
 				\draw[thick,blue] (0,1.5)-- (0,0);
 				\draw[thick,blue] (1,0.5) -- (1,2);
 				\draw[thick,red] (1,2)-- (2.5,2); 
 				\draw[thick,blue] (2.5,2) -- (2.5,0.5);
 				\draw[thick,red] (2.5,0.5)-- (1,0.5);
 				\draw [thick](0,0) -- (1,0.5);
 				\draw [thick](1.5,0) -- (2.5,0.5);
 				\draw [thick](0,1.5) -- (1,2);
 				\draw [thick](1.5,1.5)  -- (2.5,2);
 				\draw[thick,brown] (2.5,2) -- (4,2);
 				\draw[thick,blue](4,2) -- (4,0.5);
 				\draw[thick,brown] (4,0.5)-- (2.5,0.5);
 				\draw(0,0) node {$\bullet$} node[below]{$u$};
 				\draw (1.5,0) node {$\bullet$} node[below]{$u_1$};
 				\draw (0,1.5) node {$\bullet$}node[left]{$u_2$};
 				\draw[color=teal!80] (1.5,1.5) node {$\bullet$}node[above]{$x_2$};
 				\draw (1,0.5) node {$\bullet$};
 				\draw[color=teal!80] (2.5,0.5) node {$\bullet$} node[below right]{$x_3$};
 				\draw[color=teal!80] (1,2) node {$\bullet$} node[above]{$x_1$};
 				\draw[color=orange] (2.5,2) node {$\bullet$} node[above]{$v$};
 				\draw[color=teal!80] (4,2) node {$\bullet$} node[above]{$x_4$};
 				\draw (4,0.5) node {$\bullet$} node[below]{$s$};
 			\end{tikzpicture}
 			\caption{A cube complex of dimension $3$.			\label{fig:cc}}
 		\end{center}
 	\end{subfigure}
 	\hfill
 	\begin{subfigure}{.45\textwidth}
 		\begin{center}
 			\begin{tikzpicture}
 				\fill[gray!10] (1,0.5) -- (1,2)-- (2.5,0.5)-- cycle;
 				\draw (1,0.5) -- (1,2)-- (2.5,0.5)-- cycle;
 				\draw (4,2)-- (2.5,0.5);
 				\draw[color=teal!80] (1,0.5) node {$\bullet$} node[below]{$x_2$};
 				\draw[color=teal!80] (2.5,0.5) node {$\bullet$} node[below right]{$x_3$};
 				\draw[color=teal!80]  (1,2) node {$\bullet$} node[above]{$x_1$};
 				\draw[color=teal!80]  (4,2) node {$\bullet$} node[above]{$x_4$};
 			\end{tikzpicture}
 			\caption{Link of $v$. \label{fig:ex_link_flag}}
 		\end{center}
 	\end{subfigure}	
 	\caption{Example of a flag link.	\label{fig:flag_link}}
 \end{figure}

 In a metric space $(X,\dist)$, a \emph{geodesic} between two points $x,y\in X$, is a continuous path $\gamma : [0,\dist(x,y)] \rightarrow X $ realizing the distance between $x$ and $y$: $\gamma(0)=x
 $, $\gamma(\dist(x,y))=y$ and for all $t_1,t_2\in[0,\dist(x,y)] $, 
 $\dist(\gamma(t_1),\gamma(t_2))=\lvert t_1 - t_2\rvert$. For instance, in Figure \ref{fig:cc}, the respective paths in the 1-skeleton of the cube complex $\left(u,u_1,x_3,v,x_4\right)$ and $\left(u,u_2,x_1,v,x_4\right)$ are geodesics between $u$ and $x_4$ for the graph-metric, while they are not geodesics for the Euclidean metric. On the other hand, for the Euclidean metric, the diagonal $[u,v]$ of the cube union $[v,x_4]$ is a geodesic.

\medskip
 \subsubsection{CAT(0) metric spaces}
A metric space $(X,\dist)$ is \emph{geodesic}
 if any pair of points can be joined by a geodesic. In such a space a \emph{geodesic triangle} $T$ will  be the data of a triple of points $x_1,x_2,x_3\in X$ together with a choice of three geodesics $[x_i,x_j]$ for $1\leq i<j\leq 3$.
 A comparison triangle  $\bar{T}$ in $\R^2$ is a triangle $\bar{x}_1\bar{x}_2\bar{x}_3$ of $\R^2$ whose sides have the same lengths as the ones of $T$: $\lvert \bar{x}_i-\bar{x}_j\rvert=\dist(x_i,x_j)$ for $1\leq i<j\leq 3$. Hence to any point $p$ of $T$, there exists a unique comparison point $\bar{p}$ on $\bar{T}$ defined as follows: 
 $\bar{p}\in [\bar{x}_i,\bar{x}_j]$ and $\lvert \bar{p}-\bar{x}_i\rvert=\dist(p,x_i)$ where $1\leq i<j\leq 3$ are such that $p\in [x_i,x_j]$. 
 
 \begin{definition}
 	A geodesic metric space $(X,\dist)$ is \emph{a CAT(0) space} if for any geodesic triangle $T$ and for any pairs of points $p_1,p_2\in T$, the following inequality holds:
 	\[\dist(p_1,p_2)\leq \lvert \bar{p}_1-\bar{p}_2\rvert,\] where $\bar{p}_1$ and $\bar{p}_2$ are the comparison points in $\bar{T}$ of $p_1$ and $p_2$.
 \end{definition}

 A \emph{CAT(0) cube complex} is a cube complex that is a CAT(0) metric space when endowed with the Euclidean metric. They first appeared in \cite{Gromov} as a nice family of examples of CAT(0) spaces. Since then, they have been largely studied on their own. Indeed, they are powerful tools to study groups even if their CAT(0) geometry is rarely used (see for instance \cite{Genevois_median_vs_cc}). In general, checking whether a geodesic metric space is CAT(0) or not can be very difficult, nevertheless there exists a nice combinatorial criterion in the case of cube complexes.

\medskip 
 \subsubsection{Combinatorial point of view on cube complexes}
 Consider $v$ a vertex of $\mathcal{C}$. The \emph{link} of $v$, denoted by $\Lk(v)$, is the simplicial complex whose vertices are the vertices of $\mathcal{C}$ that are adjacent to $v$ and a set of vertices $\{v_1,\dots, v_n\}$ in $\Lk(v)$ spans a {$(n-1)$-}simplex in $\Lk(v)$ if $\{v, v_1,\dots, v_n\}$ is contained in the set of vertices of a $n$-cube in $\mathcal{C}$ (see for instance Figures \ref{fig:flag_link} and \ref{fig:nonflag_link}). 
 The link of $v$ is called \emph{flag} if every finite set $\{v_1,\dots, v_n\}$ of vertices of $\Lk(v)$ that are pairwise adjacent, spans a $(n-1)$-simplex in $\Lk(v)$. For instance, the link of $v$ in Figure \ref{fig:ex_link_flag} is flag while the one in Figure \ref{fig:ex_link_non_flag} is not.

 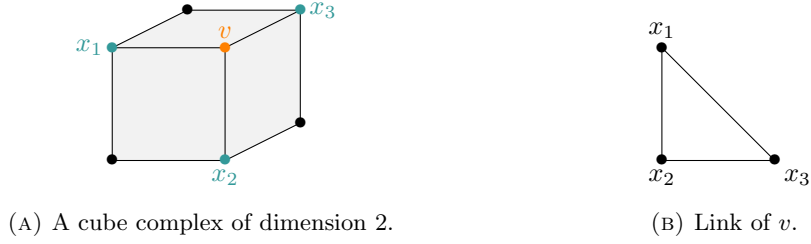
\begin{figure}
 	\begin{subfigure}{.45\textwidth}
 		\begin{center}
 			\begin{tikzpicture}
 				\fill[gray!10](0,0) -- (1.5,0)-- (1.5,1.5) -- (0,1.5)-- cycle;
 				\fill[gray!10] (1.5,0)-- (1.5,1.5) -- (2.5,2)-- (2.5,0.5) -- cycle;
 				\fill[gray!10] (1.5,1.5) -- (2.5,2)-- (1,2) -- (0,1.5) -- cycle;
 				\draw (0,0) -- (1.5,0)-- (1.5,1.5) -- (0,1.5)-- (0,0);
 				\draw  (1,2)-- (2.5,2) -- (2.5,0.5);
 				\draw (1.5,0) -- (2.5,0.5);
 				\draw (0,1.5) -- (1,2);
 				\draw (1.5,1.5)  -- (2.5,2);
 					\draw (0,0) node {$\bullet$};
 				\draw[color=teal!80] (1.5,0) node {$\bullet$}node[below]{$x_2$};
 				\draw[color=teal!80] (0,1.5) node {$\bullet$}node[left]{$x_1$};
 				\draw[color=orange]  (1.5,1.5) node {$\bullet$}node[above]{$v$};
 				\draw (2.5,0.5) node {$\bullet$};
 				\draw (1,2) node {$\bullet$};
 				\draw[color=teal!80] (2.5,2) node {$\bullet$} node[right]{$x_3$};
 			\end{tikzpicture}
 			\caption{A cube complex of dimension $2$.			\label{fig:cc2}}
 		\end{center}
 	\end{subfigure}
 	\hfill
 	\begin{subfigure}{.45\textwidth}
 		\begin{center}
 			\begin{tikzpicture}
 				\draw(1,0.5) node {$\bullet$} node[below]{$x_2$};
 				\draw(2.5,0.5) node {$\bullet$} node[below right]{$x_3$};
 				\draw (1,2) node {$\bullet$} node[above]{$x_1$};
 				\draw (1,0.5) -- (1,2)-- (2.5,0.5)-- cycle;
 			\end{tikzpicture}
 			\caption{Link of $v$.			\label{fig:ex_link_non_flag}}
 		\end{center}
 	\end{subfigure}	
 	\caption{Example of a non-flag link.	\label{fig:nonflag_link}}
 \end{figure}

By a theorem of Mikhail Gromov \cite[Section 4.2.C]{Gromov} in the finite dimensional case, and by Ian Leary \cite{Leary} in the infinite-dimensional case, we have the following useful criterion that will be taken in what follows as definition of CAT(0) cube complexes, as it is usually the one used in practice.

\begin{theorem}[\cite{Gromov}, \cite{Leary}]\label{thm_ccCAT((0))}
A cube complex is CAT(0) if and only if it is simply connected and  the link of any vertex is flag.
\end{theorem}

 For instance, the cube complex of Figure \ref{fig:cc} is CAT(0) while the one of Figure \ref{fig:cc2} is not.

 Note that with the above definition-theorem, the non-positive curvature of a cube complex can be understood only by its combinatorial structure. In what follows we will never use the Euclidean metric, but most of the time the graph metric.
 We will see in the next subsection that $1$-skeletons of CAT(0) cube complexes are median graphs and that indeed they contain already all the information needed for our purpose.

\medskip

\subsubsection{Median graphs}
In this survey, graphs are always assumed to be connected and \emph{simplicial}: no self-loops nor multiple edges are allowed.

\begin{definition} Let $X$ be a graph and $x_1,x_2,x_3 \in X$ three vertices. A \emph{median point of $x_1$, $x_2$, $x_3$}  is a vertex $m \in X$ satisfying
	\[d(x_i,x_j)= d(x_i,m)+d(m,x_j) \text{ for all } i \neq j.\]
	A graph is \emph{median} if any triple of vertices admits a unique median point.
\end{definition} 

Trees are the most basic examples of median graphs. Indeed, three points $x_1,x_2,x_3$ of a tree always define a tripod. The center of the tripod, i.e., the only vertex that is at the intersection of the three geodesics $c=[x_1,x_2]\cap[x_2,x_3]\cap[x_3,x_1]$ is the median point (see for instance Figure \ref{fig:median_tree}). Another important family of examples is the one of \emph{hypercubes} that are 1-skeletons of cubes (see for instance Figure \ref{fig:median_hypercube}).
If we remove the point $m$ of the hypercube of Figure \ref{fig:median_hypercube}, we obtain a graph which is not median anymore: the points $x_1,x_2,x_3$ do not have a median point (see Figure \ref{fig:non_median_half_hypercube}). In the complete bi-partite graph $K_{2,3}$ there exists three points admitting two median points (see Figure \ref{fig:non_median_bipartite}), so $K_{2,3}$ is not a median graph.

\begin{figure}
	\begin{subfigure}{.45\textwidth}
		\begin{center}
			\begin{tikzpicture}				
								\draw (-1,1) -- (0,0) -- (2,0) -- (3,1);
				\draw (-2,-0.5) -- (0,0)-- (1,1);
				\draw (-1,1) -- (0,0)-- (1,1);
				\draw[color=teal!80] (-1,1) node {$\bullet$}node[above]{$x_1$};
				\draw[color=teal!80] (-2,-0.5) node{$\bullet$} node[below]{$x_2$};
				\draw[color=teal!80] (2,0) node{$\bullet$} node[below]{$x_3$};
				\draw[color=orange] (0,0) node{$\bullet$} node[below]{$m$};
			\end{tikzpicture}
			\caption{Trees: $m$ is the median point of $x_1$, $x_2$ and $x_3$.\label{fig:median_tree}}
		\end{center}
	\end{subfigure}
	\hfill
	\begin{subfigure}{.45\textwidth}
		\begin{center}
			\begin{tikzpicture}
					\draw (0,0) -- (1.5,0)-- (1.5,1.5) -- (0,1.5)-- (0,0);
				\draw (1,0.5) -- (1,2)-- (2.5,2) -- (2.5,0.5)-- (1,0.5);
				\draw (0,0) -- (1,0.5);
				\draw (1.5,0) -- (2.5,0.5);
				\draw (0,1.5) -- (1,2);
				\draw (1.5,1.5)  -- (2.5,2);
				\draw[color=teal!80] (0,0) node {$\bullet$} node[left]{$x_1$};
				\draw (1.5,0) node {$\bullet$};
				\draw (0,1.5) node {$\bullet$};
				\draw (1.5,1.5) node {$\bullet$};
				\draw[color=orange] (1,0.5) node {$\bullet$}node[above left]{$m$};
				\draw[color=teal!80] (2.5,0.5) node {$\bullet$} node[right]{$x_2$};
				\draw[color=teal!80] (1,2) node {$\bullet$} node[above]{$x_3$};
				\draw (2.5,2) node {$\bullet$};
			\end{tikzpicture}
			\caption{A hypercube.			\label{fig:median_hypercube}}
		\end{center}
	\end{subfigure}	
	\caption{Examples of median graphs.	\label{fig:median_graph}}
\end{figure}
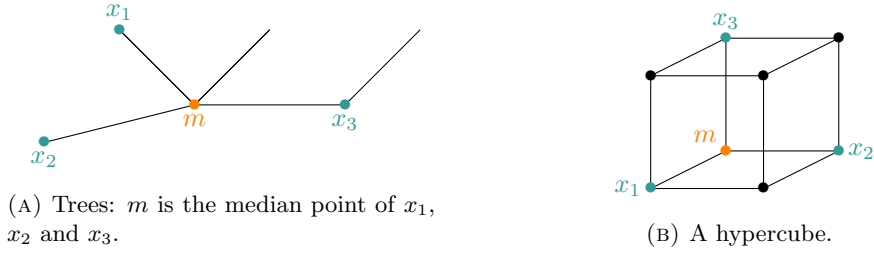

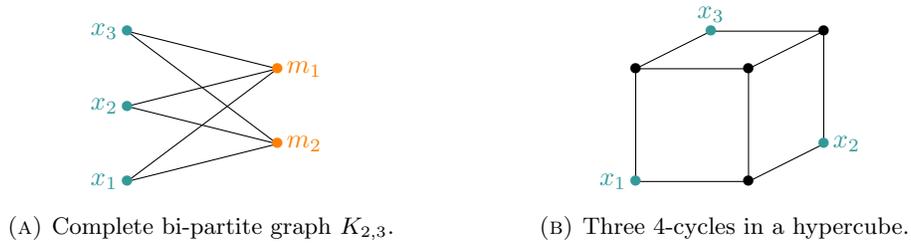
\begin{figure}
	\begin{subfigure}{.45\textwidth}
		\begin{center}
			\begin{tikzpicture}				
				\draw (0,0) -- (2,0.5) -- (0,1) -- (2,1.5) -- (0,2);
				\draw (0,0) -- (2,1.5);
				\draw (0,2) -- (2,0.5);
				\draw[color=teal!80] (0,0) node {$\bullet$} node[left]{$x_1$};
				\draw[color=teal!80] (0,1) node {$\bullet$} node[left]{$x_2$};
				\draw[color=teal!80] (0,2) node {$\bullet$} node[left]{$x_3$};
				\draw[color=orange] (2,1.5) node {$\bullet$} node[right]{$m_1$};
				\draw[color=orange] (2,0.5) node {$\bullet$} node[right]{$m_2$};
			\end{tikzpicture}
			\caption{Complete bi-partite graph $K_{2,3}$.	\label{fig:non_median_bipartite}}
		\end{center}
	\end{subfigure}
	\hfill
	\begin{subfigure}{.45\textwidth}
		\begin{center}
			\begin{tikzpicture}
				\draw (0,0) -- (1.5,0)-- (1.5,1.5) -- (0,1.5)-- (0,0);
				\draw  (1,2)-- (2.5,2) -- (2.5,0.5);
				\draw (1.5,0) -- (2.5,0.5);
				\draw (0,1.5) -- (1,2);
				\draw (1.5,1.5)  -- (2.5,2);
				\draw[color=teal!80] (0,0) node {$\bullet$} node[left]{$x_1$};
				\draw (1.5,0) node {$\bullet$};
				\draw (0,1.5) node {$\bullet$};
				\draw (1.5,1.5) node {$\bullet$};
				\draw[color=teal!80] (2.5,0.5) node {$\bullet$} node[right]{$x_2$};
				\draw[color=teal!80] (1,2) node {$\bullet$} node[above]{$x_3$};
				\draw (2.5,2) node {$\bullet$};
			\end{tikzpicture}
			\caption{Three $4$-cycles in a hypercube.	\label{fig:non_median_half_hypercube}}
		\end{center}
	\end{subfigure}	
	\caption{Example of non-median graphs.	\label{fig:non_median_graph}}
\end{figure}

\begin{theorem}[\cite{Chepoi_mediangraphs, Gerasimov_fixed_point_free_action, roller_thesis}]\label{thm:MedianVsCC}
	A graph is median if and only if it is the one-skeleton of 
	a CAT(0) cube complex. 
\end{theorem}

 Indeed, as explained in \cite{Genevois_median_vs_cc}, for the purpose of geometric group theory we more often work with the median geometry than with the CAT(0) geometry. Hence, starting from now on we will speak about ``median graphs'' instead of ``CAT(0) cube complexes'' to emphasis the phenomenon even if in the literature the term ``CAT(0) cube complex''. Nevertheless, the cubical structure is always present; so probably a better name would be median cube complexes... 
 
 In a median graph, we will call \emph{cube} the 1-skeleton of a cube.
We call \emph{cube completion} of a median graph the cube complex obtained by filling the cubes of the graph. The \emph{dimension} of a median graph is the one of its cube completion. Hence, by analogy we will say that a median graph is respectively of finite dimension, infinite dimension, locally of finite dimension, locally compact when the cube completion will be.
The \emph{cubical subdivision} of a median graph $\mathcal{G}$ is the graph whose vertices are the cubes of $\mathcal{G}$ and whose edges connect two cubes $C_1,C_2$ if one contains the other one and $\lvert \dim(C_1)-\dim(C_2)\rvert =1$ (see Figure \ref{Fig_ex} for an example). \begin{figure}
	\begin{subfigure}{.45\textwidth}
		\begin{center}
			\begin{tikzpicture}
				\fill[gray!10] (0,0) -- (1.5,0)-- (1.5,1.5) -- (0,1.5)-- cycle;
				\fill[gray!10](1,0.5) -- (1,2)-- (2.5,2) -- (2.5,0.5)-- cycle;
				\fill[gray!10] (0,0)  -- (0,1.5)-- (1,2) -- (1,0.5) -- cycle;
				\fill[gray!10](1.5,0) -- (1.5,1.5)-- (2.5,2) -- (2.5,0.5)-- cycle;
				\fill[gray!10] (2.5,2) -- (4,2) -- (4,0.5)-- (2.5,0.5)--cycle;
				\draw (0,0) -- (1.5,0);
				\draw (1.5,1.5) -- (0,1.5);
				\draw (1.5,0)-- (1.5,1.5);
				\draw (0,1.5)-- (0,0);
				\draw[dotted] (1,0.5) -- (1,2);
				\draw (1,2)-- (2.5,2); 
				\draw (2.5,2) -- (2.5,0.5);
				\draw[dotted] (2.5,0.5)-- (1,0.5);
				\draw [dotted](0,0) -- (1,0.5);
				\draw (1.5,0) -- (2.5,0.5);
				\draw (0,1.5) -- (1,2);
				\draw (1.5,1.5)  -- (2.5,2);
				\draw (2.5,2) -- (4,2);
				\draw(4,2) -- (4,0.5);
				\draw (4,0.5)-- (2.5,0.5);
				\draw(0,0) node {$\bullet$};
				\draw(1.5,0) node {$\bullet$};
				\draw(0,1.5) node {$\bullet$};
				\draw(1.5,1.5) node {$\bullet$};
				\draw(1,0.5) node {$\bullet$};
				\draw(2.5,0.5) node {$\bullet$};
				\draw(1,2) node {$\bullet$};
				\draw(2.5,2) node {$\bullet$};
				\draw(4,2) node {$\bullet$};
				\draw (4,0.5) node {$\bullet$};
			\end{tikzpicture}
		\end{center}
	\end{subfigure}
	\hfill
	\begin{subfigure}{.45\textwidth}
		\begin{center}
			\begin{tikzpicture}
				\fill[gray!10] (0,0) -- (1.5,0)-- (1.5,1.5) -- (0,1.5)-- cycle;
				\fill[gray!10](1,0.5) -- (1,2)-- (2.5,2) -- (2.5,0.5)-- cycle;
				\fill[gray!10] (0,0)  -- (0,1.5)-- (1,2) -- (1,0.5) -- cycle;
				\fill[gray!10](1.5,0) -- (1.5,1.5)-- (2.5,2) -- (2.5,0.5)-- cycle;
				\fill[gray!10] (2.5,2) -- (4,2) -- (4,0.5)-- (2.5,0.5)--cycle;
				\draw[dotted] (0.75,0) -- (1.75, 0.5)-- (1.75, 2);
				\draw (0.75,0) -- (0.75,1.5) -- (1.75, 2);
				\draw[dotted]  (0.5,1.75)--(0.5,0.25) -- (2, 0.25);
				\draw (0.5,1.75)-- (2,1.75) --  (2, 0.25);
				\draw (0,0.75) --  (1.5,0.75)-- (2.5, 1.25) -- (4,1.25);
				\draw (3.25,0.5) -- (3.25, 2);
				\draw[dotted] (0,0.75)-- (1,1.25)-- (2.5, 1.25);
				\draw[dotted] (0.75,0.75)--(1.75,1.25); 
				\draw[dotted] (0.5, 1)--(2, 1);
				\draw[dotted] (1.25, 0.25)--(1.25, 1.75);
				\draw (0,0) -- (1.5,0);
				\draw (1.5,1.5) -- (0,1.5);
				\draw (1.5,0)-- (1.5,1.5);
				\draw (0,1.5)-- (0,0);
				\draw[dotted] (1,0.5) -- (1,2);
				\draw (1,2)-- (2.5,2); 
				\draw (2.5,2) -- (2.5,0.5);
				\draw[dotted] (2.5,0.5)-- (1,0.5);
				\draw[dotted](0,0) -- (1,0.5);
				\draw(1.5,0) -- (2.5,0.5);
				\draw(0,1.5) -- (1,2);
				\draw(1.5,1.5)  -- (2.5,2);
				\draw (2.5,2) -- (4,2);
				\draw(4,2) -- (4,0.5);
				\draw (4,0.5)-- (2.5,0.5);
				\draw[red](0,0) node {$\bullet$};
				\draw[orange](0,0.75) node {$\bullet$};
				\draw[orange](0.75,0) node {$\bullet$};
				\draw[yellow](0.75,0.75) node {$\bullet$};
				\draw[orange](1.5,0.75) node {$\bullet$};
				\draw[orange](2.5,1.25) node {$\bullet$};
				\draw[orange](1,1.25) node {$\bullet$};
				\draw[yellow](1.75,1.25) node {$\bullet$};
				\draw[orange](1.75,0.5) node {$\bullet$};
				\draw[orange](1.75,2) node {$\bullet$};
				\draw[orange](0.75,1.5) node {$\bullet$};
				\draw[red] (1.5,0) node {$\bullet$};
				\draw[red] (0,1.5) node {$\bullet$};
				\draw[red] (1.5,1.5) node {$\bullet$};
				\draw[red] (1,0.5) node {$\bullet$};
				\draw[orange] (0.5,0.25) node {$\bullet$};
				\draw[orange] (2,0.25) node {$\bullet$};
				\draw[yellow] (2,1) node {$\bullet$};
				\draw[orange] (2,1.75) node {$\bullet$};
				\draw[orange] (0.5,1.75) node {$\bullet$};
				\draw[yellow] (0.5,1) node {$\bullet$};
				\draw[red] (2.5,0.5) node {$\bullet$};
				\draw[red] (1,2) node {$\bullet$};
				\draw[orange](3.25,2) node {$\bullet$};
				\draw[orange](3.25,0.5) node {$\bullet$};
				\draw[yellow](3.25,1.25) node {$\bullet$};
				\draw[orange](4,1.25) node {$\bullet$};
				\draw[red] (2.5,2) node {$\bullet$};
				\draw[red] (4,2) node {$\bullet$};
				\draw[red] (4,0.5) node {$\bullet$};
				\draw[yellow] (1.25,0.25) node {$\bullet$};
				\draw[teal] (1.25,1) node {$\bullet$};
				\draw[yellow] (1.25,1.75) node {$\bullet$};
			\end{tikzpicture}
		\end{center}
	\end{subfigure}
	\caption{Example of a median graph with its cubical subdivision (respectively red, orange, yellow and green vertices correspond respectively to 0,1,2,3-dimensional cubes).\label{Fig_ex}}
\end{figure}
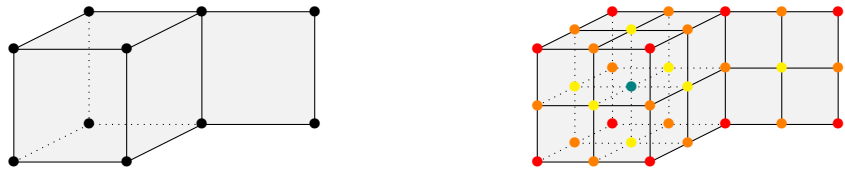

\begin{definition}
	A \emph{cubical orientation} on a median graph $\mathcal{G}$ is an orientation of all the edges of $\mathcal{G}$ so that two opposite sides in a 4-cycle have parallel orientations.
\end{definition}

Note that it is equivalent to have a map that associates to any hyperplane $H$ one of the two halfspaces that it delimits, denoted by $H^+$. 
Be aware that usually an orientation of a median graph requires also the condition that for any two hyperplane $H_1$ and $H_2$,  $H_1^+\cap H_2^+\neq\emptyset$; in this survey, we will not ask for this extra-condition.
It is always possible to endow a median graph with a cubical orientation. For instance, fix a vertex $v\in \mathcal{G}$ in the median graph. For any hyperplane $H$, $H^+$ is the halfspace containing $x$.

\subsection{Hyperplanes}
A fundamental tool in the study of median graphs is the notion of ``hyperplanes''. Throughout this section $\mathcal{G}$ is a median graph.

Two edges $e$ and $f$ of $\mathcal{G}$ are \emph{equivalent} if there exists a sequence of edges $(e_i)_{0\leq i \leq n}$ such that $e_0=e$, $e_n=f$ and for every $0\leq i \leq n-1$, the two edges $e_i$ and $e_{i+1}$ are the opposite edges of a $4$-cycle of $\mathcal{G}$. 

\begin{definition}
A \emph{hyperplane} of a median graph $\mathcal{G}$ is the equivalence class of an edge $e$ and it will be denoted by $[e]$. 
\end{definition}

For instance, the median graph of Figure \ref{fig:cc} admits $4$ hyperplanes: the blue one, the red one, the dark one and the brown one.

A subgraph $Y$ of $\mathcal{G}$ is \emph{convex} if all geodesics between two vertices in $Y$ are contained in $\mathcal{G}$.
An important property of convex subgraphs in median graphs is that they satisfy the \emph{Helly property for median graphs}:

\begin{proposition}\label{prop:Helly}
Let $\mathcal{G}$ be a median graph and $Y_1,\dots,Y_n$ convex subgraphs that pairwise intersect. Then the intersection $\bigcap_{1\leq i\leq n}Y_i$ is non-empty.
\end{proposition}

Note that the Helly property for median graphs does not hold for an infinite family of convex subgraphs.

 	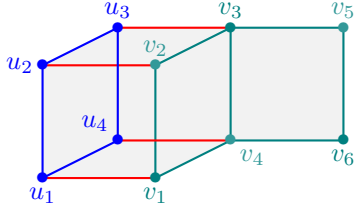
\begin{figure}
	\begin{center}
		\begin{tikzpicture}
			\fill[gray!10] (0,0) -- (1.5,0)-- (1.5,1.5) -- (0,1.5)-- cycle;
			\fill[gray!10](1,0.5) -- (1,2)-- (2.5,2) -- (2.5,0.5)-- cycle;
			\fill[gray!10] (0,0)  -- (0,1.5)-- (1,2) -- (1,0.5) -- cycle;
			\fill[gray!10](1.5,0) -- (1.5,1.5)-- (2.5,2) -- (2.5,0.5)-- cycle;
			\fill[gray!10] (2.5,2) -- (4,2) -- (4,0.5)-- (2.5,0.5)--cycle;
			\draw[thick,red] (0,0) -- (1.5,0);
			\draw[thick,red] (1.5,1.5) -- (0,1.5);
			\draw[thick,teal] (1.5,0)-- (1.5,1.5);
			\draw[thick, blue] (0,1.5)-- (0,0);
			\draw[thick, blue] (1,0.5) -- (1,2);
			\draw[thick,red] (1,2)-- (2.5,2); 
			\draw[thick, teal] (2.5,2) -- (2.5,0.5);
			\draw[thick,red] (2.5,0.5)-- (1,0.5);
			\draw [thick,blue](0,0) -- (1,0.5);
			\draw [thick,teal](1.5,0) -- (2.5,0.5);
			\draw [thick,blue](0,1.5) -- (1,2);
			\draw [thick,teal](1.5,1.5)  -- (2.5,2);
			\draw[thick,teal] (2.5,2) -- (4,2);
			\draw[thick,teal](4,2) -- (4,0.5);
			\draw[thick,teal] (4,0.5)-- (2.5,0.5);
			\draw[blue](0,0) node {$\bullet$} node[below]{$u_1$};
			\draw[teal] (1.5,0) node {$\bullet$} node[below]{$v_1$};
			\draw[blue] (0,1.5) node {$\bullet$}node[left]{$u_2$};
			\draw[color=teal!80] (1.5,1.5) node {$\bullet$}node[above]{$v_2$};
			\draw[blue] (1,0.5) node {$\bullet$} node[above left]{$u_4$};
			\draw[color=teal!80] (2.5,0.5) node {$\bullet$} node[below right]{$v_4$};
			\draw[blue] (1,2) node {$\bullet$} node[above]{$u_3$};
			\draw[teal] (2.5,2) node {$\bullet$} node[above]{$v_3$};
			\draw[color=teal!80] (4,2) node {$\bullet$} node[above]{$v_5$};
			\draw[teal] (4,0.5) node {$\bullet$} node[below]{$v_6$};
		\end{tikzpicture}
		\caption{The red hyperplane and its two halfspaces: the blue and green subgraphs.	\label{fig:halfspace}}
	\end{center}
\end{figure}

By \cite[Theorem 4.10]{Sageev-ends_of_groups}, the graph obtained from $\mathcal{G}$ by removing the (interiors of the) edges in the hyperplane $H$ has exactly two connected components, referred to as \emph{halfspaces} (see for instance Figure \ref{fig:halfspace}), which are both convex subgraphs.  
A path in $\mathcal{G}$ is said to \emph{cross} a hyperplane $H$, if there exist two consecutive vertices of this path that are connected by an edge belonging to the equivalence class $H$.
Two vertices of $\mathcal{G}$, or more generally two subgraphs of $G$ are \emph{separated} by a hyperplane $H$ if they belong respectively to different halfspaces that $H$ delimits. 
Two distinct hyperplanes are \emph{transverse} if there exists a $4$-cycle that contains an edge (indeed $2$) in both of them, or equivalently if each halfspace delimited by one intersects the two halfspaces delimited by the second hyperplane. If they are not transverse they will be called \emph{disjoint}.
For instance, in the median graph of Figure \ref{fig:cc}, the blue hyperplane is transverse to the three other hyperplanes. The red hyperplane and the brown one are disjoint, and they both separate $x_1$ from $x_4$.

 The \emph{convex hull} of a subgraph $Y$, denoted by $\conv(Y)$ is the smallest convex subgraph of $\mathcal{G}$ containing $Y$, and it can be characterized through the following intersection of halfspaces.
 
 \begin{proposition}\label{prop_convex_hull_hyperplanes}
 Let $\mathcal{G}$ be a median graph and $Y$ be a subgraph of $\mathcal{G}$. The convex hull of $Y$ coincides with the intersection of all the halfspaces containing $Y$.
 \end{proposition}

The \emph{neighborhood} of a hyperplane $H$ (sometimes called \emph{the carrier} of $H$) and denoted by $\Nb(H)$ is the subgraph whose vertices are all the vertices belonging to some edges of $H$, and an edge links two such vertices if they are linked already in $\mathcal{G}$. The neighborhood of a hyperplane is a convex subgraph. For instance, in Figure \ref{fig:cc}, the neighborhood of the blue hyperplane is all the median graph while the neighborhood of the brown one is the $4$-cycle $v-x_4-s-x_3$. 

We say that a family of hyperplanes $H_1,\dots H_n$ \emph{generates a cube} of dimension $n$ if the intersection of their neighborhoods $\cap_{1\leq i \leq n }\Nb(H_i)$ is a cube of dimension $n$. For instance, in Figure \ref{fig:cc} the red, blue and dark hyperplanes generate a cube of dimension $3$. Cubes can be generated by pairwise families of transverse hyperplanes (see \cite{Genevois_book_median} or as a consequence of \cite[Proposition~2.1]{Sageev_lecturenotes} for instance):

\begin{proposition}\label{prop:cubes_hyperplanes}
	Let $\mathcal{G}$ a median graph. 
	A finite family of $n$ hyperplanes generates a cube of dimension $n$ if and only if they are pairwise transverse.
\end{proposition}

Note that Proposition \ref{prop:cubes_hyperplanes} can not be extended to infinite families of pairwise transverse hyperplanes as it is illustrated in the following example  \cite[Example 1.6.5]{Genevois_book_median}.

\begin{example}\label{ex:pairwise_transverse_pas_cube}
Let $\{0,1\}^{\N}$ denote the set of finitely supported sequences $\N \rightarrow \{0,1\}$. We define a graph as follows. The vertex-set consists in the set $\{0,1\}^{\N} \times \N$ and the edges are of two types:
\begin{itemize}
	\item Type 1: there is an edge connecting $(u,p)$ and $(u, p\pm 1)$;
	\item Type 2: there is an edge connecting $(u, p)$ and $(v, p)$ if $u$ and $v$ differ at a single integer $q\leq p$.
\end{itemize}
It is a median graph containing an infinite family of pairwise transverse hyperplanes as all the hyperplanes defined by edges of type 2 are pairwise transverse.
 But it is locally finite so it does not contain an infinite cube.
\end{example}

Hyperplanes are also an important tool to characterize geodesics.
	\begin{theorem}[{\cite[Theorem 4.13]{Sageev-ends_of_groups}}]\label{theorem:combinatorial_geodesic}
A path joining two vertices is geodesic if and only if it crosses each hyperplane at most once. In particular, the distance between two vertices is equal to the number of hyperplanes that separate these two vertices.
\end{theorem}

For instance, in the median graph of Figure \ref{fig:cc}, the path $[x_1,v]\cup [v,x_4]$ is geodesic while any other path joining $x_1$ and $x_4$ is not.

\subsection{Isometries}\label{subsection:isometries}
As we are interested in group actions, we describe in this section the possible isometries of median graphs. Let us introduce first some definitions.
	Let $f$ be an isometry of a median graph $\mathcal{G}$.
The \emph{translation length} of $f$ is the non-negative integer:
\[\ell(f)\coloneqq\min\{d(f(x), x)\mid x \in \mathcal{G}^0\}.\]
The set of vertices realizing the translation length of $f$ is called the \emph{minimizing set} of $f$ and it is denoted by
\[\Min(f)\coloneqq\{x\in V(\mathcal{L}) \mid \dist(x,f(x))=\ell(f)  \} .\]An isometry $f$ \emph{inverses a hyperplane $H$} if it switches its two corresponding halfspaces: $f(H^+)=H^-$ (and thus $f(H^-)=H^+$). If such hyperplane $H$ does not exist, we say that $f$ \emph{acts without inversion}. An isometry $f$ \emph{acts stably without inversion} if $f$ and all its iterates act without inversion. 

\medskip
We define now several isometry's types.
	\begin{definition}\label{def_isom}
	Let $f$ be an isometry of a median graph $\mathcal{G}$.
	\begin{itemize}
		\item $f$ is \emph{elliptic} if it fixes a vertex. In this case, the minimizing set of $f$ is also called the \emph{fixed-point set} of $f$ and it is denoted by $\Fix(f)$.
		\item $f$ is \emph{periodic} if it preserves a cube.
		\item $f$ is \emph{loxodromic} if it preserves a bi-infinite geodesic path and acts as a {non-trivial} translation along it. Such a geodesic path is called an \emph{axis of $f$}.
		\item $f$ is \emph{ helixodromic} if it stabilizes a product $Q \times L$ of a cube $Q$ with a bi-infinite geodesic path $L$ such that it acts on $L$ as a translation and on $Q$ with no fixed point.
	\end{itemize}
\end{definition}

Note that an elliptic isometry is also periodic, but all the others cases are disjoint. A periodic isometry of a locally finite dimensional median graph $\mathcal{G}$ induces an elliptic isometry on the cubical subdivision of $\mathcal{G}$; the isometry induced fixes the vertex corresponding to the cube preserved by the initial isometry.
Be aware that, often in the literature, ``elliptic isometries'' are defined as the ones fixing a point of the space and not only a vertex. Hence, when the median graph is locally of finite dimension, they correspond to our notion of ``periodic isometries''.

As shown in the following theorem the isometry's types defined in Definition \ref{def_isom} are the only ones existing for median graphs.

\begin{theorem}[\cite{Haglund_isometries_semisimple},\cite{Genevois_book_median}]\label{thm_classification_isometries}
	Let $f$ be an isometry of a median graph $\mathcal{G}$. Then $f$ is either periodic, or loxodromic, or helixodromic; the different cases being exclusive. Moreover, if $f$ acts stably without inversion then it is either elliptic or loxodromic, and in particular, for any $x\in \Min(f)$, $\dist(x,f^n(x))=n\ell(f)$ for all $n\in \N$.
\end{theorem}

The heart of the above theorem is the second part. It is due to \cite{Haglund_isometries_semisimple} where it is also shown that the condition of acting stably without inversion is not restrictive. Indeed, any isometry of a median graph $\mathcal{G}$ acts without inversion on the cubical subdivision of $\mathcal{G}$. The general statement comes from the fact that
an isometry of a median graph $\mathcal{G}$ that is elliptic on the cubical subdivision of $\mathcal{G}$ is periodic on $\mathcal{G}$, and one that is loxodromic on the cubical subdivision of $\mathcal{G}$ is either loxodromic or helixodromic on $\mathcal{G}$. 
Theorem \ref{thm_classification_isometries} can be found in \cite{Genevois_book_median}.

The following corollary is a direct consequence of Theorem \ref{thm_classification_isometries} and it is one of the main tools we will use. 
\begin{corollary}\label{prop_action_semisimple_hag}
	Let $f$ be an isometry of a cube oriented median graph $\mathcal{G}$ that preserves this orientation. Then $f$ is either elliptic or loxodromic, and in particular, for any $x\in \Min(f)$, $\dist(x,f^n(x))=n\ell(f)$ for all $n\in \N$.
\end{corollary}

The second main tool that we will use is given in the following proposition, which has been shown in \cite{Gerasimov_fixed_point_free_action} for the case of finitely generated groups and in \cite[Corollary 7.G.4 and Remark 7.F.8]{Cornulier_wallings} for the general case.

\begin{proposition}[\cite{Gerasimov_fixed_point_free_action},\cite{Cornulier_wallings}]\label{prop:fixedpoint}
	If a group acts isometrically on a median graph with a bounded orbit then it preserves a cube. Moreover, if the action preserves a cubical orientation the group fixes a vertex.
\end{proposition}

The last assertion, is an immediate consequence of the fact that a cube of an oriented median graph has always a vertex such that all the edges of the cube issued from this vertex are oriented to go away of this vertex. Hence, a group preserving a cube such that the cubical orientation is preserved, fixes this vertex.

Note that in the finite dimensional case and using the CAT(0) metric of the cube completion, Proposition \ref{prop:fixedpoint} is very well known. When the dimension is finite, the cube completion is a complete CAT(0) metric space and so such a group fixes the circumcenter of the orbit (see \cite{Bridson_Haefliger}) and then the cube preserved is the smallest (for the dimension) cube containing this fixed point.

The general case relies on the finite dimensional case as one can prove that if a group acts isometrically on a median graph with a bounded orbit, it also admits a finite orbit in the graph. The convex hull of this finite orbit is a finite median subgraph invariant by the action of the group. In order to end the proof without using the CAT(0) metric, a median argument is given in \cite[Lemma 4.1.2]{Genevois_book_median}: because the median graph is finite, halfspaces contain a finite number of vertices. Take the intersection of all halfspaces that are strictly bigger than their opposite halfspace. By a cardinality argument two such halfspaces pairwise intersect and so by Helly property for median graphs \ref{prop:Helly}, this intersection is a non-empty convex subgraph. Moreover it contains only hyperplanes whose two halfspaces contain the same number of vertices and so they are pairwise transverse, forming a cube by Proposition \ref{prop:cubes_hyperplanes}.

\subsection{Purely elliptic actions on median graphs}\label{Subsection_purely_elliptic}
In this subsection, we are interested in the orbits of groups acting on median graphs \emph{purely elliptically} meaning that any element fixes a vertex of the graph.  Notice that having a bounded orbit is equivalent to having all orbits bounded. Moreover, by Proposition \ref{prop:fixedpoint}, it is also equivalent to preserve a cube. 

\begin{question}\label{question:purely_elliptic}
	Let $G$ be a finitely generated group acting purely elliptically on a median graph. Are the $G$-orbits bounded?
\end{question}

Without the assumption that the group is finitely generated, we can answer negatively  Question \ref{question:purely_elliptic}. For instance, the infinite torsion group $\bigoplus_{\Z} \Z / 2\Z$ admits a purely elliptic action on a tree with unbounded orbits. It is an example of a more general construction that can be found in \cite[Corollary 4.2.2]{Genevois_book_median}. The construction is explicit and it is as follows. For $i\in \Z$, let $H_i$ denote the subgroup of $\bigoplus_{\Z} \Z / 2\Z$ corresponding to the $\Z/2\Z$ of the i-th factor, and for $i\in \N^*$, let $G_i$ be the proper subgroup of $\bigoplus_{\Z} \Z / 2\Z$ generated by the subgroups $H_{-i},\dots, H_{-1}, H_0,H_1,\dots,H_i$.
Consider the graph $\mathcal{T}$ with vertex-set $\{gG_i\mid g\in\bigoplus_{\Z} \Z / 2\Z \text{ and } i\in \N^* \}$ and whose edges connect $gG_i$ and $gG_{i+1}$ for all $g\in \bigoplus_{\Z} \Z / 2\Z$ and $i\in \N^*$. $\bigoplus_{\Z} \Z / 2\Z $ acts by left multiplication on the vertex-set and the vertex-stabilizers are conjugate of the subgrougs $G_i$.

It remains to understand why $\mathcal{T}$ is a tree.
 Indeed, it is connected because any vertex $gG_i$ can be linked by a path to $G_1$:  due to the fact that $G_1\subset G_2\subset \dots $ covers $\bigoplus_{\Z} \Z / 2\Z$ there exists $j\in \N^*$ such that $g\in G_j$. If $j\leq i$ then $gG_i=G_i$ and the path is obvious, otherwise $gG_i-gG_{i+1}-\dots-gG_j=G_j-G_{j-1}-\dots G_1$ is a path between $gG_i$ and $G_1$.
To see that there is no cycle, define the height of a vertex $gG_i$ as $i$. $\mathcal{T}$ can not have any cycle because a vertex of height $i$ has a unique adjacent vertex that has height $i+1$.

\medskip

Adding a restriction on the dimension of the median graph, Question \ref{question:purely_elliptic} has been answered positively. The theorem below is optimal in the sense that there exist counter-examples to Question \ref{question:purely_elliptic} when the median graph has infinite cubes (see Subsection \ref{subsec:purely_elliptic_counterex}).
\begin{theorem}[\cite{GLU_Neretin}]\label{thm:purely_elliptic_locally_finite_dim}
	Let $G$ be a finitely generated group acting purely elliptically on a locally finite dimensional median graph $\mathcal{G}$, then $G$ has a bounded orbit.
\end{theorem}

The finite dimensional case was already known: it is a consequence of \cite[Theorem 5.1]{Sageev-ends_of_groups} that has been remarked by Pierre-Emmanuel Caprace in \cite[Proposition B.8]{Chatterji_Fernos_Iozzi}. A proof can also be found in \cite[Theorem 3.1]{Fioravanti_Tits} or in  \cite{Leder_varghese_Elliptic_action}.

\medskip
\subsubsection{Finite dimensional case}\label{subsec:purely_elliptic_finite_dim}In this section, we focus on the idea of the proof of Theorem \ref{thm:purely_elliptic_locally_finite_dim}, in the specific case of finite dimensional median graphs following \cite{Leder_varghese_Elliptic_action}. It illustrates nicely the importance of hyperplanes in median graphs.

Let us start first with a baby case: the case of (simplicial) trees. 

\begin{theorem}[\cite{Serre1980}]\label{thm:elliptic_action_tree}
Let $G$ be a finitely generated group acting on a tree $\mathcal{T}$ and generated by $g_1,\dots, g_n$. Assume that $g_ig_j$ induces an elliptic isometry of $\mathcal{T}$ for all $0\leq i<j\leq n$ (with the convention that $g_0=\id$). 
Then the group $G$ fixes a vertex of $\mathcal{T}$.
\end{theorem}

\begin{corollary}[\cite{Serre1980}]Let $G$ be a finitely generated group acting purely elliptically on a tree, then the group $G$ fixes a vertex of this tree.
\end{corollary}

Periodic isometries of trees are either elliptic, or fix the middle of an edge. Indeed, as already mentioned, they induce elliptic isometries on the cubical subdivision of the initial tree, which is also a tree. Hence, as an immediate consequence we have:

\begin{corollary}[\cite{Serre1980}]Let $G$ be a finitely generated group acting purely periodically on a tree, then the group $G$ fixes a vertex or the middle point of an edge.
\end{corollary}

Indeed actions on trees are really rigid. For instance, if you consider an elliptic isometry $g$ of a tree $\mathcal{T}$, then it has the strong property that for any vertex $v\in \mathcal{T}^0$, the middle point $m(v,gv)\in\mathcal{T}^0$ between $v$ and $gv$ is fixed by $g$.

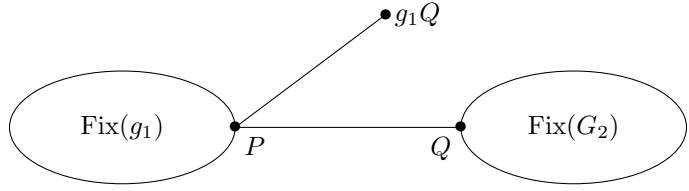
\begin{figure}
	\begin{center}
		\begin{tikzpicture}
			\draw (0,0) node {$\bullet$} node[below right]{$P$};
			\draw (3,0) node {$\bullet$} node[below left]{$Q$};
			\draw (2,1.5) node {$\bullet$} node[right]{$g_1Q$};
			\draw (0,0) -- (2,1.5);
			\draw (0,0) -- (3,0);
			\draw (4.5,0) node{$\Fix(G_2)$};
			\draw (-1.5,0) node{$\Fix(g_1)$};
			\draw (-1.5,0) circle (1.5 and 0.75);
			\draw (4.5,0) circle (1.5 and 0.75);
		\end{tikzpicture}
		\caption{The fixed-point sets of $g_1$ and $G_2$ are disjoint. \label{fig:purely_elliptic_tree}}
	\end{center}
\end{figure}

\begin{proof}[Proof of Theorem \ref{thm:elliptic_action_tree}]
	We prove the theorem by induction on the number of generators of $G$. If $G$ is cyclic, the result is immediate. Assume that $G$ is generated by $n$ elements. We denote by $G_2$ the subgroup of $G$ generated by $g_2,\dots, g_n$. 
Assume by contradiction that the fixed-point sets of $g_1$ and $G_2$ are disjoint. Because we are in a tree, there exists a unique geodesic joining $\Fix(g_1)$ and $\Fix(G_2)$, denoted by $[P,Q]$ with $P\in\Fix(g_1)$, $Q\in \Fix(G_2)$ and this geodesic realizes the distance between the two fixed-point sets (see Figure \ref{fig:purely_elliptic_tree}).
 This implies that $g_1Q$ is not equal to $Q$ and that the concatenation of the geodesics $[g_1Q,P]\cup[P,Q]$ is still geodesic. Moreover, $g_1Q=g_1g_kQ$ for all $2\leq k\leq n$, and $g_1g_k$ induces an elliptic isometry by assumption. As a consequence, $P$, which is the middle of $g_1g_kQ$ and $Q$, is fixed by $g_1g_k$, for all $2\leq k\leq n$. This implies that $P$ is fixed by $g_k$, for all $2\leq k \leq n$, hence by $G_2$, that is a contradiction.  
\end{proof}

\begin{theorem}[\cite{Sageev-ends_of_groups}]\label{thm:purely_elliptic_finite_case}
	Let $G$ be a finitely generated group acting purely elliptically on a finite dimensional median graph, then $G$ has a bounded orbit.
\end{theorem}

The strategy of the proof is the following. Assuming that $G$ does not have a bounded orbit will allow us to build a loxodromic element of $G$ that will lead to a contradiction. To do so, the game is to find three parallele hyperplanes in a same orbit:

\begin{lemma}[\cite{Sageev-ends_of_groups}]\label{fact:loxodromic}
Consider the action of a group $G$ on a median graph such that there exists a hyperplane $H$ and two elements $g,h\in G$ such that the hyperplanes $H$, $gH$ and $hH$ are pairwise disjoint with  $gH$ that separates $H$ and $hH$. Then $G$ contains a loxodromic element.
\end{lemma}

Indeed, the assumption implies that once given a cubical orientation on the median graph $\mathcal{G}$, there exist two positive halfspaces or two negative halfspaces of these hyperplanes such that one is included in the other one providing a loxodromic element; if for instance $gH^+\subset H^+$ then $g$ is loxodromic.

Finding three parallele hyperplanes in the same orbit when the median graph is of finite dimension can be easily done if there exists an orbit of hyperplanes containing enough hyperplanes compared to the dimension of the median graph. For instance in the case of trees, any family of three hyperplanes has this property, but already in dimension $2$ we need more: in the standard Cayley graph of $\Z^2$ (see Figure \ref{fig:Cay_Z2}), any family of at least $5$ hyperplanes has the desired property, but not any family of $4$ hyperplanes.

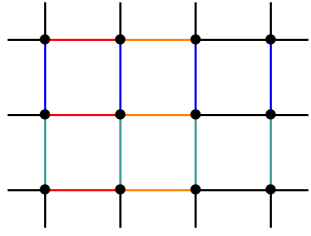
\begin{figure}
	\begin{center}
\begin{tikzpicture}
	\foreach \y in {0,1,2} \draw[thick,red] (0,\y)--(1,\y);
	\foreach \y in {0,1,2} \draw[thick,orange] (1,\y)--(2,\y);
	\foreach \y in {0,1,2} \draw[thick] (2,\y)--(3.5,\y);
	\foreach \y in {0,1,2} \draw[thick] (-0.5,\y)--(0,\y);
	\foreach \x in {0,1,2,3} \draw[thick,teal!80] (\x,0)--(\x,1);
	\foreach \x in {0,1,2,3} \draw[thick,blue] (\x,1)--(\x,2);
	\foreach \x in {0,1,2,3} \draw[thick] (\x,-0.5)--(\x,0);
	\foreach \x in {0,1,2,3} \draw[thick] (\x,2.5)--(\x,2);
	\foreach \y in {0,1,2} \foreach \x in {0,1,2,3} \draw(\x,\y)node{$\bullet$};
\end{tikzpicture}
\caption{The family of hyperplanes red, orange, blue and green does not contain three pairwise disjoint hyperplanes. \label{fig:Cay_Z2}}
	\end{center}
\end{figure}

\begin{lemma}[{\cite[Lemma 5.2]{Sageev-ends_of_groups}}]\label{fact_hyperplane_dim}
	Let $\mathcal{G}$ be a median graph of dimension $d$ and $S$ be a finite set of hyperplanes of $\mathcal{G}$. If the cardinality of $S$ satisfies $\lvert S\rvert \geq d+d(d+1)$ then there exists in $S$ three hyperplanes that are pairwise disjoint. 
\end{lemma}
 
\begin{proof}[Idea of proof of Theorem \ref{thm:purely_elliptic_finite_case}]
Consider $g_1,\dots g_n\in G$ a symmetric generating set of $G$. Assume by contradiction that its orbits are unbounded. Let $v$ be a vertex of $\mathcal{G}$.
For $1\leq i\leq n$, denote by $S_i$ the set of hyperplanes separating $v$ and $g_iv$, and by $S$ the union of the $S_i$'s: $S=\bigcup_{1\leq i\leq n}S_i$. Note that $S$ is finite. By assumption on the action of $G$, its orbits are unbounded so in particular there exists an element $g\in G$ satisfying $N:=\dist(v,gv)\geq  \lvert S\rvert \left(d+d(d+1)\right)$. Denote by $\mathcal{K}:=\{K_1,\dots,K_N\}$ the set of hyperplanes separating $v$ and $gv$. Writing $g$ as a product of generators $g=g_{i_1}\dots g_{i_k}$ for some $1\leq i_1,\dots,i_k\leq n$, the union of the following geodesics $[v,g_{i_1}v]\cup[g_{i_1}v, g_{i_1}g_{i_2}v]\cup \dots \cup [g_{i_1}\dots g_{i_{k-1}}v,gv]$ is a path joining $v$ and $gv$. Hence any hyperplane $K_j$ crosses at least one of the above geodesics, meaning, by construction, that they are in the orbit of a hyperplane of $S$. By assumption on $N$ and by the pigeonhole principle, there exist a hyperplane $J$ of $S$ and a subfamily of $\mathcal{K}$ of size at least $d+d(d+1)$ whose all hyperplanes are in the orbit of $J$. By Lemma \ref{fact_hyperplane_dim}, this implies that in the orbit of $J$ there exists 3 hyperplanes that are pairwise disjoint. We get the contradiction using Lemma \ref{fact:loxodromic} who provides a loxodromic element in $G$.
\end{proof}

\subsubsection{Case of locally finite dimensional median graphs}\label{subsec:purely_elliptic_loc_finite_dim}
In the infinite-dimensional case, the previous proof fails because we can not use Lemma \ref{fact_hyperplane_dim} anymore. 
Nevertheless, with the restriction that the dimension is locally finite, the same result holds \cite{GLU_Neretin}. We present here an idea of the proof.

The following lemma allows us to first restrict the action of $\mathcal{G}$ on the convex hull of the orbit of a vertex $x_0\in \mathcal{G}$. It is well known, and it is contained in the proof of \cite[Theorem~5.1]{Sageev-ends_of_groups} (see also \cite{GLU_Neretin} or \cite{Genevois_book_median} for instance).

\begin{lemma}\label{lem:ConvexHullOrbit}
	Let $G$ be a finitely generated group acting on a median graph $\mathcal{G}$. For every vertex $x_0 \in \mathcal{G}^0$, the convex hull of the orbit $G \cdot x_0$ is a $G$-invariant convex subgraph on which $G$ acts with finitely many orbits of hyperplanes.
\end{lemma}

Recall that the distance $\dist_{\infty}$ has been introduced in Subsection \ref{subsec:cc} for CAT(0) cube complexes, but we can define it for median graphs using cube completion. Another equivalent way to define it for median graphs, using Proposition \ref{prop:cubes_hyperplanes}, is as follows. Given a median graph $\mathcal{G}$ and two vertices $x,y \in \mathcal{G}^0$, $\dist_{\infty}(x,y)$ is the maximal number of pairwise disjoint hyperplanes that separate $x$ and $y$.

The main difficulty to deal with, in the infinite-dimensional case, is that the dimension of cubes being not uniformly bounded, it could happen that $G$ has an unbounded orbit but that the orbit of a point is bounded with respect to the distance $\dist_{\infty}$. Indeed, in this case, it is not always possible to find three pairwise disjoint hyperplanes; preventing us to use Lemma \ref{fact:loxodromic} to build easily a loxodromic element.

\begin{proof}[Idea of proof of Theorem \ref{thm:purely_elliptic_locally_finite_dim}.]
	As a consequence of Lemma \ref{lem:ConvexHullOrbit}, we can assume that $\mathcal{G}$ coincides with the convex hull of the orbit $G \cdot x_0$ for some $x_0 \in \mathcal{G}^0$.

	 $\mathcal{G}$ has to be bounded with respect to $\dist_{\infty}$. Indeed, this can be proved by contradiction, using that, by Lemma \ref{lem:ConvexHullOrbit}, the number of orbits of hyperplanes is finite and then using the pigeonhole principle and Lemma \ref{fact:loxodromic} to build a loxodromic element.

	As a consequence, given a halfspace $H_1$ delimited by a hyperplane $H$, we can define the \emph{depth} of $H_1$ by 
	\[p(H_1):= \max \left\{ d_\infty(x,N(H)) \mid x \in H_1 \right\}.\]
	A hyperplane is \emph{balanced} if the two halfspaces it delimits have the same depth, and it is \emph{unbalanced} otherwise. If $H$ is an unbalanced hyperplane, we call \emph{larger halfspace} (respectively \emph{thinner halfspace}) the halfspace that has the larger (respectively the thinner) depth and we denote them respectively by $H^{\ell}$ and $H^t$. If $H$ is a balanced hyperplane, we denote by $p(H)$ the common depth of the two halfspaces it delimits. For instance in Figure \ref{fig:halfspace}, the red hyperplane is unbalanced, the blue halfspace has depth $0$ and it is the thinner one, whereas the halfspace green has depth $1$ and is the larger one. In Figure \ref{fig:cc}, the blue hyperplane is balanced and its depth is $0$, whereas the red, the brown and the black ones are unbalanced.

	 It is not so hard to see that the intersection of two larger halfspaces is never empty, and that two balanced hyperplane are always transverse. 
	 
	 If $\mathcal{G}$ contains only balanced hyperplanes, they are pairwise transverse and so by Proposition \ref{prop:cubes_hyperplanes} it must be a cube, possibly infinite dimensional. But we know by assumption that $\mathcal{G}$ does not contain an infinite-dimensional cube, so we conclude that $G$ stabilises a finite cube, and that its orbits are bounded, as desired.
	
	Otherwise, we consider the intersection of all the larger halfspaces:
   \[	\mathcal{C}:=\bigcap\limits_{\text{$H$ unbalanced}} H^{\ell}.\]
   For instance in Figure \ref{fig:cc}, the intersection of all the larger halfspace is the edge $[v,x_3]$.
	If $\mathcal{C}$ is non-empty, then it defines a $G$-invariant convex subcomplex as a consequence of the convexity of halfspaces. Moreover, it contains only pairwise transverse hyperplanes, so we conclude as previously that $\mathcal{C}$ is a finite cube, and that the orbits of $G$ are bounded.

	The final step is to prove that $\mathcal{C}$ can not be empty. We do it by contradiction. By Helly property for median graphs (Proposition \ref{prop:Helly}) it implies in particular that the set of unbalanced hyperplanes is infinite. Using the following lemma, it remains to construct a geodesic ray. 
	
	\begin{lemma}\label{fact:RayDiamInfty}
		Let $\mathcal{G}$ be a median graph and $\rho$ a geodesic ray. If $\mathcal{G}$ does not contain an infinite cube, then $\rho$ has infinite diameter with respect to $d_\infty$.
	\end{lemma}
	
We do it as follows.
	We construct a sequence of vertices $(v_i)_{i \geq 0}$ and a sequence of unbalanced hyperplanes $(H_i)_{i\geq 1}$ in the following way.
	\begin{itemize}
		\item We fix an arbitrary vertex $v_0 \in \mathcal{G}^0$.
		\item If $v_0,\ldots, v_i$ and $H_1,\ldots, H_i$ are defined, we fix an unbalanced hyperplane $H_{i+1}$ satisfying $v_i \in H_{i+1}^{t}$ and we define the vertex $v_{i+1}$ as the projection of $v_i$ onto $H_1^{\ell} \cap \cdots \cap H_{i+1}^{\ell}$.
	\end{itemize}
	Note that the projection is well defined because the intersection $H_1^{\ell} \cap \cdots \cap H_{i+1}^{\ell}$ is non-empty by the Helly property for median graphs (Proposition \ref{prop:Helly}). Observe that, by construction, we have $v_i \in \bigcap_{k=1}^i H_k^{\ell}$ for every $i \geq 1$, hence by assumption, it is always possible to find a hyperplane $H_{i+1}$ such that $v_i\in H_{i+1}^t$.

	Let us illustrate this construction in the case of Figure \ref{fig:cc}. Note that we can not get an infinite sequence as our cube complex is finite. Nevertheless it illustrates how the construction works. If we start for instance with the vertex $u$, we can choose as first hyperplane the red one (the black one would have been also a possible choice). Then the projection is the vertex $u_1$. Then the second hyperplane has to be the black one and the vertex $v_2$ is $x_3$. And then we can not continue because we are indeed in the intersection of all the larger halfspaces of the median graph.

For every $i \geq 0$, fix a geodesic $[v_i,v_{i+1}]$ between $v_i,v_{i+1}$. The last step in order to achieve the proof is to show that the concatenation $[v_0,v_1] \cup [v_1, v_2] \cup \cdots$ defines a geodesic ray.
\end{proof}

\subsubsection{Case of infinite-dimensional median graphs not locally of finite dimension}
\label{subsec:purely_elliptic_counterex}

In the case of infinite-dimensional median graphs that are not locally of finite dimension, the answer to Question \ref{question:purely_elliptic} is no. Indeed, there exist several counterexamples to this question. 

 For instance, consider the Grigorchuk group. It is a famous example of a finitely generated infinite torsion group. Hence it acts on any median graph purely elliptically by Proposition \ref{prop:fixedpoint}.
It is known that Grigorchuk groups have Schreier graphs with 2 ends (\cite{Grigorchuk_Krylyuk_Shreiergraph}), and \cite{Sageev-ends_of_groups} proved that if a finitely generated group has a Schreier graph with more than one end, one can construct an unbounded action on a median graph. 
 More recently, \cite{Schneeberger_CCC_Grigorchuk} proved that many Grigorchuk groups act properly on a median graph.
 
Another counterexample can be found in \cite{Osajda_Cubization}. Recall that the free Burnside group $\Burn(m,n)$ is defined by the presentation
 \[ \Burn(m,n)= \left\langle s_1,s_2,...,s_m \mid w(s_1,s_2,...,s_m)^n\right\rangle \]
 where $w$ runs over all words in $s_1, s_2, \dots , s_m$.
 
 \begin{theorem}[\cite{Osajda_Cubization}]
 	If the free Burnside group $\Burn(m,n)$ is infinite then, for every integer $k>1$, the free Burnside group $\Burn(m, kn)$ acts without bounded orbits on a median graph.
 \end{theorem}
 
 Note that knowing if $\Burn(m,n)$ is finite or not is known as the Burnside problem. Many results have been proven in this direction. Let us mention for instance, the result from
 \cite{Adjan_Novikov_unbounded_Burnside_problem}
and from \cite{Ivanov_Burnside_groups} that $\Burn(m,n)$ is infinite when $m>1$, $n\geq 248$ and $n$ is divisible by $29$ if $n$ is even.

Wreath products is a good family to find counterexamples to Question \ref{question:purely_elliptic}. See for instance \cite[Corollary 16.1.2(i)]{Genevois_book_median} or the appendix of \cite{Osajda_Cubization} by Mikaël Pichot. Let $H$ be an infinite finitely generated torsion group (for instance a infinite free Burnside group) and consider the wreath product $\Z/2\Z\  \wr \ H$. It is a finitely generated torsion group and it acts on $\bigoplus_H \Z/2\Z$, which can be seen as the 1-skeleton of an infinite cube and so a median graph. It can be shown that it acts with unbounded orbits (see \cite[Fact 16.1.1]{Genevois_book_median}).

\medskip

\subsubsection{Related questions}
Instead of adding restrictions in order to solve positively Question \ref{question:purely_elliptic}, one can also ask which groups that act purely elliptically on a median graph, have bounded orbits?

An instance of such groups are the ones having the \emph{FW property}, i.e., they are the ones such that every action on a median graph admits bounded orbits. Several equivalent definitions of this property exist (see for instance \cite{Cornulier_wallings}). For instance, groups with \emph{FH property}, i.e. groups such that every isometric action on a Hilbert space has a fixed point, have the FW property. By a theorem of Delorme and Guichardet, property FH is equivalent to the Kazhdan's Property T in the case of countable groups. Such groups are for instance: special linear groups $\SL(n,\Z)$ of dimension $n>2$, but also simple real Lie groups of real rank at least $2$, or more generally simple algebraic groups of rank at least $2$ over a local field.
Nevertheless, groups with property FH are not the only examples of groups with property FW: for instance, 
the groups $\SL(2,\Z[\sqrt{k}])$, where $k$ is a non-square positive integer, or $\SL(2,\mathcal{O}_K)$ where $\mathcal{O}_K$ is the ring of integers of a number field $K$ with $[K:\Q]\geq 2$; see for instance \cite[Example 6.A.8]{Cornulier_wallings}.

 Refinements of the FW property have been investigated by \cite{Genevois_fixedpt_property_median}, where he introduced the property $FW_n$ that denotes the fixed-point property for median graphs of dimension $n$. He constructed groups satisfying the $FW_n$ property but not the $FW_{n+1}$ one. 

Finally, let us mention that  \cite{Haettel_Osajda_purely_ellitpic} have investigated Question \ref{question:purely_elliptic} extending the scope of the actions considered; from median graphs to complexes. They conjectured that every purely elliptic action of a finitely generated group on a finite-dimensional nonpositively curved complex has bounded orbit; for reasonable notions of dimension and nonpositively curvature. They proved their conjecture for a large class of complexes including for instance all infinite families of Euclidean buildings, Helly complexes, uniformly locally finite Gromov hyperbolic graphs.

\section{Cremona groups}\label{Section_Cremona_groups}

The aim of this section is to introduce Cremona groups as well as all material of algebraic geometry needed 
in order to understand the constructions of the median graphs. In Subsection \ref{Subsection_State_art}, we do a brief state of the art to understand in which context these constructions have been developed.

\subsection{Birational geometry}\label{Subsection_bir_geometry}
This subsection is aimed at non-algebraic geometers and will allow us to settle the notation. This material is standard and this part is largely inspired by \cite{Lamy_book}. It can be also found in any book which introduces algebraic geometry, like for instance \cite{Shafarevich_book_1} and \cite{Shafarevich_book_2}.

We make the choice to introduce everything over an algebraically closed field and not to speak about schemes in order to introduce the least possible amount of notions. Hence we fix here once and for all $\kk$ to be an algebraically closed field.

\medskip
\subsubsection{Objects}
The objects studied in algebraic are the algebraic varieties. So let us introduce them first.

Recall that the \emph{affine space} of dimension $n$ is the set of all $n$-tuples of elements in $\kk$, it is denoted by $\A^n$:
\[ \A^n:=\{(a_1,\dots,a_n)\mid a_i\in \kk\},\]
and that the\emph{ projective space} $\PP^n$ is the quotient of the space $\A^{n+1}\setminus \{(0,\dots,0)\}$ by homotheties: $(a_0,\dots, a_n)\sim (b_0,\dots, b_n)$ if and only if there exists $\lambda \in \kk^*$ such that $(a_0,\dots, a_n)=\lambda (b_0,\dots, b_n)$. Such a class will be denoted by $[a_0:\dots: a_n]\in \PP^n$.

A \emph{projective variety} (respectively an \emph{affine variety}) is a subset $X$ of the projective space $\PP^n$ (respectively $Y$ of the affine space $\A^n$) defined as the common zero locus of a finite collection of homogeneous polynomials in $n+1$ variables over $\kk$, $P_i\in\kk[X_0,\dots,X_n]$ for $1\leq i \leq k$ (respectively of a finite collection of polynomials in $n$ variables over $\kk$, $Q_i\in\kk[X_1,\dots,X_n]$ for $1\leq i \leq k$): 
\begin{align*}
X&=\{[a_0:\dots: a_n]\in \PP^n \mid P_i(a_0,\dots,a_n)=0 \text{ for all } 1\leq i \leq k \},\\
Y&=\{(a_1,\dots, a_n)\in \A^n \mid Q_i(a_1,\dots,a_n)=0 \text{ for all } 1\leq i \leq k \}.
\end{align*}

\begin{example}\label{ex_Segre_embedding} $\PP^2\times \PP^1$ is a projective variety via the Segre-embedding:
\[\begin{array}{ccc}
	\PP^2\times \PP^1& \hookrightarrow &\PP^5 \\
	\left([x_0:x_1:x_2], [t_0:t_1]\right)&\mapsto& [x_0t_0:x_0t_1:x_1t_0:x_1t_1:x_2t_0:x_2t_1].
\end{array}\]
Denoting by $[X_0:\dots :X_5]$ the homogeneous coordinates on $\PP^5$, the image of this embedding, called the Segre threefold is given by  the intersection:
\[\{X_0X_3-X_1X_2=0\}\cap \{X_0X_5-X_4X_1=0\}\cap \{X_2X_5-X_3X_4=0\}.\]
\end{example}

\begin{example}\label{ex_hirzebruch_surfaces}
The \emph{ Hirzebruch surface of index $n\geq 0$}, denoted by $\F_n$, is:
\[\F_n =\{\left([x_0:x_1:x_2], [t_0,t_1]\right)\in \PP^2\times \PP^1\mid x_0t_1^n=x_1t_0^n \}.\] It is a projective surface via the Segre-embedding $\PP^2\times \PP^1 \hookrightarrow \PP^5$ of Example \ref{ex_Segre_embedding}.
\end{example}

We can define a topology on a projective variety $X$ (respectively on an affine variety $Y$), called the \emph{Zariski topology}, by taking as closed sets the projective subvarieties of $X$ (respectively the affine subvarieties of $Y$). A variety is \emph{irreducible} if it is not the union of two proper closed subsets. The \emph{dimension} of an irreducible variety $X$ is the supremum of the length of chains $Y_0\subsetneq Y_1\subsetneq \dots \subsetneq Y_n$ where each $Y_i$ is an irreducible subvariety of $X$. Note that the projective space $\PP^n$ is an irreducible projective variety of dimension $n$ and that $\A^n$ is an irreducible affine variety of dimension $n$.

Consider $\PP^n$ with coordinates $[x_0:\dots:x_n]$. Then the \emph{standard open cover} of $\PP^n$ is given by $U_i=\PP^n\setminus \{x_i=0\}$, for $0\leq i\leq n$. Note that each $U_i$ is homeomorphic to $\A^n$ using deshomogenization:
\[\begin{array}{cccc}
	& U_i &\rightarrow & \A^n \\
	& [x_0:x_1:\dots : x_n] & \mapsto & (\frac{x_0}{x_i},\dots,\frac{x_{i-1}}{x_i},\frac{x_{i+1}}{x_i},\dots, \frac{x_n}{x_i})\\
\end{array}\]
and homogenization:
\[\begin{array}{cccc}
	& \A^n &\rightarrow & U_i \\
	& (t_1, \dots, t_n)& \mapsto &  [t_1:\dots : t_{i-1} :1 : t_{i+1}: \dots  t_n]. \\
\end{array}\]
Note that restricting to the $U_i$'s, any projective variety admits an affine open cover.

Let $X\subset \PP^n$ be an irreducible projective variety. We consider the ideal associated to it:
\[I(X):=\left\{ P\in \kk[X_0,\dots,X_n]\mid P(X)=0\right\}.\] Let $ \{P_i\}_{1\leq i\leq k}$ be a minimal generating set of $I(X)$. The variety $X$ is \emph{smooth} if for any point $p\in X$, the rank of the Jacobian $\left(\frac{\partial P_i}{\partial X_j}(p) \right)$ is equal to the codimension $n-\dim X$ of $X$.

A \emph{quasi-projective variety}, is an open subset of a projective variety. When not mentioned otherwise, in all what follows a \emph{variety} will be a quasi-projective variety. Both affine and projective varieties are examples of quasi-projective varieties. On the other hand, $\PP^2\times \A^1$ is a quasi-projective variety that is neither projective nor affine. 
\medskip 

\subsubsection{Maps}
Let us define now morphisms between varieties, they are given locally by polynomials. Let us be more precise.
Let $V\subset \A^n$ be an affine variety. A function $f:V\rightarrow \kk$ is \emph{regular} if there exists a polynomial $P\in\kk[X_1,\dots,X_n]$ such that $f(x)=P(x)$ for all $x\in V$. 

A map $f: U\subset \A^m\rightarrow V\subset \A^n$ between affine varieties is a \emph{morphism} if and only if $f=(f_1,\dots,f_n)$ with $f_i$'s regular for each $i$.
A map $f: X \rightarrow Y$ between varieties is a \emph{morphism} if locally at each point of $X$ it restricts as a morphism between affine open sets. In this case, we say that $X$ dominates $Y$. 
A morphism $f: X\rightarrow Y$ is an \emph{isomorphism} if there exists a morphism $g: Y\rightarrow X$ such that $gf=\id_X$ and $fg=\id_Y$. 
\begin{example}\label{ex_blow_up}
Consider the Hirzebruch surface of index $1$:
\[\F_1 =\{\left([x_0:x_1:x_2], [t_0,t_1]\right)\in \PP^2\times \PP^1\mid x_0t_1=x_1t_0 \}.\] 
The projection $\pi : \F_1\rightarrow \PP^2$ on the first factor is a morphism. Indeed, in the chart $\{x_2=1\}\times \{t_0=1\}$ we can write $\pi$ locally as:
\[\begin{array}{ccccc}
	 \A^2 &\hookrightarrow& \F_1 &\rightarrow &\A^2 \\
	(x_0,t_1)& \mapsto &  \left([x_0:x_0t_1:1],[1:t_1]\right) & \mapsto& (x_0,x_0t_1)\\
\end{array}.\] 
\end{example}

The map of Example \ref{ex_blow_up} is called \emph{the blow-up of $[0:0:1]$ in $\PP^2$}. The blow-up of $[0:0:1]$ in $\PP^2$ is unique up to automorphisms of $\PP^2$ fixing the point $[0:0:1]$. Note that 
replacing $\PP^2$ by $\A^2$ in $\F_1$, we obtain the blow-up of the origin of the affine plane, which can be written locally in the same way as above: $(x_0,t_1)\mapsto (x_0,x_0t_1)$. Hence, it is possible to blow-up any point of any smooth quasi-projective variety.

The blow-up of $[0:0:1]$ is not an isomorphism as the preimage of the origin is $\PP^1$:  $\pi^{-1}(0,0)=\{[0:0:1]\}\times\PP^1$. We will call this preimage \emph{the exceptional divisor of the blow-up} and we denote it by $E_{[0:0:1]}$. We refer to the nice cover of the book \cite{Shafarevich_book_1} to have a picture of a blow-up. 
Note that there exists a bijection between the points on the exceptional divisors and the lines passing through the point $[0:0:1]$. The preimages of these lines do not intersect anymore in $\F_1$.

We can also blow-up points in variety of higher dimension. For instance, the blow-up of $\PP^n$ at the point $p=[0:\dots:0:1]$ is defined as the variety \[Y=\{\left([x_0:\dots:x_n], [t_0:\dots: t_{n-1}]\right)\in\PP^n\times \PP^{n-1} \mid x_it_j-x_jt_i=0 ,\  0\leq i<j\leq n-1 \}.\]


 Isomorphisms are very rigid and so the classification of varieties up to isomorphisms is really too complicated. Hence, in this context, the good notion to consider is the one of birational maps, i.e., maps that are given locally by quotients of polynomials.

Let $X$ and $Y$ be varieties. A \emph{rational map} $X\dashrightarrow Y$ is an equivalence class of pairs $(U,g)$, where $U\subset X$ is an open dense set and $g\colon U\to Y$ a morphism. Two pairs $(U, g_1)$ and $(V, g_2)$ are equivalent if $g_1|_{U\cap V}=g_2|_{U\cap V}$. By abuse of notation, we denote a rational map represented by $(U,g)$ just by $g$. 
The \emph{indeterminacy locus} of a rational map $f\colon X \dashrightarrow Y $ is the closed subset $\Ind(f)\subset X $ consisting of all the points of $X $, where $f$ is not defined, i.e., the points $p\in X $ such that there exists no representative $(U , g)$ of $f$ such that $U $ contains $p$. Let $f: X\dashrightarrow Y$ be a rational map and $V\subset X$ be an irreducible subvariety of $X$ not contained in $\Ind(f)$. \emph{The strict transform of $V$ by $f$} is $f(V):=\overline{f(V\setminus \Ind(f))}$; it is a well-defined irreducible subvariety of $Y$.

For instance, the inverse of the blow-up seen in Example \ref{ex_blow_up} $\pi^{-1} : \PP^2 \dashrightarrow \F_1$ is a rational map given by the equivalence class $\left(\PP^2\setminus \{[0:0:1]\}, \pi^{-1}\right)$ and \[\begin{array}{ccc}
	\PP^2\setminus \{[0:0:1]\} &\hookrightarrow& \F_1\\
	\left[x_0:x_1:x_2\right] & \mapsto &  \left([x_0:x_1:x_2],[x_0:x_1]\right)\\
\end{array}.\] Its indeterminacy locus is the point $\Ind(\pi^{-1})=\{[0:0:1]\}$. The strict transform of the line $\{x_1=0\}\subset \PP^2$ intersects $E_{[0:0:1]}$ in the point $[1:0]$.

A rational map represented by $(U, g)$ is \emph{dominant}, if the image of $g$ contains an open dense set of $Y$. If $f\colon Y\dashrightarrow Z$ and $g\colon X\dashrightarrow Y$ are dominant rational maps, we can compose $f$ and $g$ in the obvious way and obtain a dominant rational map $fg\colon X\dashrightarrow Z$. A \emph{birational map} from $X$ to $Y$ is a rational dominant map $f :X\dashrightarrow Y$ that admits an inverse which is a dominant rational map $f^{-1}\colon Y\dashrightarrow X$, i.e. such that $ff^{-1}=\id_Y$ and $f^{-1}f=\id_X$. Another way to say it is that there exist $U\subset X$ and $ V\subset Y$ two dense open subsets such that the map restricts to an isomorphism between $X$ and $Y$, i.e. outside a finite number of subvarieties of codimension at least one.

\begin{remark}Let $f :X\dashrightarrow Y$ be a birational map. There exists a \emph{unique maximal} pair of open subsets $U_{\max}\subset X$ and $ V_{\max}\subset Y$ such that $f$ restricts as an isomorphism.
	We can construct them as follows:
	\begin{align*}
		U_{\max}&=\{x \in X\setminus \Ind(f) \mid f(x)\in Y\setminus \Ind(f^{-1})\}\\
		V_{\max}&= \{y \in Y\setminus \Ind(f^{-1}) \mid f^{-1}(y)\in X\setminus \Ind(f)\}.
	\end{align*}
\end{remark}

\begin{remark}
	The set of birational maps from $X$ to $Y$ is in one to one correspondence with the set of field isomorphisms between the function fields $\kk(X)$ of $X$ and $\kk(Y)$ of $Y$.
\end{remark}

The \emph{exceptional locus} of a birational map $f\colon X \dashrightarrow Y $, denoted by $\Exc(f)$, is the closed subset of points of $X $ where $f$ is not a local isomorphism, i.e., the points that are not contained in any open set $U \subset X $ such that the restriction of $f$ to $U $ induces an isomorphism to the image.
More precisely, $\Exc(f)=X\setminus U_{\max}$ and  $\Exc(f^{-1})=Y\setminus V_{\max}$.

For instance, the blow-up of the point $[0:0:1]$ seen in Example \ref{ex_blow_up}, is a birational morphism and it induces an isomorphism between $\F_1\setminus E_{[0:0:1]} $ and $\PP^2\setminus \{[0:0:1]\}$.

If there exists a birational map between two varieties $X$ and $Y$, they are said to be \emph{birationally equivalent}. A variety birationally equivalent to $\PP^n$ is called a \emph{rational variety}.
Given a variety $X$, the set of birational transformations from $X$ to itself with the composition form a group, denoted by $\Bir(X)$. The \emph{Cremona group of rank $n$} is the group of birational transformations of the projective space $\PP^n$.

A subgroup $G\subset \Bir(X)$ is \emph{regularizable} if there exist a variety $Y$ and a birational map $\varphi:Y \dashrightarrow X$ such that $\varphi^{-1} G\varphi$ is a subgroup of the group of automorphisms of $Y$. If the variety $Y$ can be chosen projective, $G$ is called \emph{projectively regularizable}.

\subsection{Examples of Cremona transformations}\label{Subsection_Examples_Cremona}
Un element $f\in\Bir(\PP^n)$ of the Cremona group can be written as follows: 
\[\begin{array}{cccc}
	f : & \PP^n &\dashrightarrow & \PP^n\\
	& [x_0:x_1:\dots : x_n] & \dashmapsto &[f_0(x_0,\dots,x_n):f_1(x_0,\dots,x_n):\dots: f_n(x_0,\dots,x_n)]
\end{array}\]
where $f_i\in\kk[x_0,\dots,x_n]$ are homogeneous polynomials of the same degree without common factor; with an inverse of the same form. 
We call \emph{degree of  $f$}, the degree of the polynomials $f_i$. The indeterminacy locus of $f$
 is the common zero locus of the $f_i$'s $\Ind(f)=\bigcap_{i=0}^n \{f_i=0\}$. It is a subvariety of codimension greater or equal than $2$. 

 Remark that locally, for instance in the chart $U_n=\{x_n\neq 0\}$, $f$ can be written as quotients of polynomials \[(x_0,\dots ,x_{n-1}) \dashmapsto (\frac{f_0(x_0,\dots,x_{n-1},1)}{f_n(x_0,\dots,x_{n-1},1)},\dots, \frac{f_{n-1}(x_0,\dots,x_{n-1},1)}{f_n(x_0,\dots,x_{n-1},1)}).\]

In rank $1$, the Cremona group is isomorphic to $\PGL_2(\kk)$. In higher ranks they are far more complicated.

Here are a few examples of applications and subgroups contained in Cremona groups.

\begin{example}
	The \emph{group of automorphisms of $\PP^n$}, denoted by $\Aut(\PP^n_{\kk})$, consists in the birational transformations of degree $1$; this group is isomorphic to the projective linear group $\PGL_{n+1}(\kk)$. Writting the linear polynomials $f_i$ as follows gives the isomorphism: 
	\[f_i=\sum_{j=0}^{n}a_{i,j}x_j, \text{ for } 0\leq i\leq n \text{ and }  a_{i,j}\in \kk \leftrightsquigarrow \left(a_{i,j}\right)_{0\leq i,j\leq n}\in \PGL_{n+1}(\kk).\]
	Note that the indeterminacy locus and the exceptional locus of such maps are empty.
\end{example}

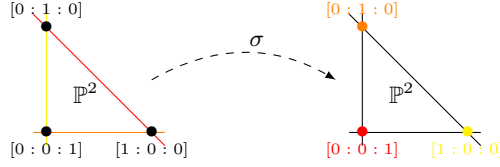
\begin{figure}
	\begin{center}
		\begin{tikzpicture}[scale=0.7]
			\draw[color=orange] (-0.25,0)--(2.25,0);
			\draw[color=yellow] (0,-0.25)--(0,2.25);
			\draw[color=red] (-0.25,2.25)--(2.25,-0.25);
			\draw (0,0) node{$\bullet$} node[below] {\tiny{$[0:0:1]$}};
			\draw (0,2) node {$\bullet$} node[above] {\tiny{$[0:1:0]$}};
			\draw (2,0) node {$\bullet$} node[below] {\tiny{$[1:0:0]$}};
			\draw (0.75,0.75) node{\small{$\PP^2$}};
			\draw[->,>=latex, dashed] (2,1) to[bend left] (5.5,1);
			\draw (4,1.75) node {\small{$\sigma$}};
			\begin{scope}[xshift=6cm]
				\draw (-0.25,0)--(2.25,0);
				\draw (0,-0.25)--(0,2.25);
				\draw (-0.25,2.25)--(2.25,-0.25);
				\draw[color=red] (0,0) node{$\bullet$} node[below] {\tiny{$[0:0:1]$}};
				\draw[color=orange] (0,2) node {$\bullet$} node[above] {\tiny{$[0:1:0]$}};
				\draw[color=yellow] (2,0) node {$\bullet$} node[below] {\tiny{$[1:0:0]$}};
				\draw (0.75,0.75) node{\small{$\PP^2$}};
			\end{scope}
		\end{tikzpicture}
		\caption{The standard quadratic plane Cremona involution \label{fig:sigma}}
	\end{center}
\end{figure}

\begin{example}\label{ex:standard} The \emph{standard quadratic Cremona involution} in $\Bir(\PP^2)$ (see Figure \ref{fig:sigma}):  \[\sigma: [x_0:x_1:x_2] \dashmapsto [x_1x_2:x_0x_2:x_0x_1]=[\tfrac{1}{x_0}:\tfrac{1}{x_1}:\tfrac{1}{x_2}].\]
	This birational transformation is of degree $2$ and contracts the lines $\{x_0=0\}$, $\{x_1=0\}$ and $\{x_2=0\}$ respectively to the points $[1:0:0]$, $[0:1:0]$ and $[0:0:1] $. Its exceptional locus is $\Exc(\sigma)=\cup_{i=0}^2\{x_i=0\}$ and its inderminacy locus is $\Ind(\sigma)=\{[0:0:1], [0:1:0],[1:0:0]\}$. Locally, in the chart $U_2=\{x_2\neq 0\}$ it can be written $(x_0,x_1)\dashmapsto (\frac{1}{x_0},\frac{1}{x_1})$. We easily see that it is an involution.
\end{example}

\begin{example}\label{ex:monomial}The \emph{subgroup of monomial transformations} is given by the following embedding of $\GL_n(\Z)$ into the Cremona group of rank $n$, written locally:
 \[(x_0,\dots,x_{n-1}) \dashmapsto (\prod_{j=0}^{n-1}x_j^{a_0,j},\dots, \prod_{j=0}^{n-1}x_j^{a_{n-1},j}) \] where $\left(a_{i,j}\right)_{0\leq i,j\leq n-1} \in \GL_n(\Z)$.
The standard quadratic Cremona involution seen in Example \ref{ex:standard} is a monomial transformation in the rank $2$ case. 
\end{example}

	\begin{example}\label{ex_automorphisms_affine}
	The \emph{group of polynomial automorphisms} $\Aut(\A^n_{\kk})$ consists in transformations having the following form and such that the inverse has also the following form:
	\[(x_0,\dots,x_{n-1}) \mapsto (\varphi_0,\dots,\varphi_{n-1})\text{  with }\varphi_i\in\kk[x_0,\dots,x_{n-1}].\]
	
	The degree of such transformations is the maximum of the degrees of the $\varphi_i$'s. When not empty, the exceptional locus is the hypersurface $\{x_n=0\}$.

	Already, when $n=2$ this group is huge and contains a non-abelian free group of rank at least $2$. 
	More precisely, consider $b: (x_0,x_1)\mapsto (-x_0+x_1^2,x_1)$, $a_\lambda : (x_0,x_1) \mapsto (\lambda x_0+x_1,x_0)$ where $\lambda\in\kk$, $a_\infty =\id$ and $g_\lambda = a_\lambda ba_\lambda^{-1}$. Then the subgroup generated by the involutions $g_\lambda$ is a free product of $\Z/2\Z$ parametrized by $\PP^1_{\kk} = \kk\cup \{\infty\}$ (see for instance \cite{Lamy_Lonjou}).
	\end{example}

Another important example of subgroup worth to mention in dimension $2$ is called the Jonquières group.

\begin{example}\label{ex_Jonquières}
The Jonquières group: $\PGL_2(\kk(x_1))\rtimes \PGL_2(\kk)$ is a subgroup of the Cremona group of rank $2$. 
Locally, such transformations are of the form: 
\[(x_0, x_1) \dashmapsto \left(  \frac{\alpha(x_1)x_0+\beta(x_1)}{\gamma(x_1)x_0+\delta(x_1)},\frac{ax_1+b}{cx_1+d} \right)\]
where $\begin{pmatrix}
\alpha & \beta\\
\gamma & \delta
\end{pmatrix}\in \PGL_2(\kk(x_1))$ and $\begin{pmatrix}
a &b\\
c &d
\end{pmatrix} \in \PGL_2(\kk)$.
These transformations preserve the pencil of lines passing through the point $[1:0:0]$. Note that the standard quadratic Cremona involution is also an example of a Jonquières transformation.
\end{example}

\subsection{Zariski theorem: an important tool in birational geometry of surfaces}
Blow-ups of points are fundamental examples of birational transformations of surfaces as any birational transformation between surfaces can be decomposed by a sequence of inverse of blow-ups followed by a sequence of blow-ups (see Theorem \ref{thm:Zariski}). Note that by convention, an isomorphism is the composition of $0$ blow-ups. This is not true anymore in higher dimension.

As blow-ups induce isomorphisms outside the exceptional divisors produced, 
it is useful to be able to identify points outside the exceptional divisors with their image. More precisely, the \emph{bubble space} of a given surface $S$ is the set $\mathcal{B}_S$ of triples $(t,T,\pi)$, where $\pi\colon T\to S$ is a birational morphism from a surface $T$ and $t$ is a point on $T$; two triples $(t,T,\pi)$ and $(t', T',\pi')$ are identified if $\pi^{-1}\pi'$ is a local isomorphism around~$t'$ mapping~$t'$ to~$t$. The points in $\mathcal{B}_S$ contained in~$S$ are called \emph{proper} points. Often in this survey, we will speak about points instead of elements of the bubble space through misuse of language.

\begin{theorem}[Zariski's theorem]\label{thm:Zariski}
	Let $f:S\dashrightarrow S'$ be a birational map between smooth projective surfaces. Then there exists a surface $W$ and compositions of blow-ups $\pi_1:W\rightarrow S$, $\pi_2:W \rightarrow S'$ such that $f=\pi_2\pi_1^{-1}$.
	\begin{center}
		\begin{tikzpicture}[baseline= (a).base]
			\node[scale=.75] (a) at (0,0){
				\begin{tikzcd}[ampersand replacement=\&]
					\& W\arrow{dl}[above left]{\pi_1} \arrow{dr}{\pi_2} \& \\
					S\arrow[dashrightarrow]{rr}{f} \& \& S'
			\end{tikzcd}};
		\end{tikzpicture}
	\end{center}
\end{theorem}
Note that in the above theorem, we allow $\pi_1$ and $\pi_2$ to be the composition of zero blow-up.
The proof of Zariski's theorem is made of two steps. The first one consists in blowing up the indeterminacy points of $f$ and repeat this process until we obtain a birational morphism $W\rightarrow S'$. The second step uses the  following well known result (a proof can be found for instance in \cite{Shafarevich_book_1} or \cite{Lamy_book}):

\begin{theorem}\label{thm_morphism_blowup}
Any birational morphism $\sigma : W \rightarrow S$ is a composition of finitely many blow-ups of points belonging to $\mathcal{B}_S$.
\end{theorem}

Let us do the first step of the proof on the example of the standard quadratic Cremona involution \ref{ex:standard}.

\begin{figure}[h!]
	\begin{center}
		\begin{tikzpicture}[scale=0.7]
			\begin{scope}[xshift=7cm, yshift=4cm]
				\draw[color=orange] (0.25,4)--(2.25,4);
				\draw[color=teal!80] (0.65,3.75)--(-0.1,5.25);
				\draw[color=yellow] (-0.1,4.75)--(0.65,6.25);
				\draw[color=red] (2,6.25)--(2.75,4.75);
				\draw[color=green!50] (0.25,6)--(2.25,6);
				\draw[color=blue!70] (2, 3.75)-- (2.75,5.25) ;		
				\draw[->,>=latex] (-0.75, 5) to (-1.75,4.5);
				\draw[color=blue!70] (-1, 4.5) node[below] {\small{$\pi_{[1:0:0]}$}};
				\draw[->,>=latex] (3.25, 5) to (4.25,4.5);
				\draw[color=orange] (3.75, 4.5) node[below] {\small{$\pi_{[0:1:0]}$}};
				\draw (1.25,5) node{\small{$S_{\sigma}$}};
			\end{scope}

			\begin{scope}[xshift=2.5cm, yshift=3.5cm]
				\draw[color=orange] (0.25,4)--(2.25,4);
				\draw[color=teal!80] (0.65,3.75)--(-0.1,5.25);
				\draw[color=yellow] (-0.1,4.75)--(0.65,6.25);
				\draw[color=red] (2,6.25)--(2,3.75);
				\draw[color=green!50] (0.25,6)--(2.25,6);
				\draw[color=blue!70] (2,4) node {$\bullet$} node[below] {\tiny{$[1:0:0]$}};		
				\draw[->,>=latex] (0.5,3.5) to (0,2.5);
				\draw[color=teal!80] (-1, 3.5) node[below right] {\small{$\pi_{[0:0:1]}$}};
			\end{scope}
			
			\begin{scope}[xshift=11.5cm, yshift=3.5cm]
				\draw[color=green!50] (0.25,4)--(2.25,4);
				\draw[color=red] (0.65,3.75)--(-0.1,5.25);
				\draw[color=blue!70] (-0.1,4.75)--(0.65,6.25);
				\draw[color=yellow] (2,6.25)--(2,3.75);
				\draw[color=teal!80] (0.25,6)--(2.25,6);
				\draw[color=orange] (0.55,6) node {$\bullet$} node[above] {\tiny{$[0:1:0]$}};		
				\draw[->,>=latex] (1.3,3.5) to (2,2.5);
				\draw[color=red] (1.5, 3) node[left] {\small{$\pi_{[0:0:1]}$}};
			\end{scope}
			
			\begin{scope}[xshift=14cm]
				\draw[color=green!50] (-0.25,4)--(2.25,4);
				\draw[color=blue!70] (0,3.75)--(0,6.25);
				\draw[color=yellow] (2,6.25)--(2,3.75);
				\draw[color=teal!80] (-0.25,6)--(2.25,6);
				\draw[color=red] (0,4) node{$\bullet$} node[below] {\tiny{$[0:0:1]$}};
				\draw[color=orange] (0,6) node {$\bullet$} node[above right] {\tiny{$[0:1:0]$}};		
				\draw[->,>=latex] (1.5,3) to (1.5,2);
				\draw[color=yellow] (1.5, 2.5) node[right] {\small{$\pi_{[1:0:0]}$}};
				\draw[color=green!50] (9.75,0)--(12.25,0);
			\end{scope}

			\draw[color=orange] (-0.25,4)--(2.25,4);
			\draw[color=yellow] (0,3.75)--(0,6.25);
			\draw[color=red] (2,6.25)--(2,3.75);
			\draw[color=green!50] (-0.25,6)--(2.25,6);
			\draw[color=teal!80] (0,4) node{$\bullet$} node[below] {\tiny{$[0:0:1]$}};
			\draw[color=blue!70] (2,4) node {$\bullet$} node[below] {\tiny{$[1:0:0]$}};		
			\draw[->,>=latex] (1.5,3) to (1.5,2);
			\draw[color=green!50] (1.5, 2.5) node[right] {\small{$\pi_{[0:1:0]}$}};

			\draw[color=orange] (-0.25,0)--(2.25,0);
			\draw[color=yellow] (0,-0.25)--(0,2.25);
			\draw[color=red] (-0.25,2.25)--(2.25,-0.25);
			\draw[color=red] (1.5, 1) node[rotate=-45] {\tiny{$\{z=0\}$}};
			\draw[color=yellow] (-0.25,1) node[rotate=90] {\tiny{$\{x=0\}$}};
				\draw[color=orange] (1,-0.25) node {\tiny{$\{y=0\}$}};
			\draw[color=teal!80] (0,0) node{$\bullet$} node[below left] {\tiny{$[0:0:1]$}};
			\draw[color=green!50] (0,2) node {$\bullet$} node[above] {\tiny{$[0:1:0]$}};
			\draw[color=blue!70] (2,0) node {$\bullet$} node[below right] {\tiny{$[1:0:0]$}};
			\draw (0.7,0.65) node{\small{$\PP^2$}};
			\draw[->,>=latex, dashed] (2,1) to (13.3,1);
			\draw (8,1.5) node {\small{$\sigma$}};
			
			\begin{scope}[xshift=14cm]
				\draw[color=green!50] (-0.25,0)--(2.25,0);
				\draw[color=blue!70] (0,-0.25)--(0,2.25);
				\draw[color=teal!80] (-0.25,2.25)--(2.25,-0.25);
				\draw[color=teal!80] (1.5, 1) node[rotate=-45] {\tiny{$\{z=0\}$}};
					\draw[color=blue!70] (-0.25,1) node[rotate=90] {\tiny{$\{x=0\}$}};
					\draw[color=green!50] (1,-0.25) node {\tiny{$\{y=0\}$}};
				\draw[color=red] (0,0) node{$\bullet$} node[below left] {\tiny{$[0:0:1]$}};
				\draw[color=orange] (0,2) node {$\bullet$} node[above] {\tiny{$[0:1:0]$}};
				\draw[color=yellow] (2,0) node {$\bullet$} node[below right] {\tiny{$[1:0:0]$}};
				\draw (0.7,0.65) node{\small{$\PP^2$}};
			\end{scope}
		\end{tikzpicture}
		\caption{Minimal resolution of $\sigma: (x_0,x_1) \dashmapsto (\frac{1}{x_0},\frac{1}{x_1})$. \label{fig:res_sigma}}
	\end{center}
\end{figure}
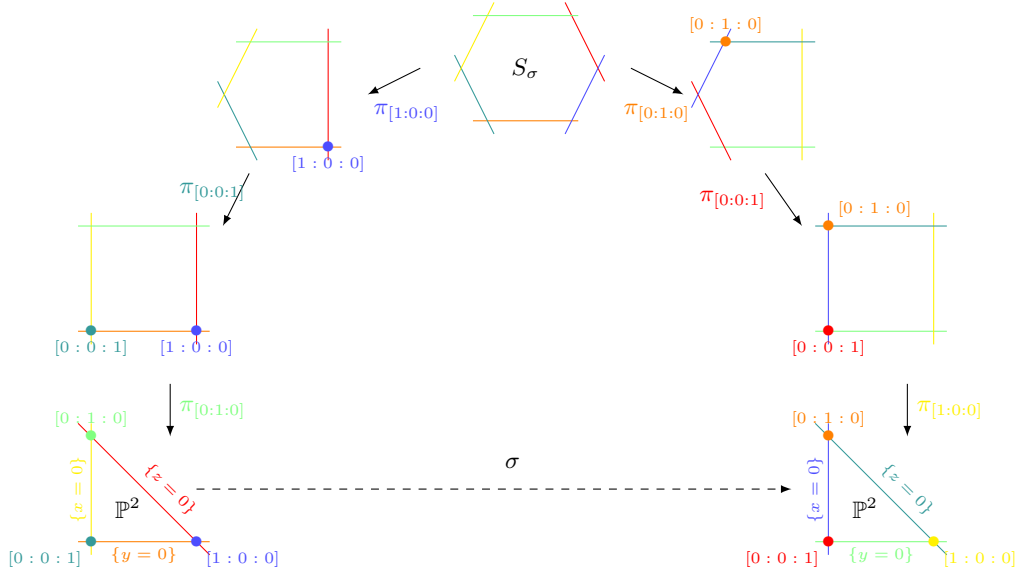

\begin{example}\label{Ex_sigma}
	Consider the standard quadratic involution seen in Example\ref{ex:standard}.
	Let us blow up the indeterminacy points until we get a morphism (see Figure~\ref{fig:res_sigma}). The birational transformation $\sigma$ has three indeterminacy points: $[0:0:1]$, $[0:1:0]$ and $[1:0:0]$. We blow-up the point $[0:1:0]$ then $\sigma\pi_{[0:1:0]}$ has two indeterminacy points: the points $[0:0:1]$ and $[1:0:0]$, using the abuse of notation mentioned above. We blow-up the point $[0:0:1]$, then $\sigma\pi_{[0:1:0]}\pi_{[0:0:1]}$ has a unique indeterminacy point $[1:0:0]$. We blow-up this point and we obtain that $\sigma\pi_{[0:1:0]}\pi_{[0:0:1]}\pi_{[1:0:0]}: S_{\sigma}\rightarrow \PP^2$ has no indeterminacy points and so it is a morphism. Note that the order of the different blow-ups does not matter as blow-ups are defined locally and the three points blown up belong to $\PP^2$. Take $\pi_1:=\pi_{[0:1:0]}\pi_{[0:0:1]}\pi_{[1:0:0]}$, $W=S_\sigma$ and $\pi_2:=\sigma\pi_{[0:1:0]}\pi_{[0:0:1]}\pi_{[1:0:0]}$.
	
	The second step of Zariski's theorem consists in proving that $\pi_2$ is a composition of blow-ups.
\end{example}

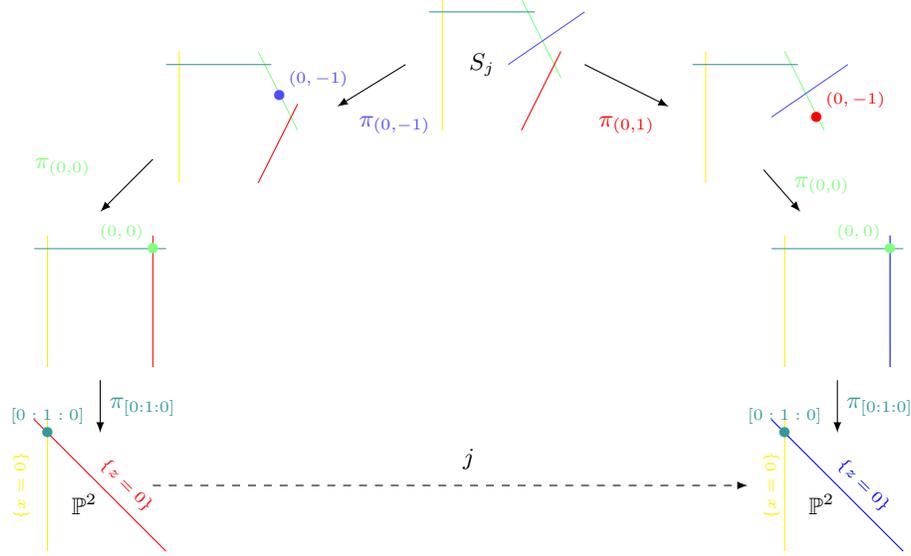
\begin{figure}
	\begin{center}
		\begin{tikzpicture}[scale=0.7]
			\begin{scope}[xshift=7cm, yshift=4cm]
				\draw[color=yellow] (0.5,3.75)--(0.5,6.25);
				\draw[color=green!50] (2,6.25)--(2.75,4.75);
				\draw[color=red] (2, 3.75)-- (2.75,5.25) ;		
				\draw[color=teal!80] (0.25,6)--(2.25,6);
				\draw[->,>=latex] (-0.2, 5) to (-1.5,4.2);
				\draw[color=blue!70] (1.75, 5)-- (3.2,6) ;		
				\draw[color=blue!70] (-0.4, 4.2) node[below] {\small{$\pi_{(0,-1)}$}};
				\draw (1.25,5) node{\small{$S_{j}$}};
				\draw[->,>=latex] (3.2, 5) to (4.8,4.2);
				\draw[color=red] (4, 4.2) node[below] {\small{$\pi_{(0,1)}$}};
			\end{scope}

			\begin{scope}[xshift=2cm, yshift=3cm]
				\draw[color=yellow] (0.5,3.75)--(0.5,6.25);
				\draw[color=green!50] (2,6.25)--(2.75,4.75);
				\draw[color=red] (2, 3.75)-- (2.75,5.25) ;		
				\draw[color=teal!80] (0.25,6)--(2.25,6);
				\draw[->,>=latex] (0, 4.2) to (-1,3.2);
				\draw[color=green!50] (-1, 3.7) node[above left] {\small{$\pi_{(0,0)}$}};
				\draw[color=blue!70] (2.4,5.4) node {$\bullet$} node[above right] {\tiny{$(0,-1)$}};		
			\end{scope}

			\draw[color=yellow] (0,3.25)--(0,5.75);
			\draw[color=red] (2,5.75)--(2,3.25);
			\draw[color=teal!80] (-0.25,5.5)--(2.25,5.5);
			\draw[color=green!50] (2,5.5) node{$\bullet$} node[above left] {\tiny{$(0,0)$}};
			\draw[->,>=latex] (1,3) to (1,2);
			\draw[color=teal!80] (1, 2.5) node[right] {\small{$\pi_{[0:1:0]}$}};

			\draw[color=yellow] (0,-0.25)--(0,2.25);
			\draw[color=yellow] (-0.5,1) node[rotate=90] {\tiny{$\{x=0\}$}};
			\draw[color=red] (-0.25,2.25)--(2.25,-0.25);
			\draw[color=red] (1.5, 1) node[rotate=-45] {\tiny{$\{z=0\}$}};
			\draw[color=teal!80] (0,2) node {$\bullet$} node[above] {\tiny{$[0:1:0]$}};
			\draw (0.7,0.65) node{\small{$\PP^2$}};
			\draw[->,>=latex, dashed] (2,1) to (13.3,1);
			\draw (8,1.5) node {\small{$j$}};

			\begin{scope}[xshift=14cm]
				\draw[color=yellow] (0,-0.25)--(0,2.25);
				\draw[color=yellow] (-0.25,1) node[rotate=90] {\tiny{$\{x=0\}$}};
				\draw[color=blue] (-0.25,2.25)--(2.25,-0.25);
				\draw[color=blue] (1.5, 1) node[rotate=-45] {\tiny{$\{z=0\}$}};
				\draw[color=teal!80]  (0,2) node {$\bullet$} node[above] {\tiny{$[0:1:0]$}};
				\draw (0.7,0.65) node{\small{$\PP^2$}};
			\end{scope}

			\begin{scope}[xshift=12cm, yshift=3cm]
				\draw[color=yellow] (0.5,3.75)--(0.5,6.25);
				\draw[color=green!50] (2,6.25)--(2.75,4.75);
				\draw[color=blue!70] (1.75, 5)-- (3.2,6) ;		
				\draw[color=teal!80] (0.25,6)--(2.25,6);
				\draw[->,>=latex] (1.6, 4) to (2.3,3.2);
				\draw[color=green!50] (2, 3.7) node[right] {\small{$\pi_{(0,0)}$}};
				\draw[color=red] (2.6,5) node {$\bullet$} node[above right] {\tiny{$(0,-1)$}};		
			\end{scope}

			\begin{scope}[xshift=14cm]
				\draw[color=yellow] (0,3.25)--(0,5.75);
				\draw[color=blue] (2,5.75)--(2,3.25);
				\draw[color=teal!80] (-0.25,5.5)--(2.25,5.5);
				\draw[color=green!50] (2,5.5) node{$\bullet$} node[above left] {\tiny{$(0,0)$}};
				\draw[->,>=latex] (1,3) to (1,2);
				\draw[color=teal!80] (1, 2.5) node[right] {\small{$\pi_{[0:1:0]}$}};
			\end{scope}
		\end{tikzpicture}
		\caption{Minimal resolution of $j : (x_0,x_1) \mapsto (x_0, x_0^2+x_1)$.\label{fig_resol_j}}
	\end{center}
\end{figure}

\begin{example}\label{Ex_j}
Consider the following Cremona transformation: \[j:=(x_0,x_1)\mapsto (x_0,x_0^2+x_1). \]Note that it is both a polynomial automorphism of $\A^2$ and a Jonquières transformation. 
Its unique indeterminacy point is $\{[0:1:0]\}$.
The unique indeterminacy point of the composition $j\pi_{[0:1:0]}$ belongs to 
the exceptional divisor $E_{[0:1:0]}$, and, in the chart $\{x_1=1\}\times \{t_0=1\}$, it is the point $(0,0)$.
Again, $j\pi_{[0:1:0]}\pi_{(0,0)}$ has a unique indeterminacy point, which belongs to the exceptional divisor $E_{(0,0)}$ and, in the chart $\{t_0=1\}$, it is the point $(0,-1)$. Blowing up this last point, we obtain that $\pi_1:=j\pi_{[0:1:0]}\pi_{(0,0)}\pi_{(0,-1)}S_j \rightarrow \PP^2$ is a morphism.
 The resolution of $j$ is given by Figure \ref{fig_resol_j}. Note that, contrary to the case of the standard quadratic Cremona involution, in the  case of $j$, the blow-ups composing $\pi_1$ do not commute.
\end{example}

\begin{remark}
In the Zariski theorem, $W$ can be chosen minimal and it is called the \emph{minimal resolution} of $f$.
	\begin{center}
		\begin{tikzpicture}[baseline= (a).base]
			\node[scale=.65] (a) at (0,0){
				\begin{tikzcd}[ampersand replacement=\&]
					\& W'\arrow[red]{d} \arrow[bend right]{ddl}[above left]{\pi_1'} \arrow[bend left]{ddr}[above right]{\pi_2'}\& \\
					\& W\arrow{dl}[above left]{\pi_1} \arrow{dr}{\pi_2} \& \\
					S\arrow[dashrightarrow]{rr}{f} \& \& S'
			\end{tikzcd}};
		\end{tikzpicture}
	\end{center}
\end{remark}

\begin{definition}
	The points of $\mathcal{B}_S$ blown up by $\pi_1$ in the minimal resolution of $f$ are called the \emph{base-points} of $f$, and they are denoted by $\Bs(f)$.
\end{definition}

\begin{remark}\label{rmk_basepoints_infinitely_near}
	Note that for any $f\in \Bir(S)$, $\Ind(f)\subset \Bs(f)$. In the case of the standard quadratic plane involution, it is an equality (see Example \ref{Ex_sigma}), while it is a strict inclusion for the birational transformation $j:(x,y)\mapsto (x,x^2+y)$ of Example~\ref{Ex_j}.
\end{remark}

\begin{remark}\label{rmk:bs_pt_f_invf}
	The points blown up by $\pi_2$ in  Theorem \ref{thm:Zariski} are the base points of $f^{-1}$. Hence if $S'=S$ then $\lvert \Bs(f)\rvert =\lvert\Bs(f^{-1})\rvert$.
\end{remark}

Over any field, Zariski theorem \ref{thm:Zariski} still holds by replacing points with closed points (in the context of the scheme theory), see for instance \cite[\href{https://stacks.math.columbia.edu/tag/0C5Q}{Tag 0C5Q}]{Stack_project}. Proposition~\ref{prop_catalog_blowup_cc} needs to be slightly adapted to this context. 
 Note that by \cite{Lin_Shinder_Zimmermann}, over perfect fields, the birational morphisms $\pi_1$ and $\pi_2$ of Theorem \ref{thm:Zariski}
blow up the same number of closed points. It is unknown if the situation is similar over non-perfect fields.

\begin{remark}\label{rmk_bspt_subadditive}
	Let $f:S\dashrightarrow S'$ and $g: S' \dashrightarrow S''$, then, we have \[\lvert \Bs(gf)\rvert \leq \lvert\Bs(f)\rvert+\lvert\Bs(g)\rvert.\] 
	Let $W$ and $W'$ be the respective minimal resolution of $f$ and $g$, and $Z$ the one of $\sigma_1^{-1}\pi_2$:
		\begin{center}
		\begin{tikzpicture}[baseline= (a).base]
			\node[scale=.65] (a) at (0,0){
				\begin{tikzcd}[ampersand replacement=\&]
					\&\& Z\arrow{dl}[above left]{\sigma_1'} \arrow{dr}{\pi_2'}  \&\& \\
					\& W\arrow{dl}[above left]{\pi_1} \arrow{dr}{\pi_2} \& \& W'  \arrow{dl}[above left]{\sigma_1} \arrow{dr}{\sigma_2} \\
					S\arrow[dashrightarrow]{rr}{f} \& \& S'\arrow[dashrightarrow]{rr}{g} \&\& S''
			\end{tikzcd}};
		\end{tikzpicture}
	\end{center}
Then $Z$ is a resolution of $gf$ but not always the minimal one. For instance, when $\Bs(f^{-1})\cap \Bs(g)\neq \emptyset$, $Z$ is not the minimal one. Notice that it is not the only possibility see for instance \cite[Example 4.16]{Lamy_book}.
\end{remark}

\subsection{State of the art from a geometric group theoretic point of view}\label{Subsection_State_art}
The Cremona group can be studied from many different points of view making it very interesting: algebraic, dynamical, topological, geometrical... Nevertheless in this small state of the art, we will be mainly interested in its geometric group theoretic features in order to explain the context in which actions on median graphs have been constructed. It is not exhaustive at all. There exists several surveys and books in the literature about Cremona groups (see for instance \cite{Cantat_review}, \cite{Deserti_subgroup_Cremona}, \cite{Lamy_book} \cite{Zimmermann_survey_normal_subgroups}).

Cremona introduced the Cremona group of rank $2$ in 1863-1865 \cite{Cremona}. 
The first theorem on the structure of this group has been partially\footnote{Indeed, the proof contained a mistake} proved by Noether and then completely by Castelnuovo. They gave a nice generating set. 

\begin{theorem}[{\cite{Noether}, \cite{Castelnuovo}}]
	Over an algebraically closed field $\kk$, the Cremona group of rank $2$ is generated by the automorphisms group of $\PP^2$ and the standard quadratic Cremona involution
	$	\sigma \colon
	(x,y)  \mapsto (\frac{1}{x},\frac{1}{y})$.
\end{theorem}
Note that even when the field is finite the Cremona group is never finitely generated (see for instance \cite[Proposition 3.6]{Cantat_review}); if it would be the case you could construct a finitely generated field that is algebraically closed, which is impossible. Another generating set, which is often used, is the set of  automorphisms of $\PP^2$ and the de Jonquières maps.

A breakthrough in the study of the Cremona group of rank $2$ was done in the years 2010 by constructing an action of the Cremona group of rank $2$ on a hyperbolic space (see \cite{Cantat_groupes_birat} and \cite{Cantat-Lamy}). 
This space is an analogous in infinite dimension of the hyperboloid model of the hyperbolic plane (see Figure \ref{fig_hyperboloid}): \[\Hh^2=\{(x,y,z)\in\R^3\mid x^2-y^2-z^2=1\text{ and }x>0\}.\]

\begin{figure}
	\begin{center}
	\def\svgwidth{50pt}
	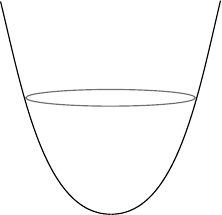
\end{center}
\caption{Hyperboloid model of the hyperbolic plane. \label{fig_hyperboloid}}
\end{figure}
Let us briefly give an idea of this construction, even if it is not an object of interest in this survey (for more details see for instance \cite{Cantat-Lamy}, \cite{Lonjou_non_simplicity} or \cite{Lamy_book}).
Consider a smooth projective surface $S$.
A \emph{divisor} $D$ on $S$, is a finite formal sum of irreducible curves of $S$ with integer coefficients: $D=\sum_{i=0}^{n}a_iC_i$ where $a_i\in \Z$, $C_i$ irreducible curve of $S$.
We denote by $\Div(S)$ the group of divisors on $S$. There exists a well defined intersection theory on divisors (see for instance \cite{Fulton}). Two divisors are said \emph{numerically equivalent} if they have the same intersection number with any irreducible curve of $S$.
The \emph{Néron-Severi group} of $S$, denoted by $\NS(S)$, is the group of divisors up to numerical equivalence, tensorised by $\R$. It is a finitely generated free abelian group (see for instance \cite{Shafarevich_book_1}). Its rank, called the \emph{Picard rank of the surface} $S$, is denoted by $\rho(S)$. For instance $\NS(\PP^2)\simeq \R$, is generated by the class of the line.
Consider a birational morphism $\sigma : S' \rightarrow S$. Let $D\in\Div(S)$ given by the local equations $\{g_\alpha\}$ on $S$. We define the \emph{pull-back} of $D$ by $\sigma$ as the divisor on $S'$, denoted by $\sigma^*(D)$, defined locally by the equations $\{g_\alpha \sigma\}$ on $S'$. By Theorem \ref{thm_morphism_blowup}, $\sigma$ is the composition of blow-ups of points in $\mathcal{B}_S$. Let us denote them by $p_1,\dots, p_n$. Then
\begin{equation}\label{eq_NS}
\NS(S')=\sigma^*(\NS(S))\oplus \left(\underset{1\leq i\leq n}{\oplus}\R\sigma^*(E_i)\right).
\end{equation} 
For instance, if you blow-up two points $p,q\in \PP^2$, then the Néron-Severi group of the surface obtained is isomorphic to $\R^3$ and it is generated by the class of the line and by the classes of the exceptional divisors $E_p $ and $E_q$.
Now, we consider the \emph{Picard-Manin space}, $\mathcal{Z}(S)$, which is the $\ell^2$-completion of the inductive limit of the Néron-Severi groups of the surfaces dominating $S$:\[\mathcal{Z}(S):=\{d + \sum_{p\in \mathcal{B}_S}\lambda_pe_p\mid \lambda_p\in \R, \  d\in \NS(S), \ \sum_{p\in \mathcal{B}_S}\lambda_p^2<\infty\},\] where $e_p$ denotes the class in the inductive limit of the exceptional divisor obtained by blowing up the point $p\in \mathcal{B}_S$. More precisely, denote by $(p,T,\pi)$ a representative of the point $p$, and $\pi_p:T'\rightarrow T$ the blow-up of the point $p$, then in each surface $W$ dominating $T'$: $\rho:W \rightarrow T'$, $e_p$ represents the divisor $\rho^*(E_p)$. The intersection form on divisors induces a well defined intersection form on classes of the Picard-Manin space, which is a bilinear form of signature $(1,\infty)$. Let $a\in \NS(S)$ be any ample class, i.e, the intersection number between $a$ and the class of any irreducible curve of $S$ is positive and the self-intersection of $a$ is positive. Taking the classes of the Picard-Manin space of self-intersection $1$ that intersect $a$ positively gives a hyperbloid as announced:
\[\mathbb{H}^\infty(S):=\{c\in \mathcal{Z}(S)\mid c\cdot c=1 \text{ and } c\cdot a >0\}.\]
Note that $\mathbb{H}^\infty (S)$ does not depend on the choice of $a$.
Equipped with the distance $\dist(c_1,c_2):=\cosh^{-1}(c_1\cdot c_2)$, it is a complete CAT(0) metric space with constant curvature $-1$, which is also Gromov-hyperbolic. 

As a consequence of \eqref{eq_NS}, a birational morphism $\sigma : S'\rightarrow S$ induces a canonical bijection of the Picard-Manin spaces $\mathcal{Z}(S)$ and $\mathcal{Z}(S')$ preserving the intersection form; and so it induces an isometry between the hyperbolic spaces $\mathbb{H}^\infty(S)$ and $\mathbb{H}^\infty(S')$ denoted by $\sigma_{\#}$.
As a consequence, the Cremona group acts on $\mathbb{H}^\infty(\PP^2)$ as follows: consider the minimal resolution of $f=\pi_2\pi_1^{-1}$ given by the Zariski theorem, then $f_\#=(\pi_2)_{\#}(\pi_1^{-1})_{\#}$.

The isometries of the action of the Cremona group on this hyperbolic space are very well understood. Indeed, let $f\in \Bir(\PP^2)$ and consider the sequence of degrees of the iterates of $f$: $\{\deg(f^n)\mid n\in\N\}$. The isometry $f$ is respectively elliptic, parabolic and loxodromic if and only if this sequence is respectively bounded, respectively grows linearly or quadratically, and respectively grows exponentially (\cite{Cantat_groupes_birat},\cite{Diller_Favre_Dynamics_bimeromorphic}).

Using this action, two important theorems about the structure of the Cremona group of rank $2$ have been proven making it having some features of both hyperbolic and linear groups.

It was a long standing open question (since Enriques) to know whether the Cremona group was a simple group or not. It was first solved by Cantat and Lamy over algebraically closed field, then by the author over any field.

\begin{theorem}[{\cite{Cantat-Lamy}, \cite{Lonjou_non_simplicity}, \cite{Shepherd_Barron}}]
	The Cremona group of rank $2$ is not a simple group over any field.
\end{theorem}

Both proofs use small cancellation theory. The article of Cantat-Lamy is in two part. They first elaborate a criterion, called ``tight'', ensuring that, in a group acting by isometries on a Gromov hyperbolic space, the normal subgroup generated by some power of a tight element is a proper normal subgroup. In the second part of their article they find such elements in the Cremona group of rank $2$ over algebraically closed fields. 

In \cite{Lonjou_non_simplicity}, the author exhibits a family of ``Weakly Properly Discontinuous elements'' in the Cremona group of rank $2$ (over any field) and uses the machinery of \cite{DGO} ensuring that groups containing WPD elements have proper normal subgroups. Moreover, this does not show only that the Cremona group of rank $2$ is not simple, but also that it has many quotients (it has the SQ universal property) and that it is part of the family of acylindrically hyperbolic groups.

Note that the notions tight and WPD are related to the small cancellation theory. Even if they have some similitude, these notions are not equivalent; even more, none of these two notions implies the other one (see for instance \cite[Examples 5.10 and 5.11]{Lamy_Lonjou}). Indeed, if an element $g$ of a group is tight (when acting on a geodesic metric space $X$) then the stabilizer of its axis has to be the normalizer of the subgroup generated by $g$; but there is no finiteness condition on this normalizer. For instance, it could exist infinitely many elements fixing pointwise the axis of $g$. This is not allowed for WPD elements. On the other hand, the stabilizer of the axis of a WPD element $f$ does not have to normalize the subgroup generated by $f$.

In \cite{Shepherd_Barron}, Shepherd-Barron proved also the non-simplicity of the Cremona group of rank $2$ for various fields using the tight criterion of Cantat-Lamy. He also classifies the elements in the group having this property.
In \cite{Cantat_Guirardel_Lonjou}, all elements of infinite order such that, up to some power, the normal subgroup generated by them is a proper subgroup, are classified; note that they are not all loxodromic.

The second important result is about the classification of the subgroups of the Cremona groups.
We say that a group $G$ satisfies the \emph{Tits alternative} (respectively \emph{the Tits alternative for finitely generated subgroups}) if any subgroup (respectively any finitely generated subgroup) contains either a non-abelian free subgroup, or a solvable group of finite index. This alternative has been proved by Tits for linear groups (up to restricting to finitely generated subgroup when the field is of positive characteristic). Note that over the algebraic closure of a finite field, the linear group (of any dimension) does not contain any non-abelian free subgroup (any element is of finite order) and it does not contain any solvable subgroup of finite index. As a consequence the above restriction, when working over fields of positive characteristic is necessary.
\begin{theorem}[{\cite{Cantat_groupes_birat}, \cite{Urech_Tits}, \cite{Lamy_book}}]
	The Cremona group of rank $2$ satisfies the Tits alternative if the characteristic of the field is $0$ and the Tits alternative for finitely generated subgroups otherwise.
\end{theorem}
Cantat proved that the Cremona group of rank $2$ over a field of characteristic $0$ satisfies the Tits alternative for finitely generated subgroups. Then it has been improved by Urech for any subgroups. Recently, Lamy and Urech proved that in positive characteristic the Cremona group of rank $2$ satisfies the Tits alternative for finitely generated subgroups (see \cite{Lamy_book} for more details). Note that the Cremona group of rank $2$ over a field of positive characteristic $\kk$ does not satisfy the Tits alternative (for any subgroup) as it contains some linear groups over $\kk$.

Even if the Cremona group of rank $2$ is now quite well understood, some questions remain open, like for instance Question \ref{question_loc_elliptic}. The recent constructions on median graphs \cite{GLU_Neretin} allowed to answer this question positively for Cremona groups of rank $2$ over finite fields (see Subsection \ref{subsubsection_open_reg}) but it is still open over arbitrary fields.

Cremona groups are not linear over $\C$ (see \cite{Cornulier_linear}). A natural question asked by Cantat is the following: are finitely generated subgroups of the Cremona group of rank $2$ residually finite? Recall that a group $G$ is residually finite if for any non-identity element $g\in G$ there exists a group homomorphism from the group to a finite group that does not send $g$ to the identity element.

In higher dimension, the situation is drastically different. Cremona groups are far more complicated and less understood. For instance no nice family of generators is known. Even worse, we do know that it can not be like in dimension 2.
\begin{theorem}[{\cite{Pan_generateurs}}]
For $n\geq 3$, $\Bir(\PP^n_{\C})$ is not generated by $\Aut(\PP^n)$ and finitely many elements, or more generally by any set of elements of $\Bir(\PP^n)$ of bounded degree.
\end{theorem}

Some attempts to generalize the construction of the hyperbolic space for Cremona groups in higher ranks have been done (see for instance \cite{Dang_Favre}) but it seems not possible to obtain a geometric space with non positive curvature.

The median graphs constructed in \cite{Lonjou_Urech_cube_complexe} and that we will see in Section \ref{Section_median_graphs_higher_ranks} are the first geometric constructions allowing to study Cremona groups of higher ranks from a geometric group theoretic point of view.

Several recent progresses have been done in the understanding of Cremona groups of higher ranks, like for instance, the non-simplicity of the Cremona group in higher rank over any subfield of $\C$.
\begin{theorem}[{\cite{Blanc_Lamy_Zimmermann_quotients}}]Let $\kk$ be a subfield of $\C$ and $n\geq 3$, then $\Bir_{\kk}(\PP^n)$ is not a simple group.
\end{theorem}
It has been obtained using advanced results of algebraic geometry, and using the Minimal Model Program and the factorization into Sarkisov links.
They also proved that in higher dimension for subfields of $\CC$, the Cremona group can not be generated by automorphisms of $\PP^n$ and higher dimensional analogue of de Jonquières transformations.

An exciting progress would be to know whether Cremona groups of higher ranks satisfy the Tits alternative.

\section{Median graphs for groups of birational transformations of surfaces}\label{Section_Median_graphs_surfaces}
In this section, we focus on surfaces, that will be for us (except said otherwise), irreducible smooth projective varieties of dimension $2$ over an algebraically closed field. We refer to the original article \cite{Lonjou_Urech_cube_complexe} for a larger scope: projective regular surfaces $S$ over arbitrary fields $\kk$. 

We will present and study two constructions of actions of groups of birationl transformations of surfaces on median graphs (the blow-up graphs in Subsection~\ref{Subsection_blow-up graph} and the rational blow-up graph in Subsection~\ref{Subsection_rational_blow_up graph}).

\subsection{Action of the Cremona group of rank 2 on the blow-up graph}\label{Subsection_blow-up graph}
We first introduce the construction of the blow-up graph and prove that it is a median graph. There is a natural action of Cremona groups of rank $2$ on it. We present a nice dictionary between the median objects (hyperplanes, distance, minimizing set etc) and birational notions and survey which results can be deduced from this action.

\medskip
\subsubsection{Construction of the blow-up graph}
This graph and its cube completion have been constructed and studied in \cite{Lonjou_Urech_cube_complexe} in terms of CAT(0) cube complexes.

Let $S$ be a smooth projective surface. 
A \emph{marked surface} $(T ,\varphi)$ is a pair where $T $ is a smooth projective surface over $\kk$ and $\varphi\colon T \dashrightarrow S $ is a birational map. Two marked surfaces $(T ,\varphi)$ and $(T ',\varphi')$ are equivalent if the map $\varphi'^{-1}\varphi\colon T \to T '$
\begin{center}
			\begin{tikzcd}[ampersand replacement=\&]
				T\arrow[dashrightarrow]{dr}[below left]{\varphi} \arrow{rr}[above]{\sim} \&\& T'\arrow[dashrightarrow]{dl}[below right]{\varphi'}  \\
				\& S\&
		\end{tikzcd}
\end{center}
is an isomorphism. Such a class, will be denoted by $[(T ,\varphi)]$. For instance, consider an automorphism $a$ of $S$ then the marked surfaces $(S,\id)$ and $(S, a)$ are equivalent.

\begin{definition}[\cite{Lonjou_Urech_cube_complexe}]The \emph{blow-up graph} associated to $S$ and denoted by $\Cb(S)$ is the graph whose vertices are equivalent classes of marked surfaces $[(T,\varphi)]$, and where two vertices $[(T_1,\varphi_1)]$ and $[(T_2,\varphi_2)]$ are connected by an edge if $\varphi_2^{-1}\varphi_1$ is the blow-up of a point of $T_2$ or the inverse of the blow-up of a point of $T_1$. 
\end{definition}

We can define a \emph{height} $h$ on the set of vertices as the Picard rank of the surface given by a representative:
$h([(T,\varphi)]):=\rho(T)$. It is well defined as isomorphic surfaces have the same Picard rank. 
Note that in the case $S=\PP^2$, $h([(T,\varphi)])=\lvert \Bs(\varphi^{-1}) \rvert - \lvert\Bs(\varphi) \rvert +1$.
As the height does not depend on the marking but only on the surface, we will sometimes do a slight abuse of notation denoting by $h(T)$ the height of a vertex.

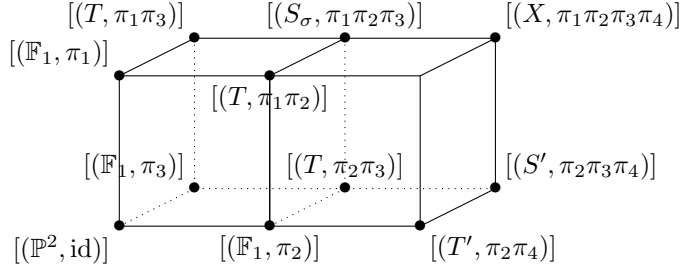
\begin{figure}
	\begin{tikzpicture}
	\draw (0,0) -- (2,0) -- (2,2) -- (0,2)-- (0,0);
	\draw (2,0) -- (4,0) -- (4,2) -- (2,2)-- (2,0);
	\draw (3,2.5) -- (1,2.5);
	\draw[dotted]  (1,2.5)--(1,0.5) -- (3,0.5);
	\draw[dotted]  (3,2.5)--(3,0.5) -- (5,0.5);
	\draw[dotted] (2,0) -- (3,0.5); 
	\draw (4,0) -- (5,0.5)--(5,2.5)--(3,2.5); 
	\draw (2,2) -- (3,2.5);
	\draw (0,2) -- (1,2.5);
	\draw (4,2)--(5,2.5);
	\draw[dotted] (0,0) -- (1,0.5);
	\draw (0,0) node {$\bullet$} node[below left] {$[(\PP^2,\id)]$};
	\draw (2,0) node {$\bullet$} node[below] {$[(\F_1,\pi_{2})]$} ;
	\draw(2,2) node {$\bullet$} node[below] {$[(T,\pi_{1}\pi_{2})]$} ;
	\draw (0,2) node {$\bullet$} node[above left] {$[(\F_1,\pi_{1})]$};
	\draw (1,0.5) node {$\bullet$} node[above left] {$[(\F_1,\pi_{3})]$};
	\draw (3,0.5) node {$\bullet$} node[above] {$[(T,\pi_{2}\pi_{3})]$} ;
	\draw(3,2.5) node {$\bullet$} node[above] {$[(S_\sigma,\pi_{1}\pi_{2}\pi_{3})]$};
	\draw (1,2.5) node {$\bullet$} node[above left] {$[(T,\pi_{1}\pi_{3})]$} ;
	\draw (4,0) node {$\bullet$} node[below right] {$[(T',\pi_{2}\pi_{4})]$} ;
	\draw (5,2.5) node {$\bullet$} node[above right] {$[(X,\pi_{1}\pi_2\pi_{3}\pi_4)]$} ;
	\draw (5,0.5) node {$\bullet$} node[above right] {$[(S',\pi_2\pi_{3}\pi_4)]$} ;
\end{tikzpicture}
\caption{The subgraph of $\Cb(\PP^2)$ generated by the blow-up of three points $p_1,p_2,p_3$ of $\PP^2$ and by the blow-up of $p_4\in E_{p_2}$. \label{fig:example_cube}}
\end{figure}

	\begin{example}\label{ex_cubes}
			Consider $p_1,p_2,p_3$ three distinct points of $\PP^2$. For $1\leq i\leq 3$, denote by $\pi_i$ the blow-up map of the point $p_i$ and by $E_i$ the exceptional divisor obtained.  
			Let $p_4$ be a point on $E_2$. We denote by $\pi_4$ the blow-up map of $p_4$ and by $E_4$ the exceptional divisor obtained. Let us denote respectively by $T$ and $S_\sigma$ the surfaces (up to isomorphism) obtained from $\PP^2$ by blowing up respectively two and three points in $\PP^2$, and by $T'$ and $S'$ the surfaces obtained from respectively $T$ and $S$ by blowing up a point on one of the exceptional divisors.

	As the blow-ups of two distinct points of the same surface commute, the different possible orders to blow-up the points $p_1,p_2,p_3$ span a cube of dimension $3$. Blowing up the fourth points $p_4$ does not span a cube of dimension $4$ because the points $p_2$ and $p_4$ can not belong to a common surface (see Figure \ref{fig:example_cube}).
	\end{example}

As illustrated in Example \ref{ex_cubes}, blowing up a family of $n$ points in a surface $W$ generates a cube of dimension $n$. And indeed, it is the only way to have cubes. More precisely:

\begin{fact}\label{cc_blowup_cubes}
In the blow-up graph, $n$ distinct vertices $[(T_1,\varphi_1)],\dots, [(T_{2^n},\varphi_{2^n})]$ span a cube of dimension $n$ if and only if there exists $1\leq r\leq 2^n$ such that for any $1\leq j\leq 2^n$: 
\begin{itemize}
	\item there exists $n$ distinct points $p_1,\dots, p_n$ in $T_r$,
	\item $\varphi_{r}^{-1}\varphi_{j}: S_j\rightarrow T_r$ is the blow-up of a subfamily of the set $\{p_1,\dots, p_n\}$.
\end{itemize}
\end{fact}

		The surface $S$ contains infinitely many points, so blowing up an arbitrary large sequence of points, we immediately see the following properties of this graph.
	\begin{proposition}[\cite{Lonjou_Urech_cube_complexe}]
		The blow-up graph is: 
		\begin{itemize}
			\item not locally compact,
			\item infinite dimensional with infinite cubes,
			 \item cubically oriented.
		\end{itemize}
	\end{proposition}
We fix once and for all the following cubical orientation: from $[(T_1, \varphi_1)]$ to $[(T_2,\varphi_2)]$ if $\rho(T_2)=\rho(T_1)+1$.

\begin{theorem}[{\cite{Lonjou_Urech_cube_complexe}}]
	The blow-up graph is a median graph.
\end{theorem}

\begin{proof}
	Using Theorem \ref{thm:MedianVsCC} and Theorem \ref{thm_ccCAT((0))} it remains to prove that the cube completion is simply connected and that links of vertices are flag simplical complexes.
	
	 \textbf{Simply-connectness.}
		First, it is connected by Zariski theorem \ref{thm:Zariski}: consider two vertices represented by $(T_1 ,\varphi_1)$ and $(T_2,\varphi_2)$, decomposing $\varphi_2^{-1} \varphi_1$ into a sequence of blow-ups followed by a sequence of contractions, yields a path between these two vertices.

	Let now $\gamma$ be a cycle in $\Cb(S )$ passing through the vertices $v_1,\dots, v_n$ represented respectively by $(T_{1},\varphi_1),\dots, (T_{n}, \varphi_n)$, which can be assumed to be without backtrack.
	By repeated applications of Zariski theorem \ref{thm:Zariski}, there exists a marked surface $(W,\psi)$ dominating all $T_i$s, meaning that $\varphi_i^{-1}\psi: W \rightarrow T_i$ is a
	morphism for all $i$, and in particular can be decomposed as compositions of blow-up of points. Hence, the vertex $w=[(W,\psi)]$ dominates all the vertices $v_i$, i.e., for any $1\leq i\leq n$, there exists a sequence of edges connecting $w=[(W,\psi)]$ to $v_i$, all oriented in the same way from $w$ to $v_i$. 

		The goal now is to show that $\gamma$ is freely homotopic to $w$. For this, we define $h_{\text{min}}(\gamma)\coloneqq\min_{1\leq i\leq n}(h(v_{i}))$, where $h(v_i)$ is the height of $v_i$, and $1\leq i_0\leq n$ the minimal index such that $h(v_{i_0})= h_{\text{min}}(\gamma)$. 
		
		By definition, there exist two distinct points $p$ and $q$ on $T_{i_0}$ such that $T_{i_0+1}$ is obtained by blowing up $p$ and $T_{i_0-1}$ by blowing up $q$. Let $\pi\colon T'_{i_0}\to T_{i_0}$ be the blow-up of $p$ and $q$ and let $v_{i_0}'$ be the vertex given by $(T_{i_0}', \varphi_{i_0}\pi)$. Since $v_{i_0}'$, $v_{i_0+1}$, $v_{i_0}$ and $v_{i_0-1}$ form a square, we can deform $\gamma$ by a homotopy such that it passes through $v_{i_0}'$ instead of $v_{i_0}$. Moreover, $v_{i_0}'$ is also dominated by $w$:
			\begin{center}
			\begin{tikzpicture}[baseline= (a).base]
				\node[scale=.6] (a) at (0,0){
					\begin{tikzcd}[ampersand replacement=\&]
							\& w \arrow{dl} \arrow{d} \arrow{dr}\& \\
						\vdots \arrow{dd}	\& \vdots \arrow{d}  \& \vdots \arrow{dd}\\
						\& v_{i_0}'\arrow{dl} \arrow{dr}\& \\
						v_{i_0-1} \arrow{dr} \& \& v_{i_0+1} \arrow{dl}\\
							\& v_{i_0}\& 
				\end{tikzcd}};
			\end{tikzpicture}
		\end{center}
	
		By induction on $(\rho_{\text{min}}(\gamma), i_0)$ with the lexicographical order, we conclude that in finitely many such steps $\gamma$ is homotopic to $w$, and that the blow-up graph is simply connected.

 \textbf{Links of vertices are flag.} 
			Let $v=[(T,\varphi)]$ be a vertex of $\Cb(S )^0$ and $v_1,\dots,v_n\in \Cb(S)^0$ adjacent vertices to $v$ that pairwise generate a square with $v$. Up to reordering, we can assume that there exists $1\leq k\leq n$ such that for $1\leq j\leq k$, the vertex $v_j$ corresponds to blowing up a point $p_j\in T$ and that for $k\leq j\leq n$, the vertex $v_j$ corresponds to the contraction of an exceptional divisor $E_j$ in $T$ into a point $p_j$.
			Since the vertices are pairwise adjacent, this means that for $1\leq j\leq k$ all the $p_i$'s are distinct, that for $k\leq j\leq n$, all $E_j$ are disjoint and for all $1\leq j_1\leq k$ and for all $j_2\leq k\leq n$, $p_{j_1}\notin E_{j_2}$. Hence, contracting all the $E_j$'s gives us a surface $(\tilde{T},\psi)$ with $n$ distinct points: the images of the $p_i$'s. By Fact \ref{cc_blowup_cubes}, we have the expected cube, and links of vertices are flag.

This ends the proof that the blow-up graph is a median graph.
\end{proof}

\subsubsection{Hyperplanes}
In the graph $\Cb(S )$ an edge is given by the blow-up of a point $p$ belonging to a marked surface $(W ,\varphi)$. We denote such an edge by $(W ,\varphi,p)$ and by $[(W ,\varphi,p)]$ its corresponding hyperplane. In this subsection, we give a birational interpretation of hyperplanes, distance and halfspaces.

The following birational characterization of the equivalence class of edges is an immediate consequence of the Zarisky theorem \ref{thm:Zariski}.
\begin{lemma}\cite{Lonjou_Urech_cube_complexe}\label{lemme:equivalence_edges-surface}
	Two edges $(W ,\varphi,p)$ and $(W',\varphi',q)$ correspond to the same hyperplane if and only if $\varphi'^{-1}\varphi$ induces a local isomorphism between a neighborhood of $p$ and a neighborhood of $q$ and $\varphi'^{-1}\varphi(p)=q$.
\end{lemma}

The following lemma describes the geodesics of the blow-up graph. It is an immediate consequence of the factorization of a birational map in inverse of blow-ups and blow-ups (see Zarisky theorem \ref{thm:Zariski}), and of the characterization of geodesics by hyperplanes Theorem~\ref{theorem:combinatorial_geodesic}.

\begin{lemma}\cite{Lonjou_Urech_cube_complexe}\label{lemma:comb_geodesic}
	Consider two vertices $[(T_1 ,\varphi_1)]$ and $[(T_2 ,\varphi_2)]$ of $\Cb$. There is a bijection between the set of all the geodesic paths joining these two vertices and the set of all the possible order to decompose $\varphi_2^{-1}\varphi_1$ as blow-ups and inverse of blow-ups of points.
	In particular, the combinatorial distance between these two vertices is equal to:
	\[ \dist \big([(T_1 ,\varphi_1)],[(T_2,\varphi_2)] \big)= \lvert\Bs(\varphi_2^{-1}\varphi_1)\rvert+  \lvert\Bs(\varphi_1^{-1}\varphi_2)\rvert.\] 
\end{lemma}

\begin{figure}
	\begin{center}
		\begin{tikzpicture}
			\fill[opacity=0.5, color=red] (0,1)--(2,1)--(3,1.5)--(1,1.5);
			\draw[color=red] (3,1.5)--(4,2);
			\fill[opacity=0.5, color=orange!60] (1,2)--(2,2.5)--(2,0.5)--(1,0);
			\draw[color=orange!60] (2,2.5)--(2,4.5);
			\fill[opacity=0.5, color=yellow!60] (0.5,2.25)--(2.5,2.25)--(2.5,0.25)--(0.5,0.25);
			\draw[color=yellow!60] (2.5,2.25)--(4.5,2.25);
			\fill[opacity=0.5, color=blue!30] (3,3.5)--(5,3.5)--(6,4)--(4,4);
			\draw[color=blue!30] (3,3.5)--(1,3.5);
			\fill[opacity=0.5, color=green!40] (3.5,4.75)--(5.5,4.75)--(5.5,2.75)--(3.5,2.75);
			\draw[color=green!40] (3.5,2.75)--(3.5,0.75);	
			\fill[opacity=0.5, color=teal!50] (4,4.5)--(5,5)--(5,3)--(4,2.5);
			\draw[color=teal!50] (4,2.5)--(3,2);	
			\draw (0,0) -- (2,0) -- (2,2) -- (0,2)-- (0,0);
			\draw (3,0.5) -- (3,2.5) -- (1,2.5);
			\draw[dotted] (1,2.5)--(1,0.5) -- (3,0.5);
			\draw (2,0) -- (3,0.5); 
			\draw (2,2) -- (3,2.5);
			\draw (0,2) -- (1,2.5);
			\draw[dotted] (0,0) -- (1,0.5);
			\draw (0,0) node {$\bullet$} node[below left] {$(\PP^2,\id)$};
			\draw (2,0) node {$\bullet$} node[below right] {$(\F_1,\pi_q)$} ;
			\draw(2,2) node {$\bullet$} ;
			\draw (0,2) node {$\bullet$} node[above left] {$(\F_1,\pi_p)$};
			\draw (1,0.5) node {$\bullet$} node[ above left] {$(\F_1,\pi_r)$};
			\draw (3,0.5) node {$\bullet$} ;
			\draw(3,2.5) node {$\bullet$} node[ above left] {$S_{\sigma}$}; ;
			\draw (1,2.5) node {$\bullet$} ;
			\draw (3,0.5) -- (4,1)--(4,2);
			\draw[dotted] (4,2)--(4,3);
			\draw(5,2.5) node {$\bullet$} node[above right] {} ;
			\draw (4,1) node {$\bullet$} ;
			\draw (3,2.5) -- (5,2.5) -- (5,4.5) -- (3,4.5)-- (3,2.5);
			\draw (6,3) -- (6,5) -- (4,5);
			\draw[dotted] (4,5)--(4,3) -- (6,3);
			\draw (5,2.5) -- (6,3); 
			\draw (5,4.5) -- (6,5);
			\draw (3,4.5) -- (4,5);
			\draw[dotted] (3,2.5) -- (4,3);
			\draw (3,2.5) node {$\bullet$} node[below left] {};
			\draw (5,2.5) node {$\bullet$} ;
			\draw(5,4.5) node {$\bullet$} ;
			\draw (3,4.5) node {$\bullet$} node[above left] {};
			\draw (4,3) node {$\bullet$} ;
			\draw (6,3) node {$\bullet$} ;
			\draw(6,5) node {$\bullet$} node[above right] {$(\PP^2,\sigma)$} ;
			\draw (4,5) node {$\bullet$} ;
			\draw (1,2.5) -- (1,4.5)-- (3,4.5);
			\draw (2,2) -- (4,2)-- (5,2.5);
			\draw (4,2) node {$\bullet$} ;
			\draw (1,4.5) node {$\bullet$} ;
		\end{tikzpicture}
		\caption{The convex hull of the vertices $(\PP^2,\id)$ and $(\PP^2,\sigma)$ in $\Cb(\PP^2)$, where $\sigma$ is the standard quadratic involution.\label{Figure_blowup_complex}} 
	\end{center}
\end{figure}
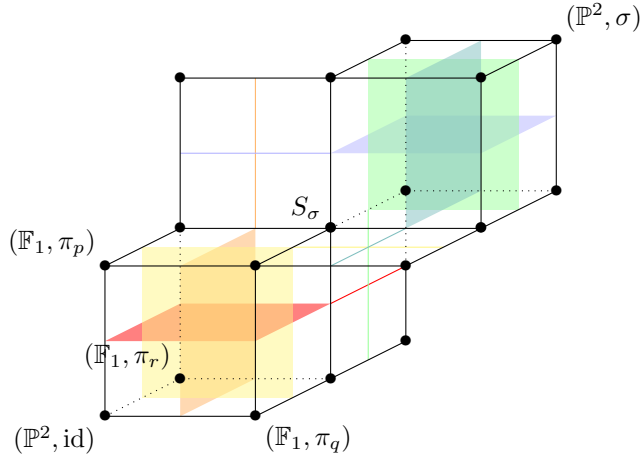

For instance, in Figure \ref{Figure_blowup_complex}, we see all the possible geodesics joining the vertices $[(\PP_{\kk}^2,\id)]$ and $[(\PP_{\kk}^2,\sigma)]$. Nevertheless, sometimes the choices are more restrictive like for instance between the vertices $[(\PP_{\kk}^2,\id)]$ and $[(\PP_{\kk}^2,j)]$ (see Figure \ref{Figure_blowup_complex_j})

There is an algebraic characterization of the halfspaces associated to a given hyperplane.

\begin{proposition}\cite{Lonjou_Urech_cube_complexe}\label{prop:halfspace}
	Consider a hyperplane $[(W ,\varphi,p)]$ in $\Cb(S )$. The set of vertices $(T ,\varphi_1)$ such that $p$ is a base point of $\varphi_1^{-1}\varphi$ determines one of the two halfspaces. 
\end{proposition}

For instance, let $f\in\Bir(S )$ and $p\in S $ be a base point of $f$. Then $[(S ,f^{-1})]$ belongs to the halfspace given by Proposition \ref{prop:halfspace}, while $[(S ,\id)]$ belongs to the other halfspace (see Figure \ref{Figure_blowup_complex} for an illustration in the case where $f=\sigma$ and $S=\PP^2$). 

\begin{figure}
	\begin{center}
		\begin{tikzpicture}
			\draw (0,0) -- (0,1.5) -- (0,3) ;
			\draw (3,3) --(3,1.5) --(3,0);
			\draw[red] (0,3) -- (1.5,4);
			\draw[red] (1.5,2) -- (3,3);
			\draw[blue] (0,3) -- (1.5,2);
			\draw[blue] (1.5,4) -- (3,3);
				\draw (0,0) node {$\bullet$} node[below] {$[(\PP^2,\id)]$};
			\draw (0,1.5) node {$\bullet$} node[left] {$[(\F_1,\pi_{[0:1:0]})]$} ;
				\draw (0,3) node {$\bullet$} node[left] {$[(\F_1,\pi_{[0:1:0]}\pi_{(0,0)})]$} ;
				\draw (1.5,2) node {$\bullet$};
					\draw (1.5,4) node {$\bullet$} node[above] {$(S_j,\pi_{[0:1:0]}\pi_{(0,0)}\pi_{(0,1)})$} ;
						\draw (3,0) node {$\bullet$} node[below right] {$[(\PP^2,j)]$};
					\draw (3,1.5) node {$\bullet$} ;
					\draw (3,3) node {$\bullet$};
		\end{tikzpicture}
	\end{center}
\caption{Convex hull of the vertices $[(\PP^2,\id)]$ and $[(\PP^2,j)]$ in the blow-up graph $\Cb(\PP^2)$ where $j:(x,y)\mapsto (x,x^2+y)$.\label{Figure_blowup_complex_j}}
\end{figure}
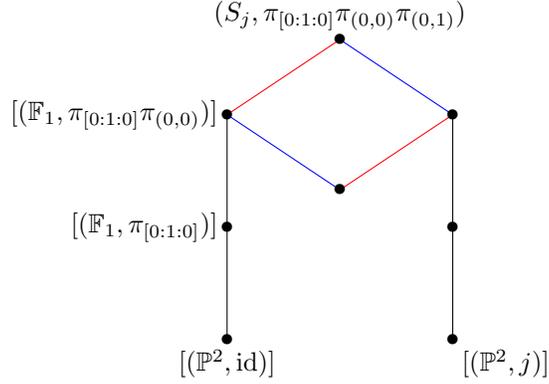

\medskip
\subsubsection{Action of the Cremona group on the blow-up graph}
The group of birational transformation $\Bir(S)$ acts faithfully on the set of marked surfaces by post-composition and preserves the equivalence class, hence it acts on the set of vertices of $\Cb(S)$: 
let $f\in \Bir(S)$ and $[(T,\varphi)]\in\Cb(S)^0$ be a vertex,  $f\bullet [(T,\varphi)] = [(T,f\varphi)]$. This gives us a faithful action by isometries on the blow-up graph. Note that this action preserves the height function and so the cubical orientation given previously. As a consequence, an element preserving a cube, fixes a vertex (the one with the smallest height).

As we will see in Proposition \ref{prop_catalog_blowup_cc}, this action encodes geometrically, and in a unified way, diverse birational notions. Let us first introduce them.

Let $f\in \Bir(S)$. The \emph{dynamical number of base points} of $f$ is defined as:\[
\mu(f)\coloneqq \lim_{n\to\infty}\frac{\lvert\Bs(f^n)\rvert}{n}.\]
As we have seen in Remark \ref{rmk_bspt_subadditive}, the number of base points is sub-additive. Hence, the limit always exists. This number has been introduced and studied in \cite{Blanc_Deserti}. The dynamical number of base points is invariant by conjugation, since conjugation by a birational transformation $g$ changes the number of base points at most by a constant only depending on $g$. 

Blanc and Déserti show that, over an algebraically closed field of characteristic zero, the dynamical number of base points characterizes regularizable elements: they are the birational transformations having dynamical number of base points equal to $0$. Note that they do not use the assumption on the characteristic of the field.
\begin{theorem}[\cite{Blanc_Deserti}]\label{Blanc_Deserti} 
	Let $S $ be a smooth projective surface and let $f\in~\Bir(S )$. Then 
	\begin{enumerate}
		\item $\mu(f)$ is an integer;
		\item there exists a smooth projective surface $T $ and a birational map $\varphi\colon T \dashrightarrow S $ such that $\varphi^{-1} f\varphi$ has exactly $\mu(f)$ base points;
		\item in particular, $\mu(f)=0$ if and only if $f$ is conjugate to an automorphism of a smooth projective surface.
	\end{enumerate}
\end{theorem}

Consider an element of the Cremona group $f\in \Bir(\PP^2)$. Its \emph{dynamical degree} is defined as :\[\lambda(f)\coloneqq \lim_{n\to\infty} \deg(f^n)^{\frac{1}{n}}.\]As the sequence $\left(\deg(f^n)\right)$ is sub-multiplicative, the limit always exists, and it is invariant by conjugacy. 
It measures the complexity of the dynamics of $f$. For instance, if $\kk$ is the field of complex numbers, $\log(\lambda(f))$ provides an upper bound for the topological entropy of $f$ and it most of the time equal to it (see \cite{Bedford_Diller_Energy_measure} and \cite{Dinh_Sibony_bornesup_entropie}). 
Note that one can define the dynamical degree for any birational transformation of a projective surface but this requires a bit more work, so we will not do it here. It has been studied by Diller and Favre \cite{Diller_Favre_Dynamics_bimeromorphic}. For instance they show that the dynamical degree is an algebraic integer. More precisely, if it is not equal to $1$ then it has to be a Salem number (algebraic integer in $]1, \infty[$ whose other Galois conjugates lie in the closed unit disk, with at least one on the boundary) or a Pisot number (algebraic integer in $]1, \infty[$ whose other Galois conjugates lie in the open unit disk). A key notion in this context is the one of algebraic stability.

Any birational transformation $f\in \Bir(S)$ induces an endormorphism $f_*$ of the Néron-Severi group $\NS(S)$. When this endomorphism satisfies:
\[(f_*)^n=(f^n)_*,\] for all positive integers $n$, $f$ is \emph{algebraically stable} (see for instance \cite{Diller_Favre_Dynamics_bimeromorphic}). From a dynamical point of view, such maps are important as in this case the dynamical degree of $f$ equals the spectral radius of the endomorphism $f_*$, and dynamical degrees are usually complicated to compute.
Diller and Favre showed that every birational transformation of a projective surface admits an algebraically stable model.

\begin{theorem}[\cite{Diller_Favre_Dynamics_bimeromorphic}]\label{thm_Diller_Favre}
	Every birational transformation of a surface $S$ is conjugate to an algebraically stable transformation by a birational map $T\dashrightarrow S$. 
\end{theorem}


Using the vocabulary from \cite{Lamy_book}, we say that a (-1)-curve $E$ on a surface $T$, i.e. a smooth rational curve with self-intersection $-1$, is \emph{redundant} for $f$ if $E\subset \Exc(f)$ but $E\cap \Ind(f)=\emptyset$. 
Contracting such a curve preserves algebraic stability (see \cite[Lemma 4.18]{Lamy_book}).
A birational transformation $f$ is \emph{minimally algebraically stable} if $f$ is algebraically stable and for any integers $n\geq 1$, there is no (-1)-curve $E$ that is simultaneously redundant for $f$ and for $f^{-n}$. Indeed, as a consequence of Theorem \ref{thm_Diller_Favre} and  \cite[Lemma 4.18]{Lamy_book}, any birational transformation can be conjugated to a minimally algebraically stable transformation.

The previous properties just seen are encoded geometrically as follows. Recall that the minimizing set of a birational transformation $f\in \Bir(S)$, denoted by $\Min(f)$ is the set of vertices realizing the translation length of $f$. Note that in \cite{Lonjou_Urech_cube_complexe}, the following proposition is for regular projective surfaces over any field. Moreover,  in \cite{Lonjou_Urech_cube_complexe} the authors called ``algebraically stable transformations''  the ones being minimally algebraically stable, and they gave only the implication of the point \ref{item_alg_stability} that allowed to reprove \ref{thm_Diller_Favre}. The geometric characterization of algebraically stable models have been done in \cite{Lamy_book}.
\begin{proposition}[\cite{Lonjou_Urech_cube_complexe}, \cite{Lamy_book}]\label{prop_catalog_blowup_cc}Let $S$ be a smooth projective surface over an algrebraically closed field. We have the following correspondences between geometric and birational notions.
	\begin{enumerate}
	\item\label{item_elliptic} Elements of $\Bir(S)$ inducing elliptic isometries of the blow-up graph $\Cb(S)$ correspond to projectively regularizable elements of $\Bir(S)$. 
	\item\label{item_distance} The distance between the vertex $[(S,\id)]$ and its image by an element $f$ of $\Bir(S)$ corresponds to twice the number of base-points of $f$:
\[	\dist\left([(S,\id)], [(S,f)]\right)=2 \lvert \Bs(f)\rvert. \]
	Moreover, there is a bijective correspondance between the set of all possible orders of decomposing $f^{-1}$ into composition of blow-ups (and inverse of blow-ups) of points and the set of geodesics between these two vertices.
	\item\label{item_translation_length} The translation length corresponds to twice the dynamical number of base-points:
	\[\ell(f)=2\mu(f).\]
	\item\label{item_alg_stability} The set of vertices of $\Min(f)$ of a birational transformation $f\in \Bir(S)$ corresponds to minimally algebraically stable models:
	$[(T,\varphi)]$ belongs to $\Min(f)$ if and only if $\varphi^{-1}f\varphi$ is minimally algebraically stable on $T$.
	\end{enumerate}
\end{proposition}

\begin{proof}[Idea of proof of Proposition \ref{prop_catalog_blowup_cc}]
The point \ref{item_elliptic} follows from the definition of the vertices of the blow-up graph, the point \ref{item_distance} has been seen above in Lemma~\ref{lemma:comb_geodesic}.

Let us focus on the point \ref{item_translation_length}. Consider $f\in\Bir(S)$ and $x\in \Min(f)$. By Proposition \ref{prop_action_semisimple_hag}, for any $n\in\N$ we have $\dist(x,f^n(x))=n\ell(f)$. Let $v=[(S , \id)]$. Then $\dist(v, f^n(v))=2\lvert\Bs(f^n)\rvert$ by Lemma \ref{lemma:comb_geodesic}. This implies that for any $n\in \N$, $n\ell(f)\leq 2\lvert\Bs(f^n)\rvert$. Taking the limit we obtain: $\ell(f)\leq 2\mu(f)$. 
	On the other hand, let $x\in \Min(f)$ and $K=\dist(v, x)$. Then for any $n\in \N$, $2\lvert\Bs(f^n)\rvert=\dist(v, f^n(v))\leq n
\ell(f)+2K$ and hence $2\mu(f)\leq \ell(f)$.
	
The point \ref{item_alg_stability} is a consequence of the fact that a birational transformation $g$ is minimally algebraically stable if and only if for all $n\geq 1$, $\lvert\Bs(g^n)\rvert=n\lvert\Bs(g)\rvert$ (\cite[Proposition 4.22]{Lamy_book}). 

Assume $[(T,\varphi)]$ belongs to $\Min(f)$, then using the characterization of the distance between two vertices in term of base-points Lemma \ref{lemma:comb_geodesic} and the fact that isometries are semi-simple Proposition \ref{prop_action_semisimple_hag}, we obtain that for all $n\geq 1$, $\lvert\Bs\left((\varphi^{-1} f\varphi)^n\right)\rvert=n\lvert\Bs\left((\varphi^{-1} f\varphi)^n\right)\rvert$ and so $\varphi^{-1} f\varphi$ is minimally algebraically stable on $T$. On the other hand, if $\varphi^{-1} f\varphi$ is minimally algebraically stable on $T$, then for all $n\geq 1$, $\lvert\Bs\left((\varphi^{-1} f\varphi)^n\right)\rvert=n\lvert\Bs\left((\varphi^{-1} f\varphi)^n\right)\rvert$ and so, by writing $x=[(T,\varphi)]$, $\dist(x,f^n(x))=n\dist (x,f(x))$ for all $n\geq 1$ implying that $\dist(x,f(x))$ is the translation length. This implies that $x\in \Min(f)$ as expected.
\end{proof}

Theorem \ref{Blanc_Deserti} is a direct consequence of the characterization of the dynamical number of base-points in term of translation length (Proposition \ref{prop_catalog_blowup_cc} \eqref{item_translation_length}), and Theorem \ref{thm_Diller_Favre} is a direct consequence of Proposition \ref{prop_catalog_blowup_cc} \eqref{item_alg_stability}.

\medskip
\subsubsection{Regularization results} 
By Proposition \ref{prop_catalog_blowup_cc}.\ref{item_elliptic}, stabilizers of vertices are projectively regularizable subgroups of $\Bir(S)$. Hence, this action allows to obtain straightforwardly and in an unified way some regularization results. 

\begin{remark}
Note that in the following regularization results, smoothness is assumed for surfaces, but up to desingularization, we can obtain the same results for any surface.
\end{remark}

For instance, it is immediate that groups with property FW are projectively regularizable. 

\begin{proposition}[{\cite[Theorem B]{Cantat_Cornulier_Commensurating}}]Let $S$ be a smooth projective surface.
Subgroups of $\Bir(S)$ having the FW property are projectively regularizable.
\end{proposition}
Note that in \cite{Cantat_Cornulier_Commensurating}, they show a stronger result, namely that a subgroup $G\subset~\Bir(\PP^2)$ with the FW property is conjugated to a subgroup of the automorphism group of a \emph{minimal} surface (without $(-1)$-curve) using some classification/geometry of surfaces.

\medskip
As the action of $\Bir(S)$ on the blow-up graph $\mathcal{C}_b(S)$ preserves the cubical orientation, a subgroup of $\Bir(S)$ having bounded orbits fixes a vertex of $\mathcal{C}_b(S)$ and by definition is projectively regularizable. Hence, using Proposition \ref{prop_catalog_blowup_cc}, projectively regularizable subgroups are the ones having a uniform bound on the number of base-points of its elements. Moreover, the number of base points of an element $g\in\Bir(\PP^2 )$ is bounded by a constant depending only on the degree of $g$. As a consequence, the \emph{bounded subgroups} of $\Bir(\PP^2)$, i.e., the subgroups such that the degree of all its elements is uniformly bounded, are projectively regularizable.

\begin{proposition}\label{thm_reg_boundedorbit}
	Let $S$ be a smooth projective surface. A subgroup $G\subset \Bir(S)$ is projectively regularizable if and only if there exists a constant $K$ such that $\lvert\Bs(f)\rvert\leq K$ for all $f\in G$.
	
	In particular, bounded subgroups of $\Bir(\PP^2)$ are projectively regularizable.
\end{proposition}

Note that the there exists projectively regularizable subgroup of $\Bir(\PP^2)$ that are not bounded. For instance, group of automorphisms of Halphen surfaces (see for instance \cite{Lamy_book}).

Several classes of birational transformations are projectively regularizable. Let us introduce them. In a group $G$, an element $g\in G$ is: 
\begin{itemize}
	\item \emph{divisible}, if for every integer $n\geq 0$ there exists an element $f\in G$ such that $f^n=g$;
	\item  \emph{distorted}, if $\lim_{n\to\infty}\frac{|g^n|_{S}}{n}=0$ for some finitely generated subgroup $\Gamma\subset G$ containing $g$, where $|g^n|_{S}$ denotes the word length of $g^n$ in $\Gamma$ with respect to some finite set $S$ of generators of $\Gamma$. 
\end{itemize}
 Distorted elements in $\Bir(\PP^2 )$ for an algebraically closed field $\kk$ of characteristic $0$ have been classified in \cite{Blanc_Furter_length} and \cite{Cantat_Cornulier_Distortion}, and a description of divisible elements in $\Bir(\PP^2)$ can be found in \cite{Liendo_Regeta_Urech_affine_toric_var}. 

\begin{proposition}[\cite{Lonjou_Urech_cube_complexe}]\label{prop_divisible_distorded_2}
Divisible and distorted birational transformations of smooth projective surfaces are regularizable.
\end{proposition}

\begin{proof}
Let $S$ be a surface, let $g\in\Bir(S)$ be a divisible or distorted element and define the length function $L(f):=\dist([(S,\id)],[(S,f)])$ on $\Bir(S)$. It is subadditive, takes integral values, and the map $n\mapsto L(f^n)$ grows asymptotically like $n\ell(f)$ by Corollary \ref{prop_action_semisimple_hag}. Since $g$ is  divisible or distorted, then the restriction of $L$ to $<g>$ has to be bounded, hence $g$ has a bounded orbit on $\Cb(S)$ and so $g$ is projectively regularizable.
\end{proof}

\begin{proposition}[\cite{Lonjou_Urech_cube_complexe}] Let $S$ be a smooth projective surface, $G\subset \Bir(S )$ be a subgroup and $H\subset G$ a subgroup of finite index. Then $G$ is projectively regularizable if and only if $H$ is.
\end{proposition}

The above proposition comes from the fact that, if $H$ is projectively regularizable, then the orbit of $G$ on any vertex fixed by $H$ is bounded and so $G$ is projectively regularizable.

\medskip
\subsubsection{Open question}\label{subsubsection_open_reg}

We are interested in the following open question (see Question \ref{question_loc_elliptic}):	
	Consider a finitely generated subgroup $G$ of the Cremona group of rank $2$ such that each of its elements is projectively regularizable. Does it imply that \emph{$G$ is projectively regularizable}?
In term of the action of the Cremona group on the blow-up graph, this question can be reformulated as follows:

\begin{question}\label{question_loc_elliptic_cube}
	Let $G$ be a finitely generated subgroup of $\Bir(S)$ acting purely elliptically on the blow-up graph. Does it implies that $G$ fixes a vertex? (or equivalently that $G$ has bounded orbits?)
\end{question}

This question has to be related to the general question about finitely generated groups acting purely elliptically on median graphs (see Question \ref{question:purely_elliptic}). As we have already seen in Section \ref{Subsection_purely_elliptic}, when the median graph is not locally of finite dimension (which is the case of the blow-up graph), there exist counterexamples, so there is no direct answer using actions on median graphs so far. Nevertheless, this action seems to be a good starting tool to answer this question. 

Note that, also in this context, it is necessary to require the subgroup to be finitely generated, otherwise, there exist counterexamples. For instance, the subgroup $G=\{(x,y)\mapsto(x,y+P(x))\mid P\in \kk[X]\}$ of the Cremona group contains only elliptic elements. Indeed, for any $g\in G$, $g^n(x,y)=(x,y+nP(x))$, so $\deg(g^n)=\deg(g)=\deg(P)$ if $\deg(P)\geq1$ and $g$ is projectively regularizable by Proposition \ref{thm_reg_boundedorbit}. Nevertheless, for any $n\in \N^*$, the element $(x,y)\mapsto (x,y+x^n)\in G$ has $2n-1$ base-points so there is no uniform bound on the number of base-points of elements of $G$ implying by Proposition \ref{thm_reg_boundedorbit} that $G$ is not projectively regularizable.

\subsection{The rational blow-up graph over finite fields}\label{Subsection_rational_blow_up graph}
In \cite{GLU_Neretin}, the authors show that over a finite field the Cremona group of rank $2$ embedds into a Neretin group. They construct a median graph for the Neretin group, inspired by the construction of the blow-up graph. Using this action, they answer positively Question \ref{question_loc_elliptic} for Cremona groups of rank $2$ over finite fields.

\begin{theorem}[{\cite{GLU_Neretin}}]\label{thm_reg_finite_field}
	Let $\F$ be a finite field and $S$ be a projective regular surface defined over $\F$. Let $G$ be a finitely generated subgroup of $\Bir(S)$ such that for any $g\in G$, $g$ is projectively regularizable. Then $G$ is protectively regularizable.
\end{theorem}

In this survey, we present how this result can be reproved by constructing a new median graph called \emph{the rational blow-up graph}. Note that this construction is different from the one of \cite{GLU_Neretin}, even when reformulating the construction to study only the Cremona group of rank $2$ over finite fields (instead of the all Neretin group). 
As we will see, the rational blow-up graph is a convex subgraph of the blow-up graph, making him immediately median. Moreover, it is locally finite dimensional so this allows to apply Theorem~\ref{thm:purely_elliptic_locally_finite_dim} in order to answer positively Question \ref{question_loc_elliptic} for Cremona groups of rank $2$ over finite fields.
Nevertheless, as we will see, the all Cremona group of rank $2$ over a finite field does not act on it, but only some subgroups, which is sufficient for our purpose.

\medskip

\subsubsection{Comments on the field}
In this section, we consider projective surfaces defined over a perfect field $\kk$, denoted by $S_{\kk}$. Let $\bar{\kk}$ be an algebraic closure of $\bar{k}$. In this context, $S_{\kk}$ is the pair $(S_{\bar{\kk}},\Gal(\bar{\kk}/\kk))$, where the Galois group $\Gal(\bar{\kk}/\kk)$ acts on the second factor of $S_{\bar{\kk}}=S \times_{\Spec(\kk)}\Spec(\bar{\kk})$.

In other words, the main problem, when working over non-algebraically closed field, is that the notion of points changes. Indeed, there is an action of the Galois group $\Gal(\bar{\kk}/\kk)$ on the set of points of $S_{\bar{\kk}}$. The orbit of such a point will be called a closed point of $S_{\kk}$. For instance, over $\R$, the Galois group $\Gal(\C/\R)$ acts by complex conjugation on the points of $\PP^2_{\C}$. The points $[0:1:i]$ and $[0:1:-i]$ are in the same Galois-orbit and form the closed point $\left\{y^2+z^2=x=0\right\}$ of $\PP^2_{\R}$. Once $\kk$ is fixed, the set of \emph{rational points} of $S$, denoted by $S(\kk)$, is the set of orbits of size $1$. For instance, the set of rational points of $\PP^2$ is $\PP^2(\kk)=\{[a:b:c]\mid a,b,c \in \kk\}$.

Notice that, when working over a finite field, a surface has infinitely many points but finitely many rational ones. As a consequence, the blow-up graph associated to a surface over a finite field is still not locally of finite dimension.

\begin{theorem}[{\cite[Theorem 1.3]{Lonjou_Urech_cube_complexe}}]\label{prop_reg_field_extension}
	Let $S$ be a surface over a perfect field $\kk$ and let $G\subset\Bir(S)$ be a subgroup defined over $\kk$. $G$ is projectively regularizable over $\kk$ if and only if it is projectively regularizable over its algebraic closure $\bar{\kk}$.
\end{theorem}

\medskip

\subsubsection{Construction}The construction of the rational blow-up graph is general and does not required the field to be finite, so let $\kk$ be any field.
A \emph{rational marked surface} $(T ,\varphi)$ is a pair where $T$ is a regular projective surface over $\kk$ and $\varphi\colon T\dashrightarrow S $ is a birational map such that all its base-points and the ones of its inverse are rational: $\Bs(\varphi)(\kk)=\Bs(\varphi)$ and $\Bs(\varphi^{-1})(\kk)=\Bs(\varphi^{-1})$. As before, two rational marked surfaces $(T ,\varphi)$ and $(T ',\varphi')$ are equivalent if the map $\varphi'^{-1}\varphi\colon T \to T '$
is an isomorphism. Such a class, will be denoted by $[(T ,\varphi)]$.

\begin{definition}The \emph{rational blow-up graph} associated to $S$ and denoted by $\Cbk(S)$ is the graph whose vertices are equivalent classes of rational marked surfaces $[(T,\varphi)]$, and where two vertices $[(T_1,\varphi_1)]$ and $[(T_2,\varphi_2)]$ are connected by an edge if $\varphi_2^{-1}\varphi_1$ is the blow-up of a rational point of $T_2$ or the inverse of the blow-up of a rational point of $T_1$. 
\end{definition}

Recall that an exceptional divisor is isomorphic to $\PP^1$, and that $\PP^1$ contains at least three rational points. As a consequence, if $\pi:T_2 \rightarrow T_1$ is the blow-up of a rational point of $T_1$, $T_2$ contains strictly more rational base points than $T_1$. This implies, that the rational blow-up graph is infinite dimensional. On the other hand, it the field is finite, any surface contains only a finite number of rational base-points and so the graph is locally compact and so locally of finite dimension.
Moreover, as consequence of Lemma \ref{lemma:comb_geodesic} and by construction of the rational blow-up graph, the later graph is a convex subgraph of the blow-up graph so it is also a median graph.

\begin{proposition}
	The rational blow-up graph $\Cbk(S)$ is an oriented and infinite dimensional median graph. Moreover, if the field is finite, it is locally compact, hence locally of finite dimension.
\end{proposition}

\begin{proof}[Proof of Theorem \ref{thm_reg_finite_field}]
	Fix a symmetric finite generating set $\{g_i\}_{1\leq i \leq n}$ of $G$. 
	Let $L$ be a finite field extension of $\F$ such that all the base-points of the $g_i$'s, and hence the base-points of all elements in $G$ are defined over $L$. We now consider the action of $\Bir(S_L)$ on the median graph $\mathcal{C}_{b,L}(S_L)$ that is locally of finite dimension because $L$ is finite. By definition, every element in $G$ is regularizable over $\F$ and so over $L$ by Theorem \ref{prop_reg_field_extension}. Hence they are elliptic. Theorem \ref{thm:purely_elliptic_locally_finite_dim} implies that $G$ fixes a vertex in $\mathcal{C}_{b,k}(S_L)$, meaning that $G$ is regularizable over $L$. Again by Theorem~\ref{prop_reg_field_extension}, $G$ is regularizable over $\F$.
\end{proof}

\subsection{The Jonquieres graph}\label{Subsection_Jonquieres}
An element of $\Bir(\PP^2)$ is \emph{algebraic} if its degree is bounded under iteration. Recall that a subgroup $G\subset\Bir(\PP^2)$ is \emph{bounded} if the set of all the degrees of the elements in $G$ is bounded. 
This subsection is dedicated to the following result, which answer positively a question of Favre \cite{favrebourbaki} and Cantat \cite{Cantat_groupes_birat}:

\begin{theorem}[\cite{Lonjou_Przytycki_Urech}]\label{thm:main}
	Finitely generated subgroups of $\Bir(\PP^2)$ containing only algebraic elements are bounded.
\end{theorem}

Note that algebraic elements of $\Bir(\PP^2)$ correspond to the elliptic isometries of the action of $\Bir(\PP^2)$ on the hyperboloid $\Hh^\infty$ introduced in Subsection \ref{Subsection_State_art}. Hence, theorem \ref{thm:main} can be rephrased as follows: finitely generated subgroups of $\Bir(\PP^2)$ acting by elliptic isometries on $\Hh^\infty$ have a fixed point in $\Hh^\infty$.

Being algebraic or bounded is invariant under field extensions so we can assume the field to be algebraically closed. 
The first step to prove Theorem~\ref{thm:main} is due to Cantat, who showed in \cite{Cantat_groupes_birat} that a finitely generated subgroup $G\subset\Bir(\PP^2)$ consisting of algebraic elements is either bounded or is conjugated to a Jonquières subgroup. It is therefore enough to show Theorem~\ref{thm:main} for finitely generated subgroups of the Jonquières subgroup.
The main ingredients are the action of the Jonquières group on a median graph, called \emph{the Jonquières graph} and the study of the dynamics of $\PGL_2(\kk)$ on $\PP^1$.

\medskip

Let us construct the Jonquières graph.
Here, a \emph{rational fibration} on a surface $S$ is a morphism $\pi\colon S\to C$, where $C$ is a 
curve, such that all the fibres are isomorphic to $\PP^1$. Let $\pi\colon S\to C$ and $\pi'\colon S'\to C$ be rational fibrations. A \emph{conic bundle} is a composition with a rational fibration $\pi\colon S\to C$ of a sequence of blow-ups $\tilde S\to S$, such that we blow up in total at most one point~$s\in S$ in each fibre.

\begin{definition}
Let $\F_0=\PP^1\times\PP^1$ with the rational fibration $\pi_0\colon \F_0\to\PP^1$ onto the first factor. The \emph{Jonqui\`eres graph} $\Jonq$ is the subgraph of the blow-up graph $\Cb(\F_0)$ induced on the set of vertices represented by marked surfaces $[(S,\varphi)]$ such that, for $\pi=\pi_0\varphi$, the rational map
$\pi\colon S\dashrightarrow\PP^1$ is a conic bundle. 
\end{definition}

	Note that the Jonquières graph $\Jonq$ is not a convex subgraph of the blow-up graph $\Cb(\F_0)$. Consider the minimal resolution of $j:(x,y)\mapsto (x,x^2+y)$ (see Example \ref{fig_resol_j}). The vertex $[(S_j,\pi_2)]$ is not a vertex of the Jonquières graph because more than one point has been blown up in the same fiber. 
Nevertheless, the graph $\Jonq$ is isometrically embedded in the graph $\mathcal \Cb(\F_0)$.

For any $p\in \PP^1$, consider the subgraph $X_p$ of $\Jonq$ induced on the vertices represented by the marked surfaces $[(T,\varphi)]$, where $\varphi\colon T\dashrightarrow \F_0$ induces an isomorphism between $T\setminus \pi^{-1}(p)$ and $\F_0\setminus \pi_0^{-1}(p)$. Fix $x_p=[(\mathbb{F}_0, \id)] $ in each $X_p$, and consider the family of pointed sets $\{(X_p,x_p)\}_{p\in \PP^1}$. Its \emph{restricted product } $\bigoplus_{p\in \PP^1}(X_p,x_p)$ is the set of sections $\{y_p\}_{p\in \PP^1}$ with $y_p\in X_p$ such that all but finitely many $y_p$ are equal to $x_p$.

The coordinate $y_p$ of a point $y\in  \bigoplus_{p\in \PP^1} (X_p,x_p)=\mathcal{J}$ should be thought of as the ``marked fibre''
in the surface corresponding to $y$ over the point $p$. Thus modifying the coordinate $y_p$ of $y$ corresponds to performing an alternating sequence of blow-ups of points and blow-downs of $-1$ curves in the fibre over $p$ of the surface corresponding to $y$.

\begin{theorem}[\cite{Lonjou_Przytycki_Urech}]\label{thm_jonq_median}
	For any $p\in \PP^1$, the graph $X_p$ is a tree and the graph $\Jonq$ is isomorphic to the restricted product: $\bigoplus_{p\in \PP^1} (X_p,x_p)$; hence it is a median graph.
\end{theorem}

The Jonqui\`eres group acts on the vertex set of $\Jonq$ by $f\bullet[(S, \varphi)]=[(S, f\varphi)]$. Moreover, the isomorphism of Theorem \ref{thm_jonq_median} induces an action of the Jonquières group on $\bigoplus_{p\in \PP^1} (X_p,x_p)$. Using the dynamics of $\PGL_2(\kk)$ on $\PP^1$, we show that this action is \emph{decent}, i.e.,  each subgroup of the Jonquières group with a finite orbit fixes a point of $\bigoplus_{p\in \PP^1} (X_p,x_p)$, and each finitely generated subgroup of the Jonquières group acting purely elliptically fixes a point of $\bigoplus_{p\in \PP^1} (X_p,x_p)$. Hence, using again Theorem \ref{thm_jonq_median}, this means that each finitely generated subgroup of the Jonquières group acting purely elliptically is conjugated to a subgroup of an automorphism of a projective smooth surface, and so it is bounded.

\section{Median graph for groups of birational transformations of varieties of any dimension}\label{Section_median_graphs_higher_ranks}
In this section, we assume varieties to be over an algebraically closed field (only for technical reasons as before).
We introduce the first and, up to the knowledge of the author, the unique geometrical spaces with  non-positive curvature known today in order to study from a geometric group theoretic point of view groups of birational transformations of varieties of arbitrary dimension. This section is taken from \cite{Lonjou_Urech_cube_complexe} and translated in term of median graphs instead of CAT(0) cube complexes.

 It is very tempting to try to adapt the construction of the blow-up graph to varieties of higher dimensions requiring vertices to be marked \emph{projective} varieties of the given dimension.
  But sadly this is not possible. One of the reason is that the group of monomial birational transformations of $\PP^3$, which is isomorphic to $\GL_3(\Z)$, has the property FW.  Hence, it would act on a median graph preserving a cubical orientation, so it should fixes a vertex. However, it is known that it is not conjugated to any subgroup of the automorphism group of a projective variety. This can be seen, for instance, by considering the degree sequence of the monomial transformation $(x,y,z)\dashmapsto (yx^{-1},zx^{-1},x)$, which does not satisfy any linear recurrence and is therefore not conjugate to an automorphism of a regular projective threefold (see \cite{Hasselblatt_Propp} or \cite{Cantat_Zeghib} for details).

As a consequence, it is necessary to leave the projective world. Hence, a natural generalization is to consider varieties as quasi-projective varieties over algebraically closed fields. This works to build the median graph $\CC^0(X)$ on which the group of birational transformations of a variety $X$ acts on. In \cite{Lonjou_Urech_cube_complexe}, the authors also build several median spaces for different subgroups of groups of birational transformations. For these later constructions, it is not clear anymore if considering only quasi-projective varieties, the graphs constructed are median. In order to prove that the links are flag, they need to be able to glue varieties along open dense subsets and stay in the same class of varieties. Hence in this section varieties are more general than quasi-projective varieties. To be precise they are integral and separated schemes of finite type over $\kk$.

\subsection{Actions of groups of isomorphisms in codimension $\ell$ on median graphs.}\label{Subsection_median_graphs_groups_isom_codim_l}

Consider two varieties $A$ and $B$. A birational map $f:A\dashrightarrow B$ is called \emph{isomorphism in codimension $\ell$} if $\Exc(f)$ and $\Exc(f^{-1})$ have codimension $>\ell$. Note that isomorphisms in codimension $1$ are usually called \emph{pseudo-isomorphisms}, and birational transformations are isomorphisms in codimension $0$.
For any $1\leq k\leq \dim(X )$ we denote by $\Exc^k(f)$ the set of irreducible components of the exceptional locus of $f$ of codimension $k$.

The following example is an easy but fundamental example for the construction of the median graph that we will construct in Section \ref{subsection:ccn}.
\begin{example}\label{ex_remov_subv}
	If $D\subset X$ is an irreducible subvariety of codimension $\ell+1$ then $\iota: X\setminus D \hookrightarrow X$ is an isomorphism in codimension $\ell$, with $\Exc(\iota)=\emptyset$ and, $\Exc(\iota^{-1})=D=\Exc^{\ell+1}(\iota^{-1})$ has a unique irreducible component of codimension $\ell+1$.
\end{example}

When we fix a variety $X$, isomorphisms in codimension $\ell$ from $X$ to itself form a subgroup of the group of birational transformations of $X$, which is denoted by $\Psaut^\ell(X)$. When $\ell$ is $0$ then it is the whole group $\Bir(X)$, on the other hand, when $\ell=\dim X$, then it is the group of automorphisms of $X$. 
A subgroup $G\subset \Bir(X)$ is called \emph{pseudo-regularizable in codimension $\ell$} if there exist a variety $Y$ and a birational map $\varphi:Y \dashrightarrow \PP^n$ such that \[\varphi^{-1}G\varphi\subset \Psaut^{\ell}(Y).\] If moreover, $\varphi$ is an isomorphism in codimension $\ell-1$, we say that $G$ is pseudo-regularizable in codimension $\ell$ by an isomorphism in codimension $\ell -1$.

\medskip
\subsubsection{Construction of $\mathcal{C}^\ell(X)$}\label{subsection:ccn}

Let $X$ be a variety over a field $\kk$ and $\ell$ be a non-negative integer. A $\ell$-marked variety is a pair $(Y ,\varphi)$, where $Y$ is a variety over $\kk$ and $\varphi\colon Y \dashrightarrow X$ is an isomorphism
in codimension $\ell$. Two $\ell$-marked varieties $(Y ,\varphi)$ and $(Y' ,\varphi')$ are equivalent if $\varphi'^{-1}\varphi : Y \dashrightarrow Y'$ \begin{center}
	\begin{tikzcd}[ampersand replacement=\&]
		Y\arrow[dashrightarrow]{dr}[below left]{\varphi} \arrow[dashrightarrow]{rr} \&\& Y'\arrow[dashrightarrow]{dl}[below right]{\varphi'}  \\
		\& X\&
	\end{tikzcd}
\end{center} is an isomorphism in codimension $\ell+1$. Such a class will be denoted by $[(Y ,\varphi)]$.

For instance, let $L= \{x=0\}\cap\{y=0\}$ be a line in $\PP^3$. The 1-marked varieties $(\PP^3,\id)$ and $(\PP^3\setminus L, \id_{\mid\PP^3\setminus L})$ are equivalent.

\begin{definition}The \emph{graph $\CC^{\ell}(X)$} associated to $X$ is the graph whose vertices are equivalent classes of $\ell$-marked varieties $[(Y,\varphi)]$, and where two vertices $v_1$ and $v_2$ are connected by an edge if $v_1$ can be represented by a couple $(Y ,\varphi)$ such that there exists an irreducible subvariety $D \subset Y $ of codimension $\ell+1$ and $v_2$ is represented by $(Y \setminus D , \varphi|_{Y \setminus D })$. 
\end{definition}

Note that we can put an orientation on the edges by saying that the edge from the above definition is oriented from $v_1$ to $v_2$. We denote this edge by $(Y,\varphi, D)$. Recall that we say that the vertex $v_2$ is dominated by $v_1$ if there exists a sequence of edges connecting $v_1$ to $v_2$ that are all oriented in the same direction from $v_1$ to $v_2$. 

\begin{example}
Consider two distinct points $p_1$ and $p_2$ in $\PP^3$. Denote respectively by $\pi_i:V_i \rightarrow \PP^3$ the blow-up of the point $p_i$, and $E_i$ the exceptional divisor obtained. In the graph $\CC^0(\PP^3)$, we have an edge from $[(V_i,\pi_i)]$ to $[(V_i\setminus E_i, {\pi_i}_{\mid V_i\setminus E_i})]$ which is the same vertex as $[(\PP^3\setminus \{p_i\}, \id_{\mid\PP^3\setminus \{p_i\}})]$ which is the same vertex as $[(\PP^3,\id)]$. Denote by $V_{1,2}$ the variety obtained by blowing up both $p_i$'s. Then from this variety, removing the $E_i$'s as well as the hypersurface at infinity $H:=\{x_3=0\}$ span a cube because up to pseudo-isomorphisms, these maps commute (see for instance Figure \ref{fig_ex_ccP3}). 
\end{example}

\begin{figure}
	\begin{center}
	\begin{tikzpicture}
		\draw (0,0) -- (2,0) -- (2,2) -- (0,2)-- (0,0);
		\draw (3,0.5) -- (3,2.5) -- (1,2.5);
		\draw[dotted]  (1,2.5)--(1,0.5) -- (3,0.5);
		\draw (2,0) -- (3,0.5); 
		\draw (2,2) -- (3,2.5);
		\draw (0,2) -- (1,2.5);
		\draw[dotted] (0,0) -- (1,0.5);
		\draw (0,0) node {$\bullet$} node[below] {$[(\PP^3,\id)]$};
		\draw (2,0) node {$\bullet$} node[below right] {$[(\A^3,\id_{\mid \A^3})]$} ;
		\draw(2,2) node {$\bullet$} node[below right] {$[(V_1\setminus H,{\pi_1}_{\mid V_1\setminus H})]$} ;
		\draw (0,2) node {$\bullet$};
		\draw (-0.5,2.2) node[left] {$[(V_1,\pi_1)]$};
		\draw (1,0.5) node {$\bullet$} node[above] {$[(V_2,\pi_2)]$};
		\draw (3,0.5) node {$\bullet$} node[right] {$[(V_2\setminus H,{\pi_2}_{\mid V_2\setminus H})]$};
		\draw(3,2.5) node {$\bullet$} node[right] {$[(V_{1,2}\setminus H,{\pi_2\pi_1}_{\mid V_{1,2}\setminus H})]$} ;
		\draw (1,2.5) node {$\bullet$} node[above left] {$[(V_{1,2},\pi_2\pi_1)]$};
	\end{tikzpicture}
\end{center}
\caption{A cube of dimension three in $\CC^0(\PP^3)$ spanned by the blow-ups of two points $p_1$ and $p_2$ in $\PP^3$ and by removing the hypersurface $H=\{x_3=0\}$ at infinity.\label{fig_ex_ccP3}}
\end{figure}
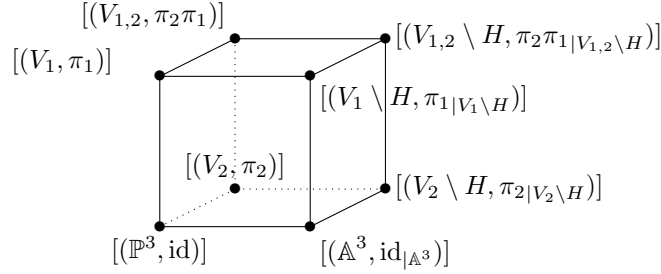

As illustrated in the example, removing a family of $n$ distinct subvarieties of codimension $\ell+1$ in a variety $Y$ generates a cube of dimension $n$. And indeed, it is the only way to have cubes. More precisely:

\begin{fact}\label{ccn_cubes}
	In the graph $\CC^{\ell}(X)$, $n$ distinct vertices $[(Y_1,\varphi_1)],\dots, [(Y_{2^n},\varphi_{2^n})]$ span a cube of dimension $n$ if and only if there exists $1\leq r\leq 2^n$ such that for any $1\leq j\leq 2^n$: 
	\begin{itemize}
			\item there exist $n$ distinct irreducible subvarieties $D_{1},\dots, D_{n}\subset Y_{r}$ of codimension $\ell+1$,
			\item for each $1\leq j\leq 2^n$ different of $r$, there exists $1\leq m\leq n$ such that $Y_{j}=Y_{r}\setminus \{D_{i_1}\cup\dots \cup D_{i_m}\}$ for some $1\leq i_1,\dots, i_m\leq n$,
			\item the $\varphi_j$ equals the restriction of $\varphi_r$ to $Y_{j}$.
		\end{itemize}
\end{fact}

The variety $X$ contains infinitely many irreducible subvarieties of codimension $\ell+1$, so removing an arbitrary large sequence of such subvarieties, we immediately see the following properties of these graphs.
\begin{proposition}
	The graph $\CC^{\ell}(X)$ is: 
	\begin{itemize}
		\item not locally compact,
		\item infinite dimensional with infinite cubes,
		\item cubically oriented.
	\end{itemize}
\end{proposition}

\begin{theorem}[\cite{Lonjou_Urech_cube_complexe}]\label{thm_cl_median}
	The graphs $\mathcal{C}^{\ell}(X)$ are median graphs for $0\leq \ell\leq \dim(X)$.
\end{theorem}

A key point to notice is the following: for any family of vertices, we can find explicitly a vertex that it is dominated by all of them. For instance, consider two vertices $v_i:=[(Y_{i},\varphi_i)]$ for $1\leq i\leq 2$ then, denoting by $U_1:=Y_1\setminus \Exc^{\ell}(\varphi_2^{-1}\varphi_1)$ and by $U_2:= Y_2\setminus \Exc^{\ell}(\varphi_1^{-1}\varphi_2)$, the vertex $[(U_1,{\varphi_1}_{\mid U_1})]=[U_2,{\varphi_2}_{\mid U_2}]$ is dominated by both $v_1$ and $v_2$.

\begin{lemma}\label{fact_intersection}
	Let $X $ be a variety and let $[(Y_{1},\varphi_1)],\dots, [(Y_{n},\varphi_n)]$ be a finite set of vertices in $\CC^{\ell}(X )$. Let $U_{i}\subset Y_{i}$ be open dense subsets with complements of codimension $>\ell$ such that for $2\leq i\leq n$, $\varphi_i^{-1}\varphi_1$ induces an isomorphism in codimension $\ell +1$ between $U_{1}$ and $U_{i}$. 
	Then the vertex $[(U_{1},\varphi_1|_{U_{1}})]$ is dominated by all the vertices $[(Y_{i},\varphi_i)]$. Indeed, $[(U_{i},\varphi_i|_{U_{i}})]$ is dominated by $[(Y_{i},\varphi_i)]$,
	and by construction, all the vertices $[(U_{i},\varphi_i|_{U_{i}})]$ coincide.
\end{lemma}

Note that this was not the case for the blow-up graph $\Cb(S)$ of a surface $S$. Indeed, any vertex $[(S',\varphi)]$ for which the surface $S'$ is minimal, i.e., does not contain any (-1)-curve, does not dominate any other vertices. 

It is still an open question whether the strong factorization theorem (Zariski theorem \ref{thm:Zariski}) remains valid or not in higher dimension. Zariski theorem was a key tool for the construction and the study of the blow-up graph in the surface case. Lemma \ref{fact_intersection} will be its replacement in the construction and the study of the graphs $\mathcal{C}^\ell(X)$.

The geometry of the graphs $\mathcal{C}^\ell(X)$ are really different from the one of the blow-up graph as we can see with the next lemma.

\begin{lemma}\label{lemma_v_dominates_cube}
	Let $v, v_1,\dots, v_n$ be vertices in  $\CC^{\ell}(X)$ such that $v$ dominates $v_1,\dots, v_n$. Then there exists a cube containing $v,v_1,\dots, v_n$ all of whose vertices are dominated by $v$.
\end{lemma}

\begin{proof}
	Let us observe that for any edge connecting a vertex $v$ with vertex $v'$, and for any representative $(Y,\varphi)$ of $v$ there exists an irreducible subvariety $D \subset Y $ of codimension $\ell+1$ and $v'$ is represented by $(Y \setminus D , \varphi|_{Y \setminus D })$. Because $v$ dominates the $v_i$ we can therefore find by induction a representative of the $v_i$ of the form $(Y \setminus D^i, \varphi|_{Y \setminus D^i })$ where $D^i$ is a finite union of irreducible subvarieties of codimension $\ell+1$ of $Y$ for each $i$. This shows that $v,v_1,\dots,v_n$ form a cube whose vertices are dominated by $v$.
\end{proof}

For instance, consider the standard quadratic Cremona involution $\sigma\in \Bir(\PP^2)$ seen in Example \ref{ex:standard}: $\sigma : [x_0:x_1:x_2] \dashmapsto [x_1x_2:x_0x_2:x_0x_1]$. For $0\leq i\leq 2$, denoting  by $L_i=\{x_i=0\}$ the coordinate lines, the vertex $[(\PP^2\setminus L_0\cup L_1\cup L_2,\id_{\mid\PP^2\setminus L_0\cup L_1\cup L_2})]$ of the graph $\CC^0(\PP^2)$ is dominated by both $[(\PP^2,\id)]$ and $[(\PP^2,\sigma)]$.

\begin{proof}[Idea of proof of Theorem \ref{thm_cl_median}]
	Using Theorem \ref{thm:MedianVsCC} and Theorem \ref{thm_ccCAT((0))}, it remains to prove that the cube completion of $\mathcal{C}^{\ell}(X)$ is simply connected and that links of vertices are flag simplical complexes. We will not discuss here about the flagness of the links of vertices but an interested reader can find it in \cite{Lonjou_Urech_cube_complexe}.
	
	By Lemma \ref{fact_intersection}, the graph $\mathcal{C}^{\ell}(X)$ is connected. 
		Let now $\gamma$ be a cycle in $\CC^{\ell}(X )$ passing through the vertices $v_1,\dots, v_n$, respectively represented by $[(Y_{i},\varphi_i)]$, which can be assumed to be without backtrack. Denote by $v=[(Y,\varphi)]$ the vertex from Lemma~\ref{fact_intersection} that is dominated by all the $v_i$, i.e., $\varphi_i^{-1}\varphi\colon Y\to Y_{i}$ is an {open immersion}. 
		The goal now is to show that $\gamma$ is homotopic to $v$. 
	For this we define $c_{\max}(\gamma):=\max_{1\leq i\leq n}(c(v_i)$, where $c(v_i)$ is the number of irreducible components of pure codimension $\ell+1$ of $Y_{i}\setminus \varphi_i^{-1}\varphi(Y)$. Note that $c(v_i)$ does not depend on the representative of $v_i$ and that $c(v_i)=0$ if and only if $v_i=v$. Denote by $1\leq i_0\leq n$ the maximal index such that $c(v_{i_0})=c_{\max}(\gamma)$.
	
	By definition, $v_{i_0}$ dominates $v_{i_0-1}$, $v_{i_0+1}$ and $v$ so by Lemma \ref{lemma_v_dominates_cube} there exists a cube containing $v_{i_0}$, $v_{i_0-1}$, $v_{i_0+1}$ and $v$. Moreover, this cube contains a square spanned by $v_{i_0}$, $v_{i_0-1}$ and $v_{i_0+1}$. We denote by $v_{i_0}'$ the fourth vertex. Since they form a square, we can deform $\gamma$ by a homotopy so that it passes through $v_{i_0}'$ instead of $v_{i_0}$. By the definition of our cube, $v_{i_0}'$ still dominates $v$ and $c(v_{i_0}')<c(v_{i_0})$. By induction on $(c_{\max}(\gamma),i_0)$ with the lexicographical order, we conclude that in finitely many such steps $\gamma$ is homotopic to $v$, and that the graph $\mathcal{C}^{\ell}(X)$  is simply connected.
\end{proof}

\begin{remark}\label{rmk_onecube}
In the case $\ell=0$, when the variety is normal, the median graph $\CC^0(X )$ is indeed median for a stupid reason; it is a single (infinite) cube. This comes from the fact that the convex hull of any finite set of vertices in $\CC^0(X )$ is a single cube. But this is not the case anymore if $\ell>0$ (see for instance \cite[Remark 4.8]{Lonjou_Urech_cube_complexe}).
\end{remark}

\subsubsection{Hyperplanes}
In the graph $\CC^{\ell}(X)$, an edge is given by removing an irreducible subvariety $D$ of codimension $\ell+1$ of a $\ell$-marked variety $(Y,\varphi)$. We denote such an edge by $(Y,\varphi, D)$ and by $[(Y,\varphi, D)]$ its corresponding hyperplane. In this subsection, we give a birational interpretation of hyperplanes, distance and halfspaces.

The following birational characterization of the equivalence of edges is an immediate consequence of the definition of Lemma \ref{fact_intersection}.

\begin{lemma}[\cite{Lonjou_Urech_cube_complexe}]\label{lemma_same_hyperplane}
	Two edges $(Y,\varphi, D)$ and $(Y',\varphi', D')$ correspond to the same hyperplane if and only if $D'$ is not contained in the exceptional locus of $\varphi^{-1}\varphi'$ and $D$ is the strict transform $\varphi^{-1}\varphi'(D')$ of $D'$ under $\varphi^{-1}\varphi'$.
\end{lemma}

The following lemma describes the geodesics of the graph $\CC^{\ell}(X)$. It is an immediate consequence of Lemma \ref{fact_intersection} that allows us to find a vertex dominated by any family of vertices, of Lemma \ref{lemma_same_hyperplane} that characterizes the hyperplanes of this graph, and of the characterization of geodesics by hyperplanes Theorem~\ref{theorem:combinatorial_geodesic}.
\begin{lemma}[\cite{Lonjou_Urech_cube_complexe}]\label{distancel}
	Let $v_1=[(Y_1,\varphi_1)]$ and $v_2=[(Y_2,\varphi_2)]$ be two vertices in $\CC^{\ell}(X)$.
	There is a bijection between the set of all the geodesic paths joining these two vertices and the set of all the possible order to remove the irreducible components of pure codimension $\ell+1$ of the exceptional locus of $\varphi_2^{-1}\varphi_1$ from $Y_1$ and adding the irreducible components of pure codimension $\ell+1$ of the exceptional locus of $\varphi_1^{-1}\varphi_2$. In particular, the distance between $v_1$ and $v_2$ equals the sum \[
	{\dist(v_1,v_2)=}\lvert\Exc^{\ell+1}(\varphi_1^{-1}\varphi_2)\rvert+\lvert\Exc^{\ell+1}(\varphi_2^{-1}\varphi_1)\rvert.
	\]
\end{lemma}
Note that working over algebraically closed fields ensure that if $f$ is an automorphism in codimension $\ell$ of a variety $Y$ then the number of irreducible components in the exceptional locus of codimension $\ell+1$ of $f$ and $f^{-1}$ are equals: $\lvert \Exc^{\ell+1}(f)\rvert =\lvert \Exc^{\ell+1}(f^{-1})\rvert$. 

\medskip

There is an algebraic characterization of halfspaces associated to a given hyperplane.

\begin{lemma}[\cite{Lonjou_Urech_cube_complexe}]\label{prop:halfspace_Cl}
	Consider a hyperplane $[(Y,\varphi, D)]$ in $\CC^\ell(X)$. The set of vertices $[(V,\psi)]$ such that $D$ is contained in the exceptional locus of $\psi^{-1}\varphi$ determines one of the two halfspaces.
\end{lemma}

\begin{remark}
	If $S$ is a projective surface, then the blow-up graph is a subgraph of $\CC^{0}(S)$. The vertices of the blow-up graph can be identified with the vertices in $\CC^0(S)$ of the form $[(S',\varphi)]$, where $S'$ is a projective surface. Indeed, the equivalence relation on the marked surfaces of the blow-up graph is a subrelation of the one put on the 0-marked varieties of the graph $\CC^0(S)$.
	However, this injection is not an isometric embedding. For instance, if we consider $j=(x,y)\mapsto (x,x^2+y)$ from Remark \ref{rmk_basepoints_infinitely_near}, in the blow-up graph $\Cb(\PP^2)$, the distance between the vertices $[(\PP^2,\id)]$ and $[(\PP^2,j)]$ is $6$ by Lemma \ref{lemma:comb_geodesic} as $j$ has three base-points (see Figure \ref{fig_resol_j}). But in the graph $\CC^{0}(\PP^2)$ it has distance $2$ because $\Exc(j)=\Exc(j^{-1})=\{z=0\}$, hence by Lemma \ref{distancel}, the path given by the vertices $[(\PP^2,\id)]$, $[(\A_z^2,\id_{\A_z^2})]$ and $[(\PP^2,j)]$ is geodesic.
\end{remark}

\subsubsection{Action of groups of automorphisms in codimension $\ell$}
For any $\ell$, the group of automorphisms in codimension $\ell$, $\Psaut^\ell(X)$, acts faithfully on the set of $\ell$-marked varieties by post-composition and preserves the equivalence class, hence it acts on the set of vertices of $\CC^\ell(X)$: let $f\in\Psaut^\ell(X)$ and $[(A,\varphi)]\in \CC^\ell(X)^0 $, $f\bullet [(Y,\varphi)] =[(Y, f\varphi)]$. This gives us a faithful action by isometries on the graph $\CC^\ell(X)$. Note that this action preserves the cubical orientation given. As a consequence, an element preserving a cube, fixes a vertex.

\medskip
As we will see in Proposition \ref{prop_catalog_ccn}, this action encodes geometrically, and in a unified way, diverse birational notions. 

\medskip

Let $f\in\Psaut^{\ell}(X)$, the \emph{dynamical number of the $(\ell+1)$-exceptional locus} of $f$ is defined as:
\[
\nu^{\ell+1}(f)\coloneqq \lim_{n\to\infty}\frac{\lvert \Exc^{\ell+1}(f^n)\rvert}{n}.
\] This limit always exists, since the function $f\mapsto \lvert \Exc^{\ell+1}(f)\rvert$ is subadditive on $\Psaut^{\ell}(X)$. The dynamical number of the $(\ell+1)$-exceptional locus is invariant under conjugacy in $\Psaut^{\ell}(X)$, since conjugation by a birational transformation $g$ changes the number of irreducible components of $\Exc^{\ell+1}(f)$ at most by a constant only depending on $g$ and $g^{-1}$.
This number has been introduced in \cite{Lonjou_Urech_cube_complexe} and is inspired by the dynamical number of base points from the surface case. Nevertheless, it is a weaker notion in the sense that if $f\in\Bir(S )$, where $S $ is a projective surface, then $\nu^{1}(f)\leq \mu(f)$. For instance, the Hénon transformation $h\colon \PP^2\dashrightarrow\PP^2$ defined with respect to affine coordinates by $h(x,y)=(y,y^2+x)$ satisfies $\nu^{1}(f)=0$, whereas $\mu(f)=3$. 

In the same way as the blow-up graph gives a geometrical interpretation of the dynamical number of base points, the graph $\CC^{\ell}(X )$ gives a geometrical interpretation of the dynamical number of the $(\ell+1)$-exceptional locus of elements of $\Psaut^{\ell}(X )$. The following proposition is an analogous of Proposition \ref{prop_catalog_blowup_cc} that was made in the context of the blow-up graph.

\begin{proposition}[\cite{Lonjou_Urech_cube_complexe}]\label{prop_catalog_ccn}We have the following correspondences between geometric and birational notions.
	\begin{enumerate}
		\item\label{item_ellipticn} Element of $\Psaut^\ell(X)$ inducing elliptic isometries of the graph $\CC^{\ell}(X )$ correspond to pseudo-regularizable elements in codimension $\ell+1$ of $\Psaut^\ell(X)$ by isomorphisms in codimension $\ell$.
		\item\label{item_distance_n} The distance between the vertex $[(X,\id)]$ and its image by an element $f$ of $\Psaut^{\ell}(X )$ corresponds to twice the number of irreducible components of the exceptional locus of $f$ of codimension $\ell+1$:
		\[	\dist\left([(X,\id)], [(X,f)]\right)=2 \lvert \Exc^{\ell+1}(f)\rvert. \]
		\item\label{item_translation_length_n} The translation length corresponds to twice the dynamical number of the ($\ell+1$)-exceptional locus of $f$:
		\[\ell(f)=2\nu^{\ell+1}(f).\]
	\end{enumerate}
\end{proposition}

\begin{proof}[Idea of proof of Proposition \ref{prop_catalog_ccn}]
	The point \ref{item_ellipticn} follows from the definition of the vertices of the graph $\CC^{\ell}(X )$, the point \ref{item_distance_n} has been seen above in Lemma \ref{distancel}.

	Let us focus on the point \ref{item_translation_length_n}. Consider $f\in\Psaut^{\ell}(X )$ and $x\in \Min(f)$. By Proposition \ref{prop_action_semisimple_hag}, for any $n\in\N$ we have $\dist(x,f^n(x))=n\ell(f)$. Let $v=[(X , \id)]$. Then $\dist(v, f^n(v))=2\lvert \Exc^{\ell+1}(f^n)\rvert$ by Lemma \ref{distancel}. This implies that for any $n\in \N$, $n\ell(f)\leq 2\lvert \Exc^{\ell+1}(f^n)\rvert$. Taking the limit we obtain: $\ell(f)\leq 2\nu^{\ell+1}(f)$.

	On the other hand, let $x\in \Min(f)$ and $K=\dist(v, x)$. Then for any $n\in \N$, $2\lvert \Exc^{\ell+1}(f^n)\rvert=\dist(v, f^n(v))\leq n\ell(f)+2K$ and hence $2\nu^{\ell+1}(f)\leq \ell(f)$.
\end{proof}

Point \ref{item_translation_length_n} of Proposition \ref{prop_catalog_ccn} implies the following theorem, which can be seen as an analogue to Theorem~\ref{Blanc_Deserti}:

\begin{theorem}[\cite{Lonjou_Urech_cube_complexe}]\label{prop_nu_0}
	Let $f\in \Psaut^{\ell}(X )$. Then 
	\begin{enumerate}
		\item $\nu^{\ell+1}(f)$ is an integer; in particular the sequence $\left(\Exc^{\ell+1}(f^n)\right)_{n\in\N}$ is either bounded or grows asymptotically linearly in $n$;
		\item there exists a variety $Y $ and an isomorphism in codimension $\ell$ $\varphi\colon Y \dashrightarrow X $ such that $\varphi^{-1} f\varphi$ has exactly $\nu^{\ell+1}(f)$ irreducible components in its exceptional locus of codimension $\ell+1$;
		\item in particular, $\nu^{\ell+1}(f)=0$ if and only if $f$ is pseudo-regularizable in codimension $\ell+1$ by an isomorphism in codimension $\ell$.
	\end{enumerate}
\end{theorem}

\medskip

\subsubsection{(Pseudo)-regularization results} 
By Proposition \ref{prop_catalog_ccn}\ref{item_ellipticn}, stabilizers of vertices are pseudo-regularizable subgroups in codimension $\ell +1$ by an isomorphism in codimension $\ell$. Hence, this action allows to obtain in an unified way some regularizable and pseudo-regularizable results.

For instance, using Proposition \ref{prop_catalog_ccn} \ref{item_distance_n}, bounded actions can be described as follows:
\begin{proposition}[\cite{Lonjou_Urech_cube_complexe}]\label{boundedcomp}
Let $G\subset\Psaut^{\ell}(X )$ for some $0\leq \ell\leq \dim(X )-1$. The subgroup $G$ is pseudo-regularizable in codimension $\ell+1$ by an isomorphism in codimension $\ell$ if and only if $\{\lvert\Exc^{\ell+1}(g)\rvert\mid g\in G\}$ is uniformly bounded.
\end{proposition}

The construction of all these median graphs (for $\ell>0$) and not only $\mathcal{C}^0(X)$ have been indeed the key point to show regularizable results instead of only pseudo-regularizable results. The following theorem has been one of the motivations of these constructions. 

\begin{theorem}[\cite{Lonjou_Urech_cube_complexe}]\label{thm_regularization}
	Let $X$ be a complete variety.
	\begin{itemize}
		\item If a subgroup $G\subset\Bir(X)$ has property FW, then $G$ is regularizable.
		\item If an element $g\in\Bir(X)$ is divisible or distorted, then $\langle g\rangle$ is regularizable.
	\end{itemize} 
\end{theorem}

\begin{proof}
	Denote by $d$ the dimension of $X$. 
	Let $G\subset\Bir(X)$ a subgroup that either has the FW property, either that is generated by a single element $g\in \Bir(X)$ that is divisible or distorted. Consider the action of $\Bir(X)$ on $\CC^{0}(X)$.
	If $G$ does not have the FW property, defines the length function $L(f):=\dist((X,\id),(X,f))$ on $\Bir(X)$. As $G$ is generated by a distorted element or a divisible one, the restriction of this length function to $G$ has to be bounded.

	As a consequence, in any case, the orbits of $G$ on  the median graph $\CC^{0}(X)$ are bounded, and because the cubical orientation is preserved, $G$ fixes a vertex $[(X_{1},\varphi_1)]$. As a consequence, $\varphi_1$ conjugates $G$ to a subgroup of $\Psaut^{1}(X_{1})$. This yields an action of $G$ on $\CC^{1}(X_{1})$. We continue inductively until we find a variety $X_{d}$ and a birational map $\varphi=\varphi_1\dots\varphi_{d}\colon X_{d}\dashrightarrow X$ that conjugates $G$ to a subgroup of $\Psaut^{d}(X_{d})=\Aut(X_{d})$. Hence, $G$ is regularizable.
\end{proof}

\subsubsection{Note on related constructions \cite{Cantat_Cornulier_Commensurating} and \cite{Cornulier_FW}}\label{subsubsection_Cornulier_Cantat}
The first part of Theorem \ref{thm_regularization} answers a question of Cantat and Cornulier (\cite[Question 10.1]{Cantat_Cornulier_Commensurating}) where they proved that subgroups of $\Bir(X)$ with the FW property are pseudo-regularizable using the notion of commensurating actions. As mentioned by the authors in \cite{Lonjou_Urech_cube_complexe}, the constructions of the cube complexes $\CC^\ell(X)$ are inspired by 
\cite{Cantat_Cornulier_Commensurating}. 

In order to be more precise, let us introduce some definitions needed. 
Consider a group $G$ acting on a set $X$. We say that a  subset $A$ of $X$ \emph{is commensurated} by $G$ if for all $g\in G$ the symmetric difference $gA\bigtriangleup A$ is finite. If there exists a subset $B$ of $A$ that is $G$-invariant and such that the symmetric difference $A\bigtriangleup B$ is finite, then $A$ is said to be \emph{transfixed} by $G$. 
Note that a transfixed subset is also commensurated. But the converse does not hold in general. Nevertheless an equivalent definition of the property FW is the following. 
A group $G$ has Property FW if, given any action of $G$ on a set $X$, all commensurated subsets of $X$ are transfixed.

In \cite{Cantat_Cornulier_Commensurating}, $X$ is assumed to be a smooth projective variety over an algebraically closed field of characteristic $0$. Inspired by the construction of the space $\Hh^\infty$ they construct a space denoted by $\Hypt(X)$ consisting in all irreducible and reduced hypersurfaces in all birational smooth projective varieties dominating $X$ (up to a natural identification). They show that $\Bir(X)$ acts on this set by strict transform, and that the set $\Hyp(X)$ of all irreducible and reduced hypersurfaces of $X$ is commensurated by $\Bir(X)$.
They prove the following theorem:
\begin{theorem}[\cite{Cantat_Cornulier_Commensurating}]
Let $X$ be a smooth projective variety over a field of characteristic $0$. Let $\Gamma$ be a subgroup of $\Bir(X)$. Then $\Gamma$ transfixes the subset $\Hyp(X)$ of $\Hypt(X)$ if and only if $\Gamma$ is pseudo-regularizable. 
\end{theorem} 

The strategy is as follows. Assume that $\Gamma$ transfixes $\Hyp(X)$. Then there exists a subset $B\subset\Hyp(X)$ $\Gamma$ invariant such that there exist two finite sets of irreducible and reduced hypersurfaces $H_1,\dots, H_n$ of $X$ and $D_1,\dots,D_\ell\in \Hypt(X)$ lying in some varieties dominating $X$ such that $B=\left(\Hyp(X) \setminus\left( H_1\cup\dots\cup H_n\right)\right)\cup \left(D_1\cup \dots\cup D_\ell\right)$. There exists a smooth projective variety $Y$ dominating $X$ such that $D_1,\dots,D_\ell\in \Hyp(Y)$. Denote by $\pi:Y\rightarrow X$, then up to replacing $\Gamma$ by  $\pi^{-1}\Gamma\pi$ and $X$ by $Y$ we can assume that $B=\Hyp(X) \setminus\left( H_1\cup\dots\cup H_n\right)$. Then we show that $\Gamma$ is a subgroup of the group of pseudo-automorphisms in codimension $1$ of the variety $Y\setminus H_1\cup\dots\cup H_n$, which is not projective anymore.
Note that the other way around is direct. 

As a consequence, they obtain the following corollary.
\begin{corollary}[\cite{Cantat_Cornulier_Commensurating}]
Let $X$ be a smooth projective variety over a field of characteristic $0$. Let $\Gamma$ be a subgroup of $\Bir(X)$ having the FW property. Then $\Gamma$ is pseudo-regularizable. 
\end{corollary}

Note that from a commensurating set, it is possible to construct a wall space on a group (\cite{Haglund_Paulin}, see also \cite{Cornulier_wallings}), and from a wall-space a CAT(0) cube complex (see \cite{Chatterji_Niblo}).
It seems that following this correspondence we would obtain $\CC^0(X)$ constructed directly in \cite{Lonjou_Urech_cube_complexe}.

\medskip

In \cite{Cornulier_FW}, Cornulier has also obtained another proof that groups with property FW are regularizable (first point of Theorem \ref{thm_regularization}), using the notion of partial actions. Unlike \cite{Lonjou_Urech_cube_complexe}, varieties in \cite{Cornulier_FW} are not assumed to be reduced.
The main tool is the notion of partial action.
\begin{definition}
Let $G$ be a group and $X$ be a set. 
A \emph{partial action of $G$ on $X$} is a map $\alpha$ from $G$ to the set of partial bijections on $X$ satisfying:
\begin{itemize}
	\item $\alpha(g):D_g\subset X \rightarrow D_g'\subset X$,
\item $\alpha(1_G)=\id_X$,
\item $\alpha(g^{-1})=\alpha(g)^{-1}$,
\item $\alpha(gh)\supset \alpha(g)\alpha(h)$.
\end{itemize}
\end{definition}

Note that the last point means that if $\alpha(h)$ is defined in $x$ and $\alpha(g)$ is defined in $\alpha(h)(x)$ then $\alpha(gh)$ is defined in $x$ and equals $\alpha(g)(\alpha(h)(x))$.

Let us give broadly the strategy. 
First look at the group of birational transformations of a variety $X$ as partially acting on $X$ as follows:
Let $f\in \Bir(X)$ then $\alpha(f): X\setminus \Exc(f)\rightarrow X\setminus \Exc(f^{-1})$.
By a result of Abadie, Kellendonk-Lawson, associated to a partial action there exists a \emph{universal globalization} $\hat{X}$, i.e, a set (unique in some sense) such that $G$ acts on it, and when restricted to $X$ we obtain the partial action given. It is
defined as follows: consider $G\times X$ with the $G$-action $g\bullet (h,x) =(gh,x)$ and identify $(g,x)$ and $(h,y)$ if there exists $k\in G$ such that $\alpha(kg)(x )$ and $\alpha(kh)(y )$ are
defined and equal.

It is almost a regularization result; the problem being that $\hat{X}$ is in general not a variety. Indeed, it is obtained by glueing copies of $X$ along open dense subsets. In \cite{Lonjou_Urech_cube_complexe} they encountered the same problem to show that the graphs $\CC^\ell(X)$ are median and more precisely to show that links of vertices are flag. Indeed, consider a vertex $v$, and a family $\{v_1,\dots, v_k\}$ of adjacent vertices at $v$ that pairwise span a square. Then we need to prove that there exists a cube containing them. To do so, we need to find a vertex dominating them (see Lemma \ref{lemma_v_dominates_cube}, so the difficulty comes from the ascending edges. To construct this vertex, we also need to glue together dense open subsets of a given variety and show that what we obtain is still a variety; in particular that it is a $\kk$-scheme that is separated.

Using the definition of the property FW in terms of commensurating sets and a consequence of Neumann's lemma (\cite{Neumann}) saying that if $\tilde{X}\setminus X$ is finite and meets no finite orbit then there exists $g\in G$ such that $\hat{X}=X\cup gX$, Cornulier proves the following theorem.
\begin{theorem}[\cite{Cornulier_FW}]
Let $G$ be a subgroup of $\Bir(X)$ having the property FW then there exists a dense $G$-invariant open subset of $\tilde{X}$ (of the partial action of $G$ on $X$) that is a variety.
\end{theorem} 

%
%
%

\medskip

	\subsubsection{Other consequence of the action of $\Bir(X) $ on $\CC^0(X)$}
	We focus on two other kind of results that are possible to obtain for Cremona transformations that are not pseudo-regularizable in codimension $1$, by using the action of the Cremona group on $\CC^{0}(\PP^n)$ from \cite{Lonjou_Urech_cube_complexe}. The first one gives new restrictions on centralizers of these Cremona transformations and the second one bounds the asymptotic degree growth by below of these Cremona transformations.

	\medskip
	Centralizers of elements in the Cremona group of rank $2$ have been studied in various papers (\cite{Cantat_groupes_birat}, \cite{Blanc_Cantat_Dynamical_degrees}, \cite{Blanc_Deserti}, \cite{Cerveau_Deserti_centralisateurs}, \cite{Zhao_Centralizers}): centralizers of general elements are virtually cyclic (see \cite{Zhao_Centralizers} and references therein). The rigid nature of isometries of median graphs allows to give new restrictions on centralizers of non-pseudo regularizable transformations of varieties of arbitrary dimension:
	
	\begin{theorem}[\cite{Lonjou_Urech_cube_complexe}]\label{centralizer}
		Let $X$ be a variety over an algebraically closed field $\kk$ of characteristic $0$. Let $f\in\Bir(X)$ be an element that is not pseudo-regularizable in codimension $1$ and let $cent(f)\subset\Bir(X)$ be its centralizer. Then either $f$ permutes the fibers of a rational map $X\dashrightarrow Y$, where $0<\dim(Y)<\dim(X)$, or $cent(f)$ contains as a finite index subgroup $\langle f\rangle\times H$, where $H\subset\Bir(X)$ is a torsion subgroup. 
	\end{theorem}
	
	Note that the above result is not new for the case of surfaces, and a more precise result is known. Indeed, in \cite{Blanc_Cantat_Dynamical_degrees}, it is shown that any $f\in \Bir(\PP^2)$ that induces a loxodromic isometry on $\Hh^\infty$, has 
	$\langle f\rangle$ as finite index subgroup of its centralizer $cent(f)$.
	Note that they are the only ones that are not pseudo-regularizable and that do not permute the fibers of a rational map. Hence, one could ask whether in any dimension, the torsion group $H$ in Theorem~\ref{centralizer} is always finite? Another interesting question is to understand if there exist non pseudo-regularizable elements having such centralizers.

	\begin{theorem}[\cite{Lonjou_Urech_cube_complexe}]\label{thm_degree_growth}
		Let $g\in \Bir(\PP^d)$ be an element that is not pseudo-regularizable in codimension $1$. Then the asymptotic growth of $\deg(g^n)$ is at least $\frac{1}{d+1}n$.
	\end{theorem}

\begin{proof}[Sketch of the proof]
For any $n$, we denote by $\delta_n$ the degree of $g^n$, and by $(g_{n,0},g_{n,1},\dots,g_{n,d})$ the polynomials of degree $\delta_n$ that define $g^n$. Then the exceptional locus of $g^n$ is the zero locus of the Jacobian of $(g_{n,0},g_{n,1},\dots,g_{n,d})$, which is of degree at most $(\delta_n-1)(d+1)$. Hence $\lvert\Exc^1(g^n)\rvert \leq \delta_n(d+1)$. Moreover, as $g$ is not pseudo-regularizable, by Theorem \ref{prop_nu_0}, the sequence $\left(\Exc^{1}(g^n)\right)_{n\in\N}$ grows linearly, which gives the excepted lower bound for $\delta_n$.
\end{proof}
	
\subsection{Application: construction of morphisms from some Cremona groups to $\Z$}
\label{Subsection_morphisms_Z}
In this subsection, we focus on an application of the action of Cremona groups on median graph in order to obtain morphisms to $\Z$. We first, introduce a bit of context.

\medskip

\subsubsection{Introduction}
As seen in Subsection \ref{Subsection_State_art}, the first proofs that the Cremona group of rank $2$ (\cite{Cantat-Lamy}, \cite{Lonjou_non_simplicity}) is not a simple group consist in exhibiting elements that generates proper normal subgroups of the Cremona group of rank $2$. 

Another strategy, not using geometric group theoretic methods and working over non algebraically closed field, emerged. It consisted in constructing non-trivial group homomorphisms from the Cremona group of rank $2$ to direct sums or free products of copies of $\mathbb{Z}/2\mathbb{Z}$ (for instance \cite{Zimmermann_abelianisation_real} over real numbers, \cite{Lamy_Zimmermann_non_closed_fields} over some fields) using Sarkisov links (special birational maps). This strategy has been fruitful as it has led to another recent breakthrough done by Blanc, Lamy and Zimmermann \cite{Blanc_Lamy_Zimmermann_quotients}: constructing non-trivial group homomorphisms from Cremona groups of rank at least $3$ over $\mathbb{C}$ to direct sums of copies of $\mathbb{Z}/2\mathbb{Z}$.

 For the first time and using motivic methods Lin and Shinder constructed non-trivial homomorphisms from some Cremona groups of rank at least three onto $\mathbb{Z}$ \cite{Lin_Shinder}. In \cite{Blanc_Schneider_Yasinsky} the result of Lin and Shinder is reproved using the techniques developed in \cite{Blanc_Lamy_Zimmermann_quotients}, where they show the much stronger result that $\Bir(\PP^n)$ surjects to a free product of infinitely many copies of $\Z$, if $n\geq 4$ and $\kk\subset \C$. Note that when the field is perfect, such morphisms can not exist in rank $2$ as it has been proved by \cite{Lamy_Schneider_involutions} that the Cremona group of rank $2$ over a perfect field is generated by involutions.
 
 \medskip 
 
An alternative proof of the result of \cite{Lin_Shinder} has been done in \cite{GLU_morphism_Z} using geometric group theoretic methods through the action of Cremona groups on the median graphs $\CC^0(\PP^n)$ defined in Subsection \ref{Subsection_median_graphs_groups_isom_codim_l}. We follow this proof.

 We need first to introduce the following definition.
  Two subvarieties $A,B\subset X$ are \emph{Cremona equivalent} if there exists a birational transformation $f$ of $X$ such that $f(A)=B$, where $f(A)$ denotes the strict transform of $A$. Denote by $\Z[\Div(X)/_\approx]$ the free abelian group on the set of Cremona equivalence classes of subvarieties of codimension $1$ of $X$. 
 
 \begin{theorem}[\cite{GLU_morphism_Z}]\label{thm_morphism_Z}
 	There exists a non-trivial homomorphism \[\varphi\colon\Bir(\PP^n)\to\Z[\Div(\PP^n)/_\approx] \]in the following cases:
 	\begin{enumerate}
 		\item $n\geq 5$ and the field is infinite,
 		\item $n\geq 4$ and the field is of characteristic $0$,
 		\item $n=3$ and the field is either a number field, a function field over a number field, or a function field over an algebraically closed field.
 	\end{enumerate} 
 	Moreover, in these cases, $\Bir(\PP^n)$ is not generated by pseudo-regularisable elements. 
 \end{theorem}

Note that it is indeed, a reinforcement of \cite{Lin_Shinder} 
who proved the existence of a non-trivial homomorphism (in the same cases) \[c\colon\Bir(\PP^n)\to\Z[\Bir_{n-1}], \] where $\Z[\Bir_{n-1}]$ is the free abelian group on the set of birational equivalence classes of varieties of codimension $1$.
Indeed being birationally equivalent is weaker than being Cremona equivalent. Remark also that to this day it is unknown if there exist morphisms from the Cremona group of rank $3$ over $\C$ to $\Z$.
 
 \medskip

 In both cases, the homomorphism is constructed as follows. Let $f\in\Bir(\PP^n)$. Let $H_1,\dots, H_k$ be the irreducible components of codimension 1 of the exceptional locus of $f$ and let $K_1,\dots, K_m$ be the irreducible components of codimension 1 of the exceptional locus of $f^{-1}$. The group homomorphism $\varphi\colon\Bir(X)\to \Z[\Div(X)/_\approx]$ is given by 
 \begin{equation}\label{morphism}
f\mapsto [K_1]+\dots +[K_m] -[H_1]-\dots-[H_k].
 \end{equation}
 In order to show Theorem \ref{thm_morphism_Z}, two ingredients are needed. First, prove that $\varphi$ is a group homomorphism. In \cite{Lin_Shinder} it is done through motivic methods, while in \cite{GLU_morphism_Z} it is done through geometric group theoretic methods.  Second, show that it is non-trivial. 
 Note that in order to be non trivial it is enough to show that there exists an $f$ such that $H_i$ that is not Cremona equivalent to any of the $K_j$. Such examples can be found in \cite{Hassett_Lai} for the case $\PP^n$ if $n\geq 4$, or in \cite{Lin_Shinder}, where they extended the example of Hasset and Lai and other examples.

 \medskip 
 
 The construction of this morphism comes from a more general construction of morphisms from groups acting on median graphs and preserving a cubical orientation, to right-angled Artin groups.
 
 \medskip
 
 \subsubsection{Morphisms from groups acting without inversion on median graphs to right-angled Artin groups}
 \label{subsect_morphsim_artin}
 
 Given a graph $\Gamma$, the \emph{right-angled Artin group} $A(\Gamma)$ is defined by the presentation
 \[\langle u \in V(\Gamma) \mid  uvu^{-1}v^{-1}=1 \ (\{u,v\} \in E(\Gamma)) \rangle\]
 where $V(\Gamma)$ and $E(\Gamma)$ denote the vertex- and edge-sets of $\Gamma$.  Note that if $\Gamma$ is a complete graph then $A(\Gamma)$ is a free abelian group. 
 
 \medskip
 
 The goal of this subsection is to explain how, given an arbitrary group $G$ acting on a median graph $\mathcal{G}$ and preserving a cubical orientation, we can construct a morphism from $G$ to some right-angled Artin group. For free actions, this is done in \cite{Haglund_Wise}. We follow \cite{GLU_morphism_Z} where the point of view of \cite{Genevois_Special_CC} is used.
 
 \medskip \noindent
 Let us fix a cubical orientation on $\mathcal{G}$. Let $\Gamma$ denote the graph whose vertices are the $G$-orbits of hyperplanes in $\mathcal{G}$ and whose edges connect two orbits whenever they contain transverse hyperplanes. 
 
 \begin{example}\label{ex_morphism_artin}
 For instance, the action of $\Z^2$ on its classical Cayley graph (see Figure \ref{fig:Cay_Z2}) gives a graph made of two vertices and an edge linking them. The action of the standard quadratic Cremona involution $\sigma : (x,y) \mapsto (\frac{1}{x},\frac{1}{y})$ on the convex hull of the vertices $[(\PP^2,\id)]$ and $[(\PP^2,\sigma)]$ in the blow-up graph (see Figure \ref{ex:standard}) gives a complete graph with three vertices. 
 \end{example}

 \medskip
 
 Notice that an oriented path $\alpha$ in $\mathcal{G}$ is naturally labelled by the word written over $V(\Gamma) \sqcup V(\Gamma)^{-1}$ given by the oriented edges crossed by $\alpha$. Fixing a vertex $o \in \mathcal{G}$,
 \[\Theta : \left\{ \begin{array}{ccc} G & \to & A(\Gamma) \\ g & \mapsto & \text{label of a path from $o$ to $g \cdot o$} \end{array} \right.\]
 defines a group homomorphism.

 If we come back to the examples above (Example \ref{ex_morphism_artin}), in the case of $\Z^2$ we are constructing the morphism identity, while in the case of the subgroup generated by the standard quadratic Cremona involution, we obtain a trivial morphism.
 
 \medskip 
 The fact that $\Theta$ is well-defined is a consequence of the following lemma, which is a consequence of the fact that filling the $4$-cycles of a median graph with squares yields a simply connected square complex.
 
 \begin{lemma}\label{lem:PathMedian}
 	Let $\mathcal{G}$ be a median graph and let $\alpha,\beta$ be two paths with the same endpoints. Then $\alpha$ can be transformed into $\beta$ by adding or removing backtracks and by flipping $4$-cycles.
 \end{lemma}
 
  Indeed, adding or removing a backtrack to a path amounts to adding or removing a subword $uu^{-1}$ or $u^{-1}u$ (where $u \in V(\Gamma)$) to its label. And flipping a $4$-cycle amounts to replacing a subword $uv$ (resp. $u^{-1}v$, $uv^{-1}$, $u^{-1}v^{-1}$) with $vu$ (resp. $vu^{-1}$, $v^{-1}u$, $v^{-1}u^{-1}$) where $\{u,v\} \in E(\Gamma)$. Thus, $\Theta$ is well-defined.
 
 \medskip 
 Let us show now that $\Theta$ is a group homomorphism. Let $g,h \in G$ be two elements and fix two oriented paths $[o,go]$ and $[o,ho]$. Then
 \begin{align*}
\Theta(gh) & = \text{label of } [o,go] \cup g [o,ho] = (\text{label of } [o,go]) \cdot (\text{label of } g[o,ho])  \\ & = (\text{label of } [o,go]) \cdot (\text{label of } [o,ho]) = \Theta(g) \Theta(h).
 \end{align*}

 Note that the group homomorphism $\Theta$ depend on the choice of the basepoint $o$. But changing the basepoint amounts to change the morphism by a conjugation. Hence,  if the right-angled Artin group $A(\Gamma)$ turns out to be abelian, the morphism $\Theta$ does not depend on the choice of the basepoint $o$. As another consequence, note also that all $g\in G$ fixing a vertex are contained in the kernel of $\Theta$.

 \medskip
 
 \subsubsection{Application to Cremona groups}
 In order to prove Theorem \ref{thm_morphism_Z}, it remains to construct the morphism of Equation \ref{morphism}.
 
 \begin{proof}
Consider the action of $\Bir(\PP^n)$ on the oriented median graph $\CC^0(\PP^n)$ constructed in Section \ref{Section_median_graphs_higher_ranks}. By Remark \ref{rmk_onecube}, all the hyperplanes are pairwise transverse, so the graph $\Gamma$ constructed in Subsection~\ref{subsect_morphsim_artin} is complete and the group $A(\Gamma)$ is therefore free abelian.

Fix the vertex $o=([\PP^n,\id])$ in $\mathcal{C}^0(\PP^n)$. By Lemma \ref{fact_intersection}, we obtain a path from $o$ to $f(o)$ defined by the vertices 
\[[\PP^n,\id], [\PP^n\setminus K_1,\id],\dots,[\PP^n\setminus \{K_1\cup\dots\cup K_m\},\id]=[\PP^n\setminus \{H_1\cup\dots\cup H_k\},f],\dots, [\PP^n,f].\]
This shows that the image $\Theta(\Bir(\PP^n))$ is in fact contained in $A(\Gamma')$, where $\Gamma'\subset \Gamma$ is the complete subgraph whose vertices are represented by hyperplanes of the form $[(\PP^n,\psi, H)]$, with $\psi\in \Bir(\PP^n)$. Moreover, by Fact \ref{lemma_same_hyperplane} two hyperplanes $[(\PP^n,\psi, H)]$ and $[(\PP^n,\psi', H')]$ represent the same vertex if and only if $H$ and $H'$ are Cremona equivalent, hence we can identify the vertices of $\Gamma'$ with the Cremona equivalence classes of hypersurfaces in $\PP^n$ and the group $A(\Gamma')$ with $\Z[\Div(\PP^n)/_\approx]$. The path $[o, f(o)]$ crosses the hypersurface classes $[K_j]$'s with positive sign and the hypersurface classes $[H_i]$'s with negative sign. This achieves the proof.

Note that every pseudo-regularisable element fixes a vertex in $\mathcal{C}^0(\PP^n)$. Hence, all pseudo-regularisable elements are in the kernel of $\Theta$. Therefore, they cannot generate $\Bir(\PP^n)$ in this case.
 \end{proof}

\bibliographystyle{amsalpha}
\bibliography{bibliography_lonjou}

\end{document}